\documentclass[11pt,a4paper,reqno]{amsart}

\usepackage{amsmath,amssymb,amsthm,graphicx,amsfonts}
\usepackage{mathrsfs}
\usepackage{rotating}
\usepackage{amsxtra}
\usepackage{float}
\usepackage[colorlinks,citecolor=red,linkcolor=blue]{hyperref}
\usepackage[usenames,dvipsnames]{xcolor}
\usepackage[all]{xy}
\usepackage{geometry,array}
\usepackage{url}
\usepackage{tikz}\usetikzlibrary{matrix}
\usetikzlibrary{matrix,positioning,decorations.markings,arrows,decorations.pathmorphing,backgrounds,fit,positioning,shapes.symbols,chains,shadings,fadings,calc}
\tikzset{->-/.style={decoration={  markings,  mark=at position #1 with
    {\arrow{>}}},postaction={decorate}}}
\tikzset{-<-/.style={decoration={  markings,  mark=at position #1 with
    {\arrow{<}}},postaction={decorate}}}
\usepackage[markup=underlined]{ch
anges}
\definechangesauthor[name=zy, color=cyan]{zy}
\definechangesauthor[name=hp, color=purple]{hp}

\numberwithin{equation}{section}

\theoremstyle{plain}
\newtheorem{theorem}{Theorem}[section]

\newtheorem{lemma}[theorem]{Lemma}
\newtheorem{corollary}[theorem]{Corollary}
\newtheorem{proposition}[theorem]{Proposition}
\theoremstyle{definition}
\newtheorem{definition}[theorem]{Definition}
\newtheorem{construction}[theorem]{Construction}

\newtheorem{remark}[theorem]{Remark}

\numberwithin{equation}{section}

\def\surf{\mathbf{S}}                       
\def\surfi{\surf^\circ}   
\def\T{\mathbf{T}}
\def\PP{\mathbf{P}}
\def\R{\mathbf{R}}
\def\D{\mathcal{D}}
\def\MM{\mathbf{M}}

\def\TA(\surf){\mathbf{A}^{\times}(\surf)}

\def\<{\langle}
\def\>{\rangle}
\def\Y{{\D_G}}

\renewcommand{\k}{\mathbf{k}}

\newcommand{\U}{\mathbf{U}}
\newcommand{\V}{\mathbf{V}}

\newcommand{\N}{\mathbf{N}}

\newcommand{\C}{\mathcal{C}}    

\renewcommand{\mod}{\operatorname{mod}}
\newcommand{\ind}{\operatorname{ind}} 

\newcommand{\Int}{\operatorname{Int}}
\newcommand{\Hom}{\operatorname{Hom}}

\newcommand{\Ext}{\operatorname{Ext}}

\newcommand{\End}{\operatorname{End}}   

\newcommand{\Fac}{\operatorname{Fac}}

\newcommand{\rank}{\operatorname{rank}}

\newcommand{\st}{\mathrm{s}\tau\operatorname{-tilt}}

\newcommand{\trp}{\tau\operatorname{-rigidp}}

\renewcommand{\r}{\operatorname{rigid-}}
\newcommand{\tr}{\tau\operatorname{-rigid}}
\newcommand{\mr}{\operatorname{max}\operatorname{rigid-}}
\newcommand{\ymr}{\operatorname{max}\operatorname{rigid}_G\operatorname{-}}
\newcommand{\xmr}[1]{\operatorname{max}\operatorname{rigid}_{#1}\operatorname{-}}

\newcommand{\eg}{\operatorname{EG}^\times}
\newcommand{\add}{\operatorname{add}}

\def\PTT{\operatorname{R}^{\times}}
\def\TA{\operatorname{A}^\times}

\newcommand{\stTA}[1]{\operatorname{A}^\times_{#1\operatorname{-st}}}
\newcommand{\cstTA}[1]{\operatorname{A}^\times_{#1\operatorname{-co-st}}}
\def\jiaodian{\mathfrak{q}}

\begin{document}

\title{Mutation graph of support $\tau$-tilting modules over a skew-gentle algebra}
	
\date{\today}

\author{Ping He}
\address{Yanqi Lake Beijing Institute of Mathematical Sciences and Applications, 
	101408,
	Beijing, China}
\email{pinghe@bimsa.cn}

\author{Yu Zhou}
\address{Yau Mathematical Sciences Center,
	Tsinghua University, 100084,
	Beijing, China}
\email{yuzhoumath@gmail.com}

\author{Bin Zhu}
\address{Department of Mathematical Sciences,
	Tsinghua University, 100084,
	Beijing, China}
\email{zhu-b@mail.tsinghua.edu.cn}
	
\thanks{The work was supported by National Natural Science Foundation of China (Grants No. 11801297, 11671221).}
	
\keywords{punctured marked surfaces; partial tagged triangulations; standard arcs; support $\tau$-tilting modules; mutation exchange graphs}

\begin{abstract}
	Let $\D$ be a Hom-finite, Krull-Schmidt, 2-Calabi-Yau triangulated category with a rigid object $R$. Let $\Lambda=\End_{\D}R$ be the endomorphism algebra of $R$. We introduce the notion of mutation of maximal rigid objects in the two-term subcategory $R\ast R[1]$ via exchange triangles, which is shown to be compatible with mutation of support $\tau$-tilting $\Lambda$-modules. In the case that $\D$ is the cluster category arising from a punctured marked surface, it is shown that the graph of mutations of support $\tau$-tilting $\Lambda$-modules is isomorphic to the graph of flips of certain collections of tagged arcs on the surface, which is moreover proved to be connected. As a direct consequence, the mutation graph of support $\tau$-tilting modules over a skew-gentle algebra is connected.
\end{abstract}
	
\maketitle
\tableofcontents
	
\section*{Introduction}

Cluster algebras were introduced by Fomin and Zelevinsky \cite{FZ} around 2000. The geometric aspect of cluster theory was explored and developed by Fomin, Shapiro and Thurston \cite{FST} and Labardini-Fragoso \cite{LF1}, where they construct a quiver with potential \cite{DWZ} from any triangulation of a marked surface. On the other hand, cluster categories of acyclic quivers were introduced by Buan, Marsh, Reineke, Reiten and Todorov \cite{BMRRT} in order to categorify cluster algebras, which were generalized later by Amiot \cite{A1} to cluster categories of quivers with potential. 

The indecomposable objects in the cluster category from a marked surface without punctures are classified via curves by Br\"{u}stle and Zhang \cite{BZ}, where the Auslander-Reiten translation is realized by the rotation. The dimension of $\Ext^1$ between certain indecomposable objects is shown by Zhang, Zhou and Zhu \cite{ZZZ} to equal the intersection number between the corresponding curves, and the middle terms between such extensions are explicitly described by Canakci and Schroll \cite{CS} via smoothing. The Calabi-Yau reduction introduced by Iyama and Yoshino \cite{IY} in this case is interpreted via the cutting of surface by Marsh and Palu \cite{MP}. For the punctured case (with non-empty boundary), Br\"{u}stle and Qiu \cite{BQ} realize the Auslander-Reiten translation via the tagged rotation, Qiu and Zhou \cite{QZ} classify certain indecomposable objects via tagged curves and show the equality between the dimension of $\Ext^1$ and the intersection number, and Amiot and Plamondon \cite{AP} give another approach via group actions and orbifolds.

The Jacobian algebra of the quiver with potential associated to a certain triangulation of a marked surface with non-empty boundary is a skew-gentle algebra with some properties (e.g. Gorenstein dimension at most one) \cite{ABCP,GLS,QZ}. Skew-gentle algebras were introduced by Gei{\ss} and de~la~Pe\~{n}a \cite{GP}, whose indecomposable modules are classified by Bondarenko \cite{B}, Crawley-Boevey \cite{CB} and Deng \cite{D}, and whose morphism spaces are described by Gei{\ss} \cite{G}. In a previous work \cite{HZZ}, we give a geometric model of the module category of an arbitrary skew-gentle algebra, inspired by the geometric model \cite{QZ} of cluster categories of punctured marked surfaces and the geometric model of the module categories of gentle algebras given by Baur and Sim\~{o}es \cite{BCS}. There are also some work on geometric models of the derived categories of gentle/skew-gentle algebras, cf. e.g.  \cite{HKK,LP,OPS,O,APS,AB,A2,LSV}.

Adachi, Iyama and Reiten \cite{AIR} introduced $\tau$-tilting theory to generalize the cluster structure to arbitrary finite dimensional algebras, via mutation of support $\tau$-tilting modules. The support $\tau$-tilting modules have been found to be deeply connected with other contents of representation theory, such as functorially finite torsion classes, 2-term silting objects, cluster tilting objects and immediate t-structures. In contrast to the classical tilting case, where an almost complete tilting module may have exactly one complement, any support $\tau$-tilting module can be always mutated at an arbitrary indecomposable direct summand to obtain a new support $\tau$-tilting module. The exchange graph $\operatorname{EG}(\st A)$ of support $\tau$-tilting modules of a finite dimensional algebra $A$ has (isoclasses of) basic support $\tau$-tilting modules over $A$ as vertices and has mutations as edges. One important problem is to count the number of connected components of $\operatorname{EG}(\st A)$.

In our previous work \cite{HZZ}, using a geometric model, we classify support $\tau$-tilting modules of skew-gentle algebras via certain dissections of marked surfaces. In the current paper, the main result is the connectedness of the exchange graph $\operatorname{EG}(\st A)$ for $A$ a skew-gentle algebra. For this, we need to generalize the geometric model from skew-gentle algebras to the endomorphism algebras of rigid objects in the cluster categories arising from punctured marked surfaces. This forces us to establish a framework on the theory of mutation in two-term subcategories of a 2-Calabi-Yau triangulated category with respect to rigid objects. We note that the connectedness of $\operatorname{EG}(\st A)$ for the case that $A$ is a gentle algebra is obtained in \cite{FGLZ}, and it is shown in \cite{As} that $\operatorname{EG}(\st A)$ has one or two components in the case that $A$ is a complete gentle (or more generally, special biserial) algebra.

The paper is organized as follows. In Section~\ref{sec:category}, we introduce and investigate mutation in two-term subcategories of a 2-Calabi-Yau triangulated category. In Section~\ref{sec:pms}, we recall basic notions and results on the cluster categories from punctured marked surfaces. In Section~\ref{sec:geo mod}, we give a geometric model for the endomorphism algebra of a rigid object in the cluster category arising from a punctured marked surface and show that this includes the class of skew-gentle algebras. Moreover, we classify support $\tau$-tilting modules via certain dissections. In Section~\ref{sec:eg}, we introduce the notion of flip of dissections and show that it is compatible with mutation of support $\tau$-tilting modules. As an application, the connectedness of the exchange graph of support $\tau$-tilting modules over a skew-gentle algebra is obtained.
	
\subsection*{Convention}
Throughout this paper, we assume $\k$ to be an algebraically closed field. Any additive category $\D$ in this paper is assumed to be
\begin{enumerate}
    \item $\k$-linear and Hom-finite, i.e., $\Hom_\D(X,Y)$ is a finite-dimensional vector space over $\k$ for any pair of objects $X,Y$, and
    \item Krull-Schmidt, i.e., any object is isomorphic to a finite direct sum of objects whose endomorphism rings are local.
\end{enumerate}
We use $X\in\D$ to denote that $X$ is an object in $\D$. For any $X\in\D$, denote by
\begin{enumerate}
    \item$|X|$ the number of isomorphism classes of indecomposable direct summands of $X$,
    \item $\add X$ the additive hull of $X$, i.e., the smallest subcategory of $\D$, which contains $X$ and is closed under isomorphisms, finite direct sums and direct summands, and
    \item ${}^\perp X$ and $X^\perp$ the full subcategories of $\D$ consisting of all objects $Y$ such that $\Hom_\D(Y,X)=0$ and $\Hom_\D(X,Y)=0$, respectively. 
\end{enumerate}
We call $X\in\D$ \emph{basic} if $|X|$ is the number of indecomposable direct summands of $X$, i.e., any two distinct indecomposable direct summands of $X$ are not isomorphic. For any object $X\in\D$ and any direct summand $Y$ of $X$, we denote by $X\setminus Y$ the direct summand of $X$ such that $X=Y\oplus(X\setminus Y)$.

We call a morphism $g\in\Hom_{\D}(X,Y)$ \emph{right minimal} if for any $h\in\Hom_{\D}(X,X)$ such that $g\circ h=g$, we have that $h$ is an isomorphism. For any subcategory $\mathcal{T}$ of $\D$, we call $f\in\Hom_\D(X,Y)$ a \emph{right $\mathcal{T}$-approximation} of $Y\in\D$ if $X\in\mathcal{T}$ and 
\[\Hom_{\D}(-,X)\xrightarrow{\Hom_\D(-,f)}\Hom_{\D}(-,Y)\longrightarrow0\]
is exact as functors on $\mathcal{T}$. A right $\mathcal{T}$-approximation is said to be \emph{minimal} if it is right minimal. We call $\mathcal{T}$ a \emph{contravariantly finite subcategory} of $\D$ if any $Y\in\D$ admits a right $\mathcal{T}$-approximation. The \emph{left minimal} maps, \emph{(minimal) left $\mathcal{T}$-approximations} and \emph{covariantly finite subcategories} are defined dually. A subcategory is said to be \emph{functorially finite} if it is both covariantly and contravariantly finite.

When $\D$ is a triangulated category, for any $X,Y\in\D$, denote by $X\ast_{\D} Y$ the full subcategory of $\D$ consisting of all $M\in\D$ such that there is a triangle
\[X_M\to M\to Y_M\to X[1]\] 
with $X_M\in\add X$ and $Y_M\in\add Y$. When there is no confusion arising, we simply denote $X\ast Y=X\ast_{\D} Y$.

\section{Categorical Interpretation}\label{sec:category}
Throughout this section, let $\D$ be a triangulated category and we use $\Hom(X,Y)$ to simply denote $\Hom_\D(X,Y)$. We assume that $\D$ is 2-Calabi-Yau, i.e., there exists a bi-functorial isomorphism
\[\Hom(X,Y)\cong D\Hom(Y,X[2]),\]
for any $X,Y\in\D$, where $D=\Hom_\k(-,\k)$. 

\begin{definition}
    An object $R\in\D$ is called
    \begin{enumerate}
        \item \emph{rigid} provided that $\Hom(R,R[1])=0$,
        \item \emph{maximal rigid} if it is maximal with respect to the rigid property, i.e., $R$ is rigid and for any object $N\in\D$ with $N\oplus R$ rigid, we have $N\in\add R$,
        \item \emph{cluster tilting} if $R$ is rigid and for any object $N\in\D$ with $\Hom(R,N[1])=0$, we have $N\in\add R$.
    \end{enumerate}
\end{definition}

Note that any cluster tilting object is maximal rigid, but the converse is not true in general (cf. \cite{BIKR,KZ}). We also note that the triangulated category $\D$ may not admit any cluster tilting object (cf. \cite{BMV}). If $\D$ admits a cluster tilting object, then any maximal rigid object is cluster tilting (see \cite{ZZ}). 

For a rigid object $R\in\D$, the full subcategory $R\ast R[1]$ of $\D$ is called the \emph{two-term subcategory} with respect to $R$. In this section, we will extend the theory of mutation of cluster tilting objects, or more generally, mutation of maximal objects in $\D$ \cite{IY,BMRRT, ZZ} to $R\ast R[1]$. 

By \cite[Proposition~2.1]{IY}, $R\ast R[1]$ is closed under taking direct summands. Any object $M\in R\ast R[1]$ admits an \emph{$R$-presentation}, i.e., a triangle
\begin{equation}\label{eq:pre}
	R^1_M\to R^0_M\overset{\iota_0}{\longrightarrow}M\overset{\iota_1}{\longrightarrow}R^1_M[1]
\end{equation}
with $R^0_M, R^1_M\in\add R$. Since $R$ is rigid, we have that $\iota_0$ is a right $\add R$-approximation of $M$ and $\iota_1$ is a left $\add R[1]$-approximation of $M$. Moreover, in \eqref{eq:pre}, $\iota_0$ can be chosen to be right minimal, or equivalently, $\iota_1$ can be chosen to be left minimal. In such case, we call \eqref{eq:pre} a \emph{minimal} $R$-presentation.

\begin{proposition}\label{prop:DK1}
    If \eqref{eq:pre} is a minimal $R$-presentation, then $R^1_M$ and $R^0_M$ do not have an indecomposable direct summand in common.
\end{proposition}

\begin{proof}
    The proof of \cite[Proposition~2.1]{DK} also works here.
\end{proof}

\subsection{Rigid objects in two-term subcategories}
Throughout the rest of this section, let $R$ be a basic rigid object in $\D$. We introduce the notion of maximal rigid objects with respect to $R\ast R[1]$.

\begin{definition}
	An object $U\in R\ast R[1]$ is called \emph{maximal rigid} with respect to $R\ast R[1]$ provided that it is rigid and for any object $N\in R\ast R[1]$ with $N\oplus U$ rigid, we have $N\in\add U$.
	
	Denote by $\r(R\ast R[1])$ the set of (isoclasses of) basic rigid objects in $R\ast R[1]$, and by $\mr(R\ast R[1])$ the set of (isoclasses of) basic maximal rigid objects with respect to $R\ast R[1]$.
\end{definition}

In the case that $R$ is maximal rigid (resp. cluster tilting), by \cite[Corollary~2.5]{ZZ}, any rigid object in $\D$ also belongs to $R\ast R[1]$. Therefore, the maximal rigid objects with respect to $R\ast R[1]$ are exactly the maximal rigid (resp. cluster tilting) objects in $\D$.

\begin{lemma}\label{lem:dual}
    For any $U\in\mr(R\ast R[1])$, we have $R\in\r(U[-1]\ast U)$. For any triangle
    \begin{equation}\label{eq:lem dual 1}
        U^0_R[-1]\overset{\iota_0}{\longrightarrow}R\overset{\iota_1}{\longrightarrow}U^1_R\overset{\iota}{\longrightarrow}U^0_R,
    \end{equation}
    the following hold.
    \begin{enumerate}
        \item If $\iota_0$ is a right $\add U[-1]$-approximation of $R$, then $U^1_R\in\add U$.
        \item If $\iota_1$ is a left $\add U$-approximation of $R$, then $U^0_R\in\add U$.
    \end{enumerate}
\end{lemma}

\begin{proof}
    Due to the existence of right $\add U[-1]$-approximations (resp. left $\add U$-approximations) of $R$, the assertion (1) (resp. (2)) implies $R\in\r(U[-1]\ast U)$. So it suffices to show (1) and (2). We only prove (1), since (2) can be proved dually. 
	
    Since $\iota_0$ is a right $\add U[-1]$-approximation, we have $U^0_R\in\add U$. Since $U\in R\ast R[1]$ and $R\ast R[1]$ is closed under taking direct summands, we have $U^0_R\in R\ast R[1]$. Hence $U^1_R\in R\ast U^0_R\subseteq R\ast R[1]$. Applying $\Hom(U[-1],-)$ to the triangle~\eqref{eq:lem dual 1}, we get a long exact sequence in $\D$
    \[\begin{array}{rl}
        &\Hom(U[-1],U^0_R[-1])\xrightarrow{\Hom(U[-1],\iota_0)}\Hom(U[-1],R)\xrightarrow{\Hom(U[-1],\iota_1)}\Hom(U[-1],U^1_R)  \\
	\to & \Hom(U[-1],U^0_R).
    \end{array}\]
    Since $U$ is rigid, the last term $\Hom(U[-1],U^0_R)=0$. Since $\iota_0$ is a right $\add U[-1]$-approximation of $R$, the morphism $\Hom(U[-1],\iota_0)$ is surjective. So $\Hom(U[-1],U^1_R)=0$, which implies $\Hom(U^1_R,U[1])=0$ by the 2-Calabi-Yau property. In particular, $\Hom(U^1_R,U^0_R[1])=0$.
		
    For any $f\in\Hom(U^1_R,U^1_R[1])$, consider the following diagram.
    $$\xymatrix@R=2em@C=3em{
        U^0_R[-1]\ar[r]&R\ar[r]^-{\iota_1}&U^1_R\ar@{.>}[dl]_(.35){f_1}\ar[r]^-{\iota}\ar[d]^(.7){f}&U^0_R\ar@{.>}[dll]^(.3){f_2}\ar@{.>}[dlll]_(.8){f_3}\\
        U^0_R\ar[r]_-{\iota_0[1]}&R[1]\ar[r]_-{\iota_1[1]}&U^1_R[1]\ar[r]_-{\iota[1]}&U^0_R[1]
    }$$
    Since $\iota[1]\circ f\in\Hom(U^1_R,U^0_R[1])=0$, there exists $f_1\in\Hom(U^1_R,R[1])$ such that $f=\iota_1[1]\circ f_1$. Since $f_1\circ \iota_1\in\Hom(R,R[1])=0$, there exists $f_2\in\Hom(U^0_R,R[1])$ such that $f_1=f_2\circ\iota$. Since $\iota_0[1]$ is a right $\add U$-approximation of $R[1]$, there exists $f_3\in\Hom(U^0_R,U^0_R)$ such that $f_2=(\iota_0[1])f_3$. Hence, we have $f=\iota_1[1]\circ\iota_0[1]\circ f_3\circ\iota$, which is zero since $\iota_1[1]\circ\iota_0[1]=0$. So $U^1_R$ is rigid, and hence $\Hom((U^1_R\oplus U),(U^1_R\oplus U)[1])=0$. Since $U$ is maximal rigid with respect to $R\ast R[1]$, we have $U^1_R\in\add U$.
\end{proof}

We use $K_0^{\mbox{sp}}(R)$ to denote the split Grothendieck group of $\add R$. For any $M\in R\ast R[1]$, define the \emph{index of $M$ with respect to $R$} as the element in $K_0^{\mbox{sp}}(R)$
\begin{equation}\label{eq:index}
    \ind_RM=[R^0_M]-[R^1_M],
\end{equation}
where $R^1_M\to R^0_M\to M\to R^1_M[1]$ is an $R$-presentation of $M$. We write $R=\oplus^n_{i=1}R_i$ with $R_i$ indecomposable. Then $[R_i],1\leq i\leq n$, form a $\mathbb{Z}$-basis of $K_0^{\mbox{sp}}(R)$. We write $$\ind_RM=\sum_{i=1}^{n}[\ind_RM:R_i][R_i].$$

\begin{remark}\label{rmk:cto}
    We refer to \cite[Section~2.3]{DK}, \cite[Section~2.1]{Pal} and \cite[Section~2.5]{Pla2} for a similar definition of index with respect to cluster tilting objects. Moreover, if $R$ is a direct summand of a cluster tilting object $T$, then for any object $M\in\D$, we have $M\in R\ast R[1]$ if and only if $[\ind_T M:X]=0$ for any indecomposable direct summand $X$ of $T\setminus R$.
\end{remark}

\begin{proposition}\label{prop:DK2}
    Let $U=\oplus^m_{i=1}U_i$ be a basic rigid object in $R\ast R[1]$ with $U_i, 1\leq i\leq m$, indecomposable. Then the elements $\ind_RU_i, 1\leq i\leq m$, are linearly independent in $K_0^{\mbox{sp}}(R)$.
\end{proposition}

\begin{proof}
    The proof of \cite[Theorem~2.4]{DK} also works here.
\end{proof}

We have the following criterion of a rigid object in $R\ast R[1]$ to be maximal rigid with respect to $R\ast R[1]$, by counting rank.

\begin{proposition}\label{prop:rank}
    For any $U\in\r(R\ast R[1])$, we have $|U|\leq|R|$, where equality holds if and only if $U\in\mr(R\ast R[1])$.
\end{proposition}
\begin{proof}
	Let $U=\oplus^m_{i=1}U_i$ be a basic rigid object in $R\ast R[1]$ with $U_i,1\leq i\leq m$, indecomposable. By Proposition~\ref{prop:DK2}, we have  
	\[|U|=\rank\{\ind_RU_j|1\leq j\leq m\}\leq\rank\{\ind_R R_i|1\leq i\leq n\}=|R|,\]
	where $\rank X$ denotes the rank of a set $X\subseteq K_0^{\mbox{sp}}(R)$.
	
	Let $U$ be a maximal rigid object with respect to $R\ast R[1]$. On the one hand, $U$ is rigid in $R\ast R[1]$, which implies $|U|\leq|R|$. On the other hand, by Lemma~\ref{lem:dual}, we have $R\in\r(U[-1]\ast U)$, which implies $|R|\leq|U|$. Hence we have $|R|=|U|$. Conversely, let $U\in\r(R\ast R[1])$ with $|U|=|R|$. For any $N\in R\ast R[1]$ such that $U\oplus N$ is rigid, we have $|U\oplus N|\leq|R|$. Then $|U|=|U\oplus N|$, which implies $N\in\add U$. So $U$ is maximal rigid with respect to $R\ast R[1]$.
\end{proof}

As a consequence of Lemma~\ref{lem:dual} and Proposition~\ref{prop:rank}, we have the following dual relation between two rigid objects.

\begin{corollary}\label{cor:obj dual}
    Let $U$ and $R$ be rigid objects in $\D$. Then $U\in\mr(R\ast R[1])$ if and only if $R\in\mr(U[-1]\ast U)$.
\end{corollary}

\begin{proof}
    For any $U\in\mr(R\ast R[1])$, by Lemma~\ref{lem:dual}, we have $R\in\r(U[-1]\ast U)$, and by Proposition~\ref{prop:rank} we have $|U|=|R|$. Then $|R|=|U[-1]|$. So by Proposition~\ref{prop:rank} again, we have $R\in\mr(U[-1]\ast U)$. The opposite implication can be obtained by switching $R$ with $U[-1]$.
\end{proof}

The following lemma is useful.

\begin{lemma}\label{cor:non-zero}
    Let $U$ and $R$ be basic rigid objects in $\D$. If $R\in\mr(U[-1]\ast U)$, then for any indecomposable summand $Y$ of $U$, we have $[\ind_{U[-1]}R:Y[-1]]\neq 0$.
\end{lemma}

\begin{proof}
    Let $N=U\setminus Y$. If $[\ind_{U[-1]}R:Y[-1]]= 0$ then, by the definition of index, $R$ admits an $N[-1]$-presentation. So by Proposition~\ref{prop:rank}, we have $|R|\leq |N|<|U|$. Since $R\in\mr(U[-1]\ast U)$, again by Proposition~\ref{prop:rank}, we have $|R|=|U|$, a contradiction.
\end{proof}

\subsection{Mutation in two-term subcategories}
By Proposition~\ref{prop:rank}, any rigid object in $R\ast R[1]$ can be completed to a maximal rigid object with respect to $R\ast R[1]$.

\begin{definition}
    A basic rigid object $N$ in $R\ast R[1]$ is called \emph{almost maximal rigid} with respect to $R\ast R[1]$ if $|N|=|R|-1$. 
    
    Let $N$ be an almost maximal rigid object with respect to $R\ast R[1]$. An indecomposable object $Y$ is called a \emph{completion} of $N$ if $N\oplus Y$ is maximal rigid with respect to $R\ast R[1]$.
\end{definition}

\begin{lemma}\label{lem:unique}
    Let $N$ be an almost maximal rigid object with respect to $R\ast R[1]$, and $Y,Y'$ be two non-isomorphic completions of $N$. Then $$[\ind_{(Y'\oplus N)[-1]}R:Y'[-1]][\ind_{(Y\oplus N)[-1]}R:Y[-1]]<0.$$
\end{lemma}

\begin{proof}
    Let $[\ind_{(Y\oplus N)[-1]}R:Y[-1]]=t$ and $[\ind_{(Y'\oplus N)[-1]}R:Y'[-1]]=s.$ Assume conversely $ts\geq 0$. By Lemma~\ref{cor:non-zero}, we have $t\neq 0$ and $s\neq 0$. So either $t<0$ and $s<0$, or $t>0$ and $s>0$. We only make a contradiction for the case $t<0$ and $s<0$, since the other case is similar. Let $U=N\oplus Y$ and $U'=N\oplus Y'$. Consider the following diagram
	$$\xymatrix{
		U_R^0[-1]\ar[r]\ar@{-->}[d]_-{\psi[-1]}&U_R^0[-1]\ar[r]^-{\delta}\ar@{-->}[d]_-{\phi}&R\ar@{=}[d]\ar[r]& U_R^1\ar@{-->}[d]_-{\psi}\\
		{U'}^0_R[-1]\ar[r]&{U'}^0_R[-1]\ar[r]_-{\delta'}&R\ar[r]& {U'}^1_R
	}$$
	where the first (resp. second) row is a minimal $U[-1]$-presentation (resp. $U'[-1]$-presentation) of $R$. Since $t<0$ and $s<0$, by Proposition~\ref{prop:DK1}, both $U_R^0$ and ${U'}^0_R$ belong to $\add N$, and $Y$ and $Y'$ are direct summands of $U^1_R$ and ${U'}^1_R$, respectively. So both $\delta$ and $\delta'$ are minimal right $\add N[-1]$-approximations of $R$. Hence there exists an isomorphism $\phi:U_R^0[-1]\to {U'}^0_R[-1]$ such that the middle square of the above diagram commutes. It follows that there exists an isomorphism $\psi:U_R^1\to {U'}_R^1$, and hence $Y\cong Y'$, a contradiction.  
\end{proof}

It follows from Lemma~\ref{lem:unique} that any almost maximal rigid object with respect to $R\ast R[1]$ has at most two completions. In what follows, we shall prove that the number of completions is exactly two. For this, we need the following notion of left/right mutation of a basic rigid object in $\D$ at an indecomposable summand, introduced in \cite[Definition~2.5]{IY}. 

\begin{definition}\label{def:mutation in sub}
Let $U=N\oplus Y$ be a basic rigid object in $\D$, with $Y$ an indecomposable direct summand of $U$. The \emph{right mutation} $\mu^+_Y(U)=N\oplus Z$ and the \emph{left mutation} $\mu^-_Y(U)=N\oplus W$ of $U$ at $Y$ are defined respectively by the triangles
\begin{equation}\label{eq:ritri}
    Z\xrightarrow{\alpha_N}N_0\xrightarrow{\alpha_Y}Y\xrightarrow{\alpha_Z}Z[1],
\end{equation}
\begin{equation}\label{eq:letri}
    Y\xrightarrow{\beta_Y}N_1\xrightarrow{\beta_N}W\xrightarrow{\beta_W}Y[1],
\end{equation}
where $\alpha_Y$ and $\beta_Y$ are minimal right and left $\add N$-approximations of $Y$, respectively. The triangles \eqref{eq:ritri} and \eqref{eq:letri} are called the \emph{right} and \emph{left exchange triangles} of $U$ at $Y$, respectively.
\end{definition}

In Definition~\ref{def:mutation in sub}, both $Z$ and $W$ are indecomposable and not isomorphic to $Y$, both $\mu^+_Y(U)$ and $\mu^-_Y(U)$ are rigid, and $|\mu^+_Y(U)|=|\mu^-_Y(U)|=|U|$, cf. \cite[Section~2.1]{MP}. Hence by Proposition~\ref{prop:rank}, we have the following result.

\begin{lemma}\label{lem:ind}
    Let $U$ be a basic rigid object in $\D$, with $Y$ an indecomposable direct summand of $U$. If $U$ is a basic maximal rigid object with respect to $R\ast R[1]$, then so is $\mu^\varepsilon_Y(U)$, provided that it is in $R\ast R[1]$, where $\varepsilon\in\{+,-\}$.
\end{lemma}

Note that any of $\mu^+_Y(U)$ and $\mu^-_Y(U)$ may not be in $R\ast R[1]$.

\begin{proposition}\label{prop:mutation in sub}
    Let $U=N\oplus Y\in\mr(R\ast R[1])$ and $Y$ be an indecomposable summand of $U$. We use the triangles~\eqref{eq:ritri} and \eqref{eq:letri}. The following are equivalent.
    \begin{enumerate}
        \item $[\ind_{U[-1]}R:Y[-1]]>0$.
        \item $\Hom(R,\alpha_Y)$ is surjective.
        \item $\mu^+_Y(U)\in\mr(R\ast R[1])$ and $[\ind_{\mu^+_Y(U)[-1]}R:Z[-1]]<0$.
    \end{enumerate}
    Dually, the following are equivalent.
    \begin{enumerate}
        \item[(1')] $[\ind_{U[-1]}R:Y[-1]]<0$.
        \item[(2')] $\Hom(\beta_Y,R[1])$ is surjective.
        \item[(3')] $\mu^-_Y(U)\in\mr(R\ast R[1])$ and $[\ind_{\mu^-_Y(U)[-1]}R:W[-1]]>0$.
    \end{enumerate}
\end{proposition}

\begin{proof}
    We only prove the equivalences between (1), (2) and (3), since the equivalences between (1'), (2') and (3') can be proved dually.
    
    $(1)\Rightarrow(2)$. Take a minimal $U[-1]$-presentation of $R$:
    \[U^1_R[-1]\longrightarrow U^0_R[-1]\xrightarrow{\iota_0} R\xrightarrow{\iota_1} U^1_R.\]
    Since $\iota_1$ is a left $\add U$-approximation of $R$, for any morphism $g\in\Hom(R,Y)$, there exists $h\in\Hom(U^1_R,Y)$ such that $g=h\circ\iota_1$. Since $[\ind_{U[-1]}R:Y[-1]]>0$, by Proposition~\ref{prop:DK1}, we have $U^1_R\in\add N$. Since $\alpha_Y$ is a right $\add N$-approximation of $Y$, there exists $h'\in\Hom(U^R_1,N_0)$ such that $h=\alpha_Y\circ h'$. Then we have $g=\alpha_Y\circ h'\circ\iota_1$, which implies $\Hom(R,\alpha_Y)$ is surjective.
    
    $(2)\Rightarrow(3)$. Let
	\[R^1_Y\to R^0_Y\overset{\iota^0_Y}{\longrightarrow}Y\overset{\iota^1_Y}{\longrightarrow} R^1_Y[1]\]
    be a minimal $R$-presentation of $Y$. Since $\Hom(R,\alpha_Y)$ is surjective, there exists $\iota'_0\in\Hom(R^0_Y,N_0)$ such that  $\iota^0_Y=\alpha_Y\circ\iota'_0$. So by the octahedral axiom, we have the following commutative diagram of triangles.
	$$\xymatrix{
		&Z\ar@{=}[r]\ar[d]&Z\ar[d]&\\
		R^0_Y\ar[r]^-{\iota'_0}\ar@{=}[d]&N_0\ar[r]\ar[d]^{\alpha_Y}&M\ar[r]\ar[d]&R^0_Y[1]\ar@{=}[d]\\
		R^0_Y\ar[r]_{\iota^0_Y}&Y\ar[r]\ar[d]&R^1_Y[1]\ar[r]\ar[d]&R^0_Y[1]\\
		&Z[1]\ar@{=}[r]&Z[1]&}$$
	Since $M\in N_0\ast R^0_Y[1]\subseteq R\ast R[1]$, we have $Z\in R^1_Y\ast M\subseteq R\ast R[1]$. Thus by Lemma~\ref{lem:ind}, we have $\mu^+_Y(U)=N\oplus Z\in\mr(R\ast R[1])$. 
	
	Applying $\Hom(R,-)$ to the right exchange triangle~\eqref{eq:ritri}, we have the following exact sequence
    \[\Hom(R,N_0)\xrightarrow{\Hom(R,\alpha_Y)}\Hom(R,Y)\longrightarrow\Hom(R,Z[1])\xrightarrow{\Hom(R,\alpha_N[1])}\Hom(R,N_0[1]).\]
    Since $\Hom(R,\alpha_Y)$ is surjective, we have $\Hom(R,\alpha_N[1])$ is injective. So by the 2-Calabi-Yau property, the morphism $\Hom(\alpha_N,R[1])$ is surjective. Hence, any $g\in\Hom(Z[-1],R)$ factors through $N_0[-1]$. So any right $\add N[-1]$-approximation of $R$ is also a right $\add\mu^+_Y(U)[-1]$-approximation of $R$. By Lemma~\ref{cor:non-zero}, it follows that $[\ind_{\mu^+_Y(U)[-1]}R:Z[-1]]<0$.
    
    $(3)\Rightarrow(1)$. Since $Z\not\cong Y$, this implication follows from Lemma~\ref{lem:unique} directly.
\end{proof}

Now we are ready to show that each almost maximal rigid object with respect to $R\ast R[1]$ has exactly two completions.

\begin{theorem}\label{thm:completion}
    Let $N$ be an almost maximal rigid object with respect to $R\ast R[1]$. Then there are exactly two complements $Y$ and $Y'$ of $N$. Moreover, we have 
    $$[\ind_{(N\oplus Y)[-1]}R:Y[-1]][\ind_{(N\oplus Y')[-1]}R:Y'[-1]]<0.$$
    In the case $[\ind_{(N\oplus Y)[-1]}R:Y[-1]]>0$ and $[\ind_{(N\oplus Y')[-1]}R:Y'[-1]]<0$, there is a triangle
    $$Y'\to E\to Y\to Y'[1],$$
    with $E\in \add N$, and which under the functor $\Hom(R,-)$ becomes an exact sequence
    $$\Hom(R,Y')\to\Hom(R,E)\to\Hom(R,Y)\to 0.$$
\end{theorem}

\begin{proof}
    By Proposition~\ref{prop:rank}, there is a completion $X$ of $N$. By Lemma~\ref{cor:non-zero}, we have $[\ind_{(N\oplus X)[-1]}R:X[-1]]\neq 0$.  If $[\ind_{(N\oplus X)[-1]}R:X[-1]]>0$, we take $Y=X$ and $Y'=\mu^+_X(N\oplus X)\setminus N$. Then by Proposition~\ref{prop:mutation in sub}, the triangle \eqref{eq:ritri} becomes the required one. If $[\ind_{(N\oplus X)[-1]}R:X[-1]]<0$, we take $Y'=X$ and $Y=\mu^-_{X}(N\oplus X)\setminus N$. Then by Proposition~\ref{prop:mutation in sub}, the triangle \eqref{eq:letri} becomes the required one.
\end{proof}

An alternative description of Theorem~\ref{thm:completion} is the following mutation version.

\begin{corollary}\label{cor:mut}
    Let $U\in\mr(R\ast R[1])$ and $Y$ be an indecomposable summand of $U$. Then there is a unique (up to isomorphism) object $\mu_Y(U)$ in $\mr(R\ast R[1])$ such that $\mu_Y(U)$ contains $U\setminus Y$ as a direct summand and is not isomorphic to $U$. Moreover, 
    $$\mu_Y(U)=\begin{cases}
    \mu_Y^+(U) & \text{if $[\ind_{U[-1]}R:Y[-1]]>0$,}\\
    \mu_Y^-(U) & \text{if $[\ind_{U[-1]}R:Y[-1]]<0$.}
    \end{cases}$$
\end{corollary}

\begin{remark}\label{rem:total1}
    By Proposition~\ref{prop:DK1}, $[\ind_{U[-1]}R:Y[-1]]>0$ (resp. $<0$) if and only if $[\ind_{U[-1]}G:Y[-1]]>0$ (resp. $<0$) for some indecomposable direct summand $G$ of $R$.
\end{remark}

\subsection{Compatibility with \texorpdfstring{$\tau$}s-tilting theory}
In this subsection, we show that the mutation in $\mr(R\ast R[1])$ is compatible with the mutation of $\tau$-tilting pairs over the endomorphism algebra $\End R$.

We briefly recall the $\tau$-tilting theory from \cite{AIR}. Let $\Lambda$ be a finite dimensional algebra. Denote by $\mod\Lambda$ the category of finitely generated right $\Lambda$-modules, and by $\tau$ the Auslander-Reiten translation in $\mod\Lambda$. For any $M\in\mod\Lambda$, we denote by $\Fac(M)$ the subcategory of $\mod\Lambda$ consisting of factor modules of direct sums of copies of $M$.

\begin{definition}
    Let $M,P\in\mod\Lambda$ with $P$ projective.
    \begin{enumerate}
        \item The module $M$ is called \emph{$\tau$-rigid} if $\Hom_{\Lambda}(M,\tau M)=0$. 
        \item The pair $(M,P)$ is called a \emph{$\tau$-rigid pair} if $M$ is $\tau$-rigid and $\Hom_{\Lambda}(P,M)=0$. 
        \item The pair $(M,P)$ is called a \emph{$\tau$-tilting} pair if it is a $\tau$-rigid pair and $|M|+|P|=|\Lambda|$. In this case, $M$ is called a \emph{support $\tau$-tilting} module. 
        \item The pair $(M,P)$ is called an \emph{almost complete $\tau$-tilting} pair if it is a $\tau$-rigid pair and $|M|+|P|=|\Lambda|-1$. In this case, $M$ is called an \emph{almost complete support $\tau$-tilting module}.
    \end{enumerate}
\end{definition}

For any basic support $\tau$-tilting module $M$, there is a unique $P$ (up to isomorphism) such that $(M,P)$ is a basic $\tau$-tilting pair. Hence, one can identify basic support $\tau$-tilting modules with basic $\tau$-tilting pairs. There is a partial order on the set of basic support $\tau$-tilting modules, given by $M\geq N$ if and only if $N\in\Fac M$.

\begin{theorem}[{\cite[Theorem~2.18 and Definition-Proposition 2.28]{AIR}}]\label{thm:air}
    Any basic almost complete $\tau$-tilting pair $(N,Q)$ is a direct summand of exactly two non-isomorphic basic $\tau$-tilting pairs $(M,P)$ and $(M',P')$. Moreover, either $M>M'$ or $M'>M$. 
\end{theorem}

In the setting of Theorem~\ref{thm:air}, suppose $M>M'$. Then $(M,P)$ is called the \emph{right mutation} of $(M',P')$ at $(N,Q)$ and denote $(M,P)=\mu^+_{(N,Q)}(M',P')$. Dually, $(M',P')$ is called the \emph{left mutation} of $(M,P)$ at $(N,Q)$ and denote $(M',P')=\mu^-_{(N,Q)}(M,P)$.

\begin{definition}
    The exchange graph $\operatorname{EG}(\st\Lambda)$ of support $\tau$-tilting modules over $\Lambda$ has basic $\tau$-tilting pairs as vertices and has mutations as edges.
\end{definition}

Let $\Lambda_R=\End R$. The following result establishes a link between $R\ast R[1]$ and $\mod\Lambda_R$.

\begin{theorem}[{\cite[Proposition~6.2]{IY}, \cite[Proposition~2.2, Theorem~3.2]{CZZ}}]\label{thm:CZZ}
	The functor $\Hom(R,-):\D\to\mod\Lambda_R$ induces an equivalence
	\[R\ast R[1]/R[1]\overset{\simeq}{\longrightarrow}\mod\Lambda_R,\]
	such that for any $X\in R\ast R[1]$ without non-zero common direct summands with $R[1]$, we have \[\tau\Hom(R,X)=\Hom(R,X[1]).\]
	Moreover, this equivalence induces a bijection
	\[\begin{array}{rccc}
	     \Phi\colon & \r(R\ast R[1]) & \to & \trp\mod\Lambda_R  \\
	     & X_1\oplus X_2 & \mapsto & (\Hom(R,X_1),\Hom(R,X_2[-1]))
	\end{array}
	\]
	where $\trp\mod\Lambda_R$ is the set of (isoclasses of) basic $\tau$-rigid pairs in $\mod\Lambda_R$, $X_2\in\add R[1]$ and $X_1$ has no non-zero common direct summands with $R[1]$. This bijection restricts to a bijection from $\mr(R\ast R[1])$ to the set of (isoclasses of) basic $\tau$-tilting pairs.
\end{theorem}

This allows us to apply our results of mutation on $R\ast R[1]$ to the $\tau$-tilting theory in $\mod\Lambda_R$.

\begin{theorem}\label{thm:mutation compatible}
    For any $U=N\oplus Y\in\mr(R\ast R[1])$ with $Y$ an indecomposable summand of $U$, we have
	\[\begin{cases}
	    \Phi(\mu^+_Y(U))=\mu^+_{\Phi(Y)}(\Phi(U))&\text{if $\mu^+_Y(U)\in R\ast R[1]$,}\\
	    \Phi(\mu^-_Y(U))=\mu^-_{\Phi(Y)}(\Phi(U))&\text{if $\mu^-_Y(U)\in R\ast R[1]$.}
	\end{cases}\]    
\end{theorem}

\begin{proof}
    Let $Y'$ be another completion of $N$. If $\mu^+_Y(U)\in R\ast R[1]$, then it is maximal rigid with respect to $R\ast R[1]$ by Lemma~\ref{lem:ind}. So by Corollary~\ref{cor:mut}, we have $\mu^+_Y(U)=N\oplus Y'$ and $[\ind_{U[-1]}R:Y[-1]]>0$. Then by Theorem~\ref{thm:completion}, there is a triangle
    $$Y'\to E\to Y\to Y'[1],$$
    with $E\in \add N$ and such that there is an exact sequence
    $$\Hom(R,Y')\to\Hom(R,E)\to\Hom(R,Y)\to 0.$$
    In particular, we have $\Hom(R,U)\in\Fac\Hom(R,N)\subseteq\Fac\Hom(R,\mu^+_Y(U))$. By definition, we have $\Phi(\mu^+_Y(U))>\Phi(U)$, which implies $\Phi(\mu^+_Y(U))=\mu^+_{\Phi(Y)}(\Phi(U))$ as required. The case $\mu^-_Y(U)\in R\ast R[1]$ can be proved dually.
\end{proof}

\subsection{Relations with reductions}

Let $G$ be a rigid object in $\D$. Denote by $\D_G={}^{\perp}G[1]/\<\add G\>$ the quotient of the full subcategory ${}^{\perp}G[1]$ by the ideal $\<\add G\>$ consisting of morphisms that factor through objects in $\add G$. For any morphism $f$ in ${}^{\perp}G[1]$, denote by $\bar{f}$ the image of $f$ in the quotient $\D_G$.

\begin{theorem}[{\cite[Theorem~4.2]{IY}}]\label{thm:IY}
    The category $\D_G$ is a triangulated category whose suspension functor $\<1\>_G$ and it inverse $\<-1\>_G$ are given by the triangles in $\D$
    \[Y\xrightarrow{\beta_Y}G_1\to Y\<1\>_G\to Y[1]\]
    and 
    \[Y[-1]\to Y\<-1\>_G\to G_0\xrightarrow{\alpha_Y}Y\]
    respectively, where $\beta_X$ (resp. $\alpha_Y$) is a left (resp. right) $\add G$-approximation of $Y$. The triangles in $\D_G$ are isomorphic to
    \[A\xrightarrow{\bar{f}} B\xrightarrow{\bar{g}} C\xrightarrow{\bar{h}} X\<1\>_G\]
    induced by the following diagram of triangles in $\D$
    $$\xymatrix{
	A\ar[r]^-{f}\ar@{=}[d] & B\ar[r]^-{g}\ar[d] & C\ar[d]^{h}\ar[r] & A[1]\ar@{=}[d]\\
	A\ar[r]_-{\beta_A} & N_A\ar[r] & A\<1\>\ar[r] & A[1]
    }$$
    with $A,B,C\in{}^{\perp}G[1]$.
\end{theorem}

We simply denote $\<1\>=\<1\>_G$ and $\<-1\>=\<-1\>_G$, if there is no confusion arising.

\begin{remark}\label{rem:mu}
    Let $U=N\oplus Y$ be a rigid object in $\D$. Comparing triangles~\ref{eq:ritri} and \ref{eq:letri} with Theorem~\ref{thm:IY}, we have $\mu^+_Y(U)=N\oplus Y\<-1\>_N$ and $\mu^-_Y(U)=N\oplus Y\<1\>_N$.
\end{remark}

Now consider the case that $G$ is a direct summand of $R$. Denote by $\ymr(R\ast_\D R[1])$ the subset of $\mr(R\ast_\D R[1])$ consisting of those objects that admit $G$ as a direct summand.

\begin{proposition}\label{prop:reduction}
	Let $G$ be a direct summand of $R$. Then there is a bijection 
	$$\Psi\colon \ymr(R\ast_\D R[1])\to\mr(R\ast_\Y R\<1\>),$$
	sending $U$ to $U\setminus G$. Moreover, for any $U\in\ymr(R\ast_\D R[1])$ and any indecomposable summand $Y$ of $U\setminus G$, we have that 
	$\mu^{\varepsilon}_Y(U)\in\mr(R\ast_\D R[1])$ if and only if $\mu^{\varepsilon}_{\Psi(Y)}(\Psi(U))\in\mr(R\ast_\Y R\<1\>),$
	for any $\varepsilon\in\{+,-\}$, and in this case, $\Psi(\mu^{\varepsilon}_Y(U))=\mu^{\varepsilon}_{\Psi(Y)}(\Psi(U))$.
\end{proposition}

\begin{proof}
    Let $U=G\oplus X$ be a basic object in $R\ast R[1]$. By \cite[Lemma~4.8]{IY}, $U$ is rigid in $\D$ if and only if $X$ is rigid in $\Y$. Since an indecomposable object in $R\ast R[1]$ is isomorphic to zero in $\Y$ if and only if it is isomorphic to a direct summand of $G$, we have $|U|=|R|$ in $\D$ if and only if $|X|=|R\setminus G|$ in $\Y$. So by Proposition~\ref{prop:rank}, we have that $U\in\ymr(R\ast_\D R[1])$ if and only if $X\in\mr(R\ast_\Y R\<1\>)$. Thus, we get the bijection $\Psi$.

    By Theorem~\ref{thm:IY}, a minimal $U$-presentation
    \[R\xrightarrow{f} U^1_R\xrightarrow{g} U^0_R\to R[1]\]
    of $R[1]$ in $\D$ gives rise to a minimal $X$-presentation
    \[R\xrightarrow{\overline{f}} U^1_R\xrightarrow{\overline{g}} U^0_R\to R\<1\>\]
    of $R\<1\>$ in $\Y$. So for any indecomposable summand $Y$ of $X$, we have 
    $$[\ind_{U[-1]}R:Y[-1]]=[\ind_{U\<-1\>}R:Y\<-1\>].$$ Then for any $\varepsilon\in\{+,-\}$, by Proposition~\ref{prop:mutation in sub}, we have $\mu^{\varepsilon}_Y(U)\in\ymr(R\ast_\D R[1])$ if and only if $\mu^{\varepsilon}_{\Psi(Y)}(\Psi(U))\in\mr(R\ast_\Y R\<1\>)$. Then the last assertion follows from \cite[Proposition 4.4~(2)]{IY}.
\end{proof}

\section{Cluster categories arising from punctured marked surfaces}\label{sec:pms}
In this section, we recall the cluster categories arising from punctured marked surfaces.

\subsection{Punctured marked surfaces and triangulations}\label{subsec:pms}
We recall from \cite{FST} and \cite{QZ} some notions about punctured marked surfaces.

A \emph{punctured marked surface} $\surf$ is a triple $(S,\MM,\PP)$, where 
\begin{itemize}
    \item $S$ is a compact oriented surface with non-empty boundary $\partial S$,
    \item $\MM\subseteq\partial S$ is a finite set of \emph{marked points} such that each component of $\partial S$ contains at least one marked point in $\MM$, and
    \item $\PP\subseteq(S\setminus\partial S)$ is a finite set of \emph{punctures}.
\end{itemize}
Denote $\surfi:=S\setminus(\partial S\cup\PP)$ the interior of $\surf$. The closures of connected components in $\partial S\setminus\MM$ are called \emph{boundary segments}. Denote by $B$ the set of boundary segments of $\surf$.
	
An \emph{arc} on $S$ is an immersion $\gamma:[0,1]\to S$ such that
\begin{enumerate}
	\item $\gamma(0),\gamma(1)\in\MM\cup\PP$, $\gamma(t)\in \surfi$ for any $t\in(0,1)$,
	\item $\gamma$ is neither null-homotopic nor homotopic to a boundary segment, and
	\item $\gamma$ has no self-intersections in $\surfi$.
\end{enumerate}
Any arc is considered up to homotopy relative to its endpoints.

\begin{definition}[{\cite[Definition~2.6]{FST}}]
    An \emph{ideal triangulation} of $\surf$ is a maximal collection of arcs on $\surf$ that do not intersect each other in $\surfi$. 
\end{definition}

Any ideal triangulation divides $S$ into triangles. A triangle is called \emph{self-folded} if two of its sides coincide, as shown in the left picture of  Figure~\ref{fig:self-fold}, where we use $\epsilon'$ to denote the folded side of a self-folded triangle whose non-folded side is $\epsilon$. Note that the point $p$ enclosed by $\epsilon$ is always in $\PP$.

\begin{definition}\label{def:ideal tri}
    A \emph{partial ideal triangulation} $\R$ is a subset of an ideal triangulation $\T$ such that for any self-folded triangle of $\T$, if its non-folded side is in $\R$ then so is its folded side. 
\end{definition}

\begin{definition}[{\cite[Definition~7.1]{FST}}]\label{def:tag arc}
    A \emph{tagged arc} on $\surf$ is an arc $\gamma$ that does not cut out a once-punctured monogon by a self-intersection in $\MM\cup\PP$, and equipped with a map 
    $$\kappa_\gamma:\{t|\gamma(t)\in\PP\}\to\{-1,1\}$$
    such that $\kappa_\gamma(0)=\kappa_\gamma(1)$ if $\gamma(0)=\gamma(1)\in\PP$. The value $\kappa(t)$ is called the \emph{tagging} of $\gamma$ at the end $\gamma(t)$. Denote by $\TA(\surf)$ the set of tagged arcs on $\surf$. 
\end{definition}

In figures, we use the symbol $\times$ on one end of a tagged arc $\gamma$ to stand for that the value of $\kappa_\gamma$ at this end is taken to be $-1$.

\begin{remark}\label{rmk:crp}
    Each arc $\gamma$ can be viewed as a tagged arc $\gamma^\times$ whose underlying arc is $\gamma$ and whose tagging is 1 at each punctured end.
    
    For an arc $\epsilon$ which is a loop enclosing a puncture $p$, although it can not be completed to a tagged arc by adding tagging directly, we still associate a tagged arc $\epsilon^\times$ to it, whose underlying arc is $\epsilon'$, the unique arc in the interior of the once-puncture monogon enclosed by $\epsilon$, and whose tagging is $-1$ on the end at $p$. See Figure~\ref{fig:self-fold}.
\end{remark}

\begin{figure}[htpb]\centering
	\begin{tikzpicture}[xscale=1,yscale=1]
		\draw[red,thick](0,1)to[out=-135,in=60](-1,0)to[out=-120,in=180](0,-1)to[out=0,in=-60](1,0)node[left]{$\epsilon$}to[out=120,in=-45](0,1);
		\draw(0,1)node{$\bullet$};
		\draw(0,-.5)node{$\bullet$};
		\draw[red,thick](0,1)to(0,-.5);
		\draw(2.2,0)node{$\Leftrightarrow$}(3.85,0)node[red]{$\times$};
		\draw[red,thick,bend right=20](4,1)node[black]{$\bullet$}to(4,-.5)node[black]{$\bullet$}to(4,1);
		\draw(0,-.5)node[below]{$p$}(4,-.5)node[below]{$p$}(0,0)node[red][left]{$\epsilon'$}(3.2,.5)node[red]{$\epsilon^\times$}(4.7,.5)node[red]{$(\epsilon')^\times$};
	\end{tikzpicture}
	\caption{From ideal triangulations to tagged triangulations}
	\label{fig:self-fold}
\end{figure}
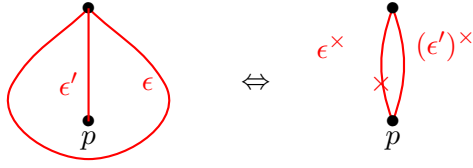

\begin{definition}[{\cite[Definition~3.3]{QZ}}]\label{def:intersection}
    Let $\gamma_1$ and $\gamma_2$ be tagged arcs in minimal position.
    \begin{enumerate}
        \item Any pair $(t_1,t_2)$ with $0<t_1,t_2< 1$ and $\gamma_1(t_1)=\gamma_2(t_2)$ is called an \emph{interior intersection} between $\gamma_1$ and $\gamma_2$.
        \item A pair $(t_1,t_2)$ with $t_1,t_2\in\{0,1\}$ and $\gamma_1(t_1)=\gamma_2(t_2)\in\PP$ is called a \emph{tagged intersection} between $\gamma_1$ and $\gamma_2$ if the following conditions hold.
        \begin{enumerate}
		\item $\kappa_{\gamma_1}(t_1)\neq\kappa_{\gamma_2}(t_2)$.
		\item If $\gamma_1|_{t_1\to(1-t_1)}\sim\gamma_2|_{t_2\to(1-t_2)}$, where $\gamma_i|_{t_i\to(1-t_i)}$ denotes the orientation of $\gamma_i$ from $t_i$ to $1-t_i$, for $i=1,2$, then $\gamma_1(1-t_1)=\gamma_2(1-t_2)\in\PP$ and $\kappa_{\gamma_1}(1-t_1)\neq\kappa_{\gamma_2}(1-t_2)$.
	\end{enumerate}
    \end{enumerate}
    Denote by $\cap(\gamma_1,\gamma_2)$ the set of interior intersections and tagged intersections, and by $\Int(\gamma_1,\gamma_2)=|\cap(\gamma_1,\gamma_2)|$ the \emph{intersection number} between the tagged arcs $\gamma_1$ and $\gamma_2$.
\end{definition}

For any two tagged arcs $\gamma_1$ and $\gamma_2$, $\Int(\gamma_1,\gamma_2)=0$ if and only if they are compatible in the sense of \cite[Definition~7.4]{FST}.

\begin{definition}
    A \emph{tagged triangulation} of $\surf$ {\cite[Section~7]{FST}} is a maximal collection $\T$ of tagged arcs on $\surf$ such that $\Int(\gamma_1,\gamma_2)=0$ for any $\gamma_1,\gamma_2\in\T$. A {partial tagged triangulation} of $\surf$ is a subset of a tagged triangulation.
\end{definition}

Let $\R$ be a partial tagged triangulation of $\surf$. We denote by $\PP_\R$ the subset of $\PP$ consisting of punctures $p$ that is an endpoint of an arc in $\R$. When $\R=\T$ is a tagged triangulation, we have $\PP_\T=\PP$. We define a map $\kappa_\R:\PP_\R\to \{-1,0,1\}$ as
\[
    \kappa_\R(p)=\begin{cases}
	-1&\emph{if any arc in $\R$ incident to $p$ has tagging $-1$ there,}\\
	1&\emph{if any arc in $\R$ incident to $p$ has tagging $1$ there,}\\
	0&\emph{otherwise.}
\end{cases}
\]

\begin{remark}\label{rmk:id}
    Using the correspondence $\gamma\mapsto\gamma^\times$ given in Remark~\ref{rmk:crp}, each partial ideal triangulation $\R$ corresponds to a partial tagged triangulation $\R^\times$.  By identifying $\R$ with $\R^\times$, we regard a partial ideal triangulation $\R$ as a partial tagged triangulation such that $\kappa_\R(p)\geq 0$ for any $p\in\PP_\R$. In particular, we regard each ideal triangulation as a tagged triangulation in such a way.
\end{remark}

Conversely, we associate a partial ideal triangulation with any partial tagged triangulation as follows.

\begin{construction}\label{cons:tag to ideal}
    Let $\R$ be a partial tagged triangulation. We construct a partial ideal triangulation $\R^\circ$ by
    \begin{enumerate}
        \item for $p\in\PP_\R$ with $\kappa_\R(p)\neq 0$, removing the taggings of tagged arcs in $\R$ at $p$, and
        \item for $p\in\PP_\R$ with $\kappa_\R(p)=0$, replacing the two tagged arcs in $\R$ incident to $p$ by a self-folded triangle enclosing $p$, as shown in Figure~\ref{fig:self-fold} (from right to left).
    \end{enumerate}
    Note that when $\R=\T$ is a tagged triangulation, $\T^\circ$ is an ideal triangulation. For any tagged arc $\gamma\in\R$, we denote by $\gamma^{\circ_\R}$ the arc in $\R^\circ$ corresponding to $\gamma$. Usually, we simply denote $\gamma^{\circ_\R}$ by $\gamma^{\circ}$, when there is no confusion arising.
\end{construction}

Note that we have $(\R^\times)^\circ=\R$ for any partial ideal triangulation $\R$, but $(\R^\circ)^\times\neq \R$ for a partial tagged triangulation $\R$ in general, where the only difference is the taggings of arcs at $p\in\PP_\R$ with $\kappa_\R(p)<0$.

\begin{construction}\label{cons:change tag}
    Let $\R$ be a partial tagged triangulation. For any tagged arc $\delta$, we construct a tagged arc $\delta^\R$, which is obtained from $\delta$ by changing taggings at $p\in\PP_\R$ with $\kappa_\R(p)<0$.
\end{construction}

\begin{remark}\label{rmk:good completion}
    For any partial tagged triangulation $\R$, there exists a tagged triangulation $\T$ such that $\R\subset\T$ and $$\kappa_{\T}(p)=\begin{cases}
    \kappa_{\R}(p)&\text{if $p\in\PP_\R$,}\\
    1&\text{otherwise.}\end{cases}$$
    Then we have
    $\gamma^{\circ_\R}=\gamma^{\circ_\T}$ for any $\gamma\in\R$, and $\delta^\T=\delta^\R$ for any $\delta\in\TA(\surf)$.
\end{remark}


\subsection{Laminates and shear coordinates}\label{subsec:lam}

We recall from \cite{FT} (cf. also \cite{Rea,Y}) the notion of laminates and their shear coordinates. The \emph{elementary laminate} $e(\delta)$ of a tagged arc $\delta\in\TA(\surf)$ is defined as follows.
\begin{itemize}
    \item $e(\delta)$ is an arc running along $\delta$ in a small neighbourhood of it;
    \item If $\delta$ has an endpoint $m$ on $\partial S$, then the corresponding endpoint of $e(\delta)$ is located near $m$ on $\partial S$ in the clockwise direction as in the left picture of Figure~\ref{fig:elt lam};
    \item If $\delta$ has an endpoint at a puncture $p$, then the corresponding end of $e(\delta)$ is a spiral around $p$ clockwise (resp. anticlockwise) if $\kappa_\delta$ takes value $1$ (resp. $-1$) at $p$ as in the right picture of Figure~\ref{fig:elt lam}.
\end{itemize}
The \emph{co-elementary laminate} $e^{op}(\delta)$ of $\delta$ is defined in the opposite direction.
\begin{figure}[htpb]\centering
	\begin{tikzpicture}[xscale=1,yscale=.5]
		\draw[ultra thick](-3,1)to[out=0,in=90](-2,0)node{$\bullet$}to[out=-90,in=0](-3,-1);
		\draw[ultra thick](3,1)to[out=180,in=90](2,0)node{$\bullet$}to[out=-90,in=180](3,-1);
		\draw[blue,thick](-2,0)to(2,0);
		\draw[blue](-1.5,.2)node{$\delta$};
		\draw[blue,thick](-2.3,-.7)to[out=0,in=180](2.3,.7);
		\draw[blue](1.,.8)node{$e(\delta)$};
	\end{tikzpicture}
	\qquad
	\begin{tikzpicture}
		\draw[blue,thick](-2,0)node[black]{$\bullet$}to(2,0)node[black]{$\bullet$};
		\draw[blue,thick](-2,.3)to[out=0,in=0](-2,-.3)to[out=180,in=180](-2,.5)to[out=0,in=180](2,.5)to[out=0,in=0](2,-.3)to[out=180,in=180](2,.3);
		\draw(-1.5,0)node[blue]{$\times$}(2.7,0)node[blue]{$e(\delta)$}(0,-.3)node[blue]{$\delta$};
	\end{tikzpicture}
	\caption{Elementary laminate}	\label{fig:elt lam}
\end{figure}
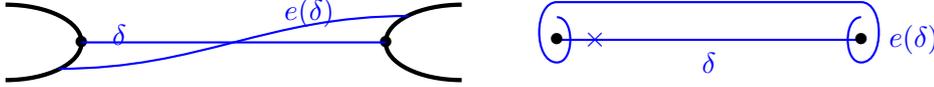

\begin{definition}[{\cite[Definition~3.4]{BQ}}]\label{def:BQ}
    The tagged rotation $\rho(\gamma)$ of a tagged arc $\gamma\in\TA(\surf)$ is obtained from $\gamma$ by moving each endpoint of $\gamma$ that is in $\MM$ along the boundary anticlockwise to the next marked point and changing the tagging as $\kappa_{\rho(\gamma)}(t)=-\kappa_{\gamma}(t)$ for any $t$ with $\gamma(t)\in\PP$.
\end{definition}

\begin{remark}\label{rmk:erho}
    For any $\delta\in\TA(\surf)$, we have $e(\rho(\delta))=e^{op}(\delta)$ and $e^{op}(\rho^{-1}(\delta))=e(\delta)$.
\end{remark}

\begin{definition}\label{def:shear}
    Let $L$ be the elementary laminate $e(\delta)$ or the co-elementary laminate $e^{op}(\delta)$ of a tagged arc $\delta$. Let $\R$ be a partial tagged triangulation.  Define $L^{\R}=e(\delta^\R)$ for the case $L=e(\delta)$, and $L^{\R}=e^{op}(\delta^\R)$ for the case $L=e^{op}(\delta)$, where $\delta^\R$ is given in Construction~\ref{cons:change tag}. We call $L$ \emph{shears} $\R$ provided that each segment of $L^\R$ divided by arcs in $\R^\circ$ cuts out an angle between two arcs in $\R^\circ\cup B$. 
\end{definition}

Note that when $\R=\T$ is a tagged triangulation, $L$ always shears $\T$.

Using the notations in Definition~\ref{def:shear}, if $L$ shears $\R$, the \emph{shear coordinate vector}
$$b_\R(L)=(b_{\gamma,\R}(L))_{\gamma\in\R}\in \mathbb{Z}^{\R}$$
of $L$ with respect to $\R$ is defined as follows.

For the case that $\R$ is a partial ideal triangulation, let $\gamma$ be an arc in $\R$ and $q$ an intersection between $\gamma$ and $L$. If $\gamma$ is not the folded side of a self-folded triangle of $\R$, then $\gamma$ is the common edge of the two angles of $\R$ that $L$ cuts out consecutively. The \emph{contribution} $b_{q,\gamma,\R}(L)$ of $q$ is defined as shown in Figure~\ref{fig:shear}.
\begin{figure}[htpb]
    \begin{tikzpicture}[scale=1.5]
		\draw[red,thick,bend right](-1,.5)node[black]{$\bullet$}to(0,1)node[black]{$\bullet$}to(1,0.5)node[black]{$\bullet$} (1,-0.5)node[black]{$\bullet$}to(0,-1)node[black]{$\bullet$}to(-1,-.5)node[black]{$\bullet$};
		\draw[red,thick](0,1)to(0,-1);
		\draw[blue,thick,bend right=10](-1,-1)to(1,1);
		\draw (0,-1.5)node{$b_{q,\gamma,\R}(L)=1$};
		\draw[red] (0,.1)node[left]{$\gamma$};
		\draw[blue] (.7,1)node{$L$};
		\draw (.15,-.25)node{$q$};
		\draw (0,-.2)node{$\boldsymbol{\cdot}$};
	\end{tikzpicture}
	\qquad
	\begin{tikzpicture}[scale=1.5]
		\draw[red,thick,bend right](-1,.5)node[black]{$\bullet$}to(0,1)node[black]{$\bullet$}to(1,0.5)node[black]{$\bullet$} (1,-0.5)node[black]{$\bullet$}to(0,-1)node[black]{$\bullet$}to(-1,-.5)node[black]{$\bullet$};
		\draw[red,thick](0,1)to(0,-1);
		\draw[blue,thick,bend right=10](-1,1)to(1,-1);	
		\draw (0,-1.5)node{$b_{q,\gamma,\R}(L)=-1$};
		\draw[red] (0,.1)node[right]{$\gamma$};
		\draw[blue] (-.7,1)node{$L$};
		\draw (-.15,-.25)node{$q$};
		\draw (0,-.2)node{$\boldsymbol{\cdot}$};
	\end{tikzpicture}
	\qquad
	\begin{tikzpicture}[scale=1.5]
		\draw[red,thick,bend right](-1,.5)node[black]{$\bullet$}to(0,1)node[black]{$\bullet$}to(1,0.5)node[black]{$\bullet$} (1,-0.5)node[black]{$\bullet$}to(0,-1)node[black]{$\bullet$}to(-1,-.5)node[black]{$\bullet$};
		\draw[red,thick](0,1)to(0,-1);
		\draw[blue,thick,bend right](-1,.7)to(1,.7);	
        \draw (0,-1.5)node{$b_{q,\gamma,\R}(L)=0$};
		\draw[red] (0,0)node[right]{$\gamma$};
		\draw[blue] (-.7,.9)node{$L$};
		\draw (-.15,.3)node{$q$};
		\draw (0,.4)node{$\boldsymbol{\cdot}$};
	\end{tikzpicture}
    \caption{Contribution of an intersection $q$ between $L$ and $\gamma\in\R$}\label{fig:shear}
\end{figure}
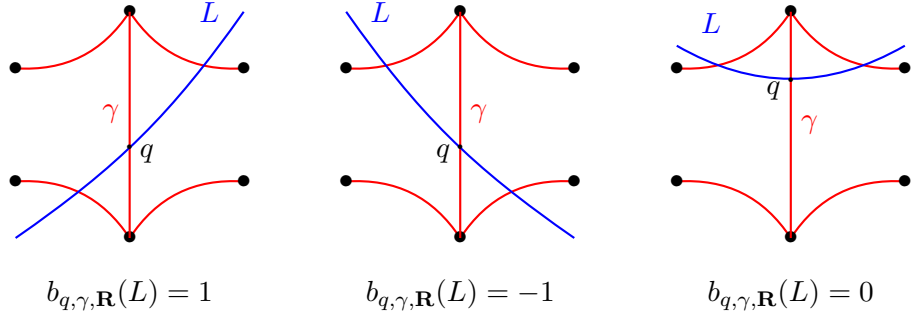
In this case, define
$$b_{\gamma,\R}(L)=\sum_{q\in\gamma\cap L} b_{q,\gamma,\R}(L),$$
where $\gamma\cap L$ is the set of intersections between $\gamma$ and $L$. Since only finitely many $q$ contribute nonzero values, the sum is well-defined. If $\gamma=\epsilon'$ is the folded side of a self-folded triangle of $\R$ whose non-folded side is $\epsilon$ and whose enclosing puncture is $p$ (cf. the left picture of Figure~\ref{fig:self-fold}), define
$$b_{\gamma,\R}(L)=b_{\epsilon,\R}(L^{(p)}),$$
where $L^{(p)}$ is obtained from $L$ by changing the directions of its spirals at $p$ if they exist.

For the case that $\R$ is a partial tagged triangulation, for any $\gamma\in\R$, define
$$b_{\gamma,\R}(L)=b_{\gamma^\circ,\R^{\circ}}(L^\R),$$
where $\R^\circ$ and $\gamma^\circ$ are defined in Construction~\ref{cons:tag to ideal}, and $L^\R$ is defined in Definition~\ref{def:shear}. By definition, for any tagged triangulation $\T$ containing $\R$ and for any $\gamma\in\R$, we have
\begin{equation}\label{eq:R=T}
    b_{\gamma,\R}(L)=b_{\gamma,\T}(L).
\end{equation}

\begin{remark}\label{rem:total}
    Let $\R$ be a partial tagged triangulation of $\surf$ and $\gamma\in\R$. Let $\delta$ a tagged arc and let $L=e(\delta)$ or $e^{op}(\delta)$. If there exists an intersection $q$ between $\gamma$ and $L$ such that $b_{q,\gamma,\R}(L)=1$ (resp. $-1$), then for any intersection $q'$ between $\gamma$ and $L$, we have $b_{q',\gamma,\R}(L)\geq 0$ (resp. $\leq 0$). This is because $L$ has no self-intersections.
\end{remark}

\subsection{Cluster categories from punctured marked surfaces}\label{subsec:QP and cluster}

A quiver is a 4-tuple $Q=(Q_0,Q_1,s,t)$, where $Q_0$ a set of vertices, $Q_1$ a set of arrows, and $s,t:Q_1\to Q_0$ are the maps sending an arrow to its start and terminal, respectively. A quiver with potential is a pair $(Q,W)$ with $Q$ a quiver and $W$ a linear combination of cycles of $Q$. For more details on quivers with potential, we refer to \cite{DWZ}.

Let $\surf$ be a punctured marked surface. To a tagged triangulation $\T$ of $\surf$, there is an associated quiver with potential $(Q^\T,W^\T)$ \cite{FST,LF1,LF3}, with $Q_0^\T=\T$. By \cite[Theorem~5.7]{LF3}, the Jacobian algebra $\Lambda^\T$ of $(Q^\T,W^\T)$ is finite dimensional and is isomorphic to 
$$\k Q^\T/\<\partial W^\T\>,$$
where $\k Q^\T$ is the path algebra of $Q^\T$ over the field $\k$ and $\<\partial W^\T\>$ is the ideal of $\k Q^\T$ generated by $\partial W^\T=\{\partial_a W^\T\mid a\in Q^\T_1\}$. Then by \cite{A1}, there is a $\k$-linear, Hom-finite, Krull-Schmidt, 2-Calabi-Yau triangulated category $\C(Q^\T,W^\T)$, called (generalized) cluster category associated to $(Q^\T,W^\T)$, with a cluster tilting object $T^\T$ such that there is an equivalence of categories
$$\C(Q^\T,W^\T)/\add T^\T[1]\simeq\mod\Lambda^\T.$$
Moreover, by \cite{FST,LF1,LF3,KY}, up to equivalence, the category $\C(Q^\T,W^\T)$ does not depend on the choice of the triangulation $\T$. So we can use $\C(\surf)$ to denote this category.
\begin{theorem}[\cite{QZ}]\label{thm:QZ}
    There is a bijection 
    $$X:\TA(\surf)\to\ind\operatorname{rigid}\C(\surf)$$ 
    from the set $\TA(\surf)$ of tagged arcs on $\surf$ to the set $\ind\operatorname{rigid}\C(\surf)$ of isoclasses of indecomposable rigid objects in $\C(\surf)$, such that the following hold.
    \begin{enumerate}
        \item For any $\gamma\in\TA{(\surf)}$, we have
        $$X(\rho(\gamma))=X(\gamma)[1].$$
        \item For any tagged arcs $\gamma_1,\gamma_2$, we have
        $$\Int(\gamma_1,\gamma_2)=\dim_\k\Hom(X(\gamma_1),X(\gamma_2)[1]).$$
        \item The bijection $X$ induces a bijection 
        $$\begin{array}{rccc}
        X:&\PTT(\surf)&\to&\r\C(\surf)\\ &\R&\mapsto&\bigoplus_{\gamma\in\R}X(\gamma)
        \end{array}$$
        from the set $\PTT(\surf)$ of partial tagged triangulations of $\surf$ to the set $\r\C(\surf)$ of isoclasses of basic rigid objects in $\C(\surf)$, such that $\R$ is a tagged triangulation if and only if $X(\R)$ is a cluster tilting object.
        \item[(4)] The cluster exchange graph of $\C(\surf)$, i.e. the graph whose vertices are cluster tilting objects in $\C(\surf)$ and whose edges are mutations, is connected.
    \end{enumerate} 
\end{theorem}

Recall from \eqref{eq:index} the definition of index $\ind_TM$ of an object $M$ in $\C(\surf)$ with respect to a cluster tilting object $T$.

\begin{proposition}\label{prop:index}
    Let $\T$ be a tagged triangulation of $\surf$. For any tagged arc $\delta$, we have
    \begin{equation}\label{eq:indb1}
        \ind_{X(\T)}X(\delta)=-b_{\T}(e(\delta)),
    \end{equation}
    and
    \begin{equation}\label{eq:indb2}
        \ind_{X(\T)[-1]}X(\delta)=-b_{\T}(e^{op}(\delta)).
    \end{equation}
\end{proposition}

\begin{proof}
    Starting from the quiver $Q^\T$, there is a cluster algebra $\mathcal{A}$ \cite{FZ}. By \cite[Theorem~7.11]{FST}, there is a bijection $x$ from the set $\TA(\surf)$ of tagged arcs on $\surf$ to the set of cluster variables of $\mathcal{A}$, such that $x$ induces a bijection from the set of tagged triangulations to the set of clusters of $\mathcal{A}$, which sends $\T$ to the initial cluster and commutes with flips of tagged triangulations and mutations of clusters.
    
    For the 2-Calabi-Yau category $\C(\surf)$ with cluster tilting object $T^\T$, by \cite{Pal} and Theorem~\ref{thm:QZ}~(4), there is a bijection $cc$ (called cluster character) from the set of indecomposable objects in $\C(\surf)$ to the set of cluster variables of $\mathcal{A}$, such that $cc$ induces a bijection from the set of basic cluster tilting objects to the set of reachable clusters of $\mathcal{A}$, which sends $T^\T$ to the initial cluster and commutes with mutations.
    
    Then by Theorem~\ref{thm:QZ}, we have the following commutative diagram of bijections
    $$\xymatrix{
    &\{\text{cluster variables in }\mathcal{A}\}\\
    \TA(\surf)\ar[ru]^{x}\ar[rr]_{X}&&\ind\operatorname{rigid}\C(\surf)\ar[lu]_{cc}
    }$$
    So for any tagged arc $\delta$, $x(\delta)=cc(X(\delta))$. By \cite[Proposition~5.2]{Rea} (see also \cite[Theorem~7.1]{LF2}), the $g$-vector of $x(\delta)$ is $-b_\T(e(\delta))$. On the other hand, by \cite[Proposition~3.6]{Pla2}, the $g$-vector of $cc(X(\delta))$ is $\ind_{T^\T}X(\delta)$. Thus, we get \eqref{eq:indb1}.
    
    For \eqref{eq:indb2}, we have 
    \begin{align*}
        \ind_{X(\T)[-1]}X(\delta)&{}=\ind_{X(\T)}X(\delta)[1]\\&=\ind_{X(\T)}X(\rho(\delta)){}\\
        &=-b_{\T}(e(\rho(\delta)))\\&=-b_{\T}(e^{op}(\delta))
    \end{align*}
    where the second equality holds by Theorem~\ref{thm:QZ}~(1), the third one is \eqref{eq:indb1}, and the last one is due to Remark~\ref{rmk:erho}.
\end{proof}

\begin{corollary}\label{cor:iff}
    Let $\T$ be a tagged triangulation of $\surf$ and $\R\subseteq\T$. Then for any $\delta\in\TA(\surf)$, $b_{\gamma,\T}(e(\delta))=0$ for all $\gamma\in\T\setminus\R$ if and only if $X(\delta)\in X(\R)\ast X(\R)[1]$.
\end{corollary}

\begin{proof}
    By Proposition~\ref{prop:index}, $b_{\gamma,\T}(e(\delta))=0$ if and only if $[\ind_{X(\T)}X(\delta):X(\gamma)]=0$. Then by Remark~\ref{rmk:cto}, we get this assertion.
\end{proof}

\begin{corollary}\label{cor:shear}
    Let $\T$ be a tagged triangulation of $\surf$. Then for any $\gamma,\delta\in\T$, we have
    $$b_{\gamma,\T}(e(\delta))=\begin{cases}
    -1&\text{if $\gamma=\delta$,}\\
    0&\text{otherwise,}
    \end{cases}$$
    and
    $$b_{\gamma,\T}(e^{op}(\delta))=\begin{cases}
    1&\text{if $\gamma=\delta$,}\\
    0&\text{otherwise.}
    \end{cases}$$
\end{corollary}

\begin{proof}
    By Theorem~\ref{thm:QZ} and the definition of index, we have $$[\ind_{X(\T)}X(\delta):X(\gamma)]=\begin{cases}1&\text{if $\gamma=\delta$,}\\0&\text{otherwise,}\end{cases}$$
    and
    $$[\ind_{X(\T)[-1]}X(\delta):X(\gamma)[-1]]=\begin{cases}-1&\text{if $\gamma=\delta$,}\\0&\text{otherwise.}\end{cases}$$
    Then by Proposition~\ref{prop:index}, the required formulas follow.
\end{proof}

\section{A geometric model of surface rigid algebras}\label{sec:geo mod}
In this section, we give a geometric model for the module category of the endomorphism algebra of a rigid object in the cluster category of a punctured marked surface. We also show that any skew-gentle algebra is contained in this case.
\begin{definition}\label{def:surf alg}
    A finite dimensional algebra over $\k$ is called a \emph{surface rigid algebra} if it is isomorphic to the endomorphism algebra $\Lambda_R:=\End_{\C(\surf)}R$ of a rigid object $R$ in the cluster category $\C(\surf)$ of a punctured marked surface $\surf$.
\end{definition}

See \cite{RR} for an unpunctured version of the above notion.

By Theorem~\ref{thm:QZ}~(3), each surface rigid algebra can be realized as the endomorphism algebra $\Lambda_\R:=\Lambda_{X(\R)}$ of $X(\R)$ for a partial tagged triangulation $\R$ of a punctured marked surface $\surf=(S,\MM,\PP)$. 

\subsection{Standard arcs and dissections}
For any tagged arc $\delta\in\TA(\surf)$ whose both endpoints are in $\PP$, the tagged arc $\rho(\delta)=\rho^{-1}(\delta)$ is said to be \emph{adjoint to} $\delta$.

For any partial ideal triangulation $\R$ of $\surf$ and any curve $\delta$ in a minimal position with (the arcs in) $\R$, an \emph{arc segment} of $\delta$ (with respect to $\R$) is a segment of $\delta$ between two neighboring intersections between $\delta$ and $\R\cup B$, where recall that $B$ is the set of boundary segments of $S$. An arc segment of $\delta$ is called an \emph{end} arc segment if it has an endpoint in $\MM\cup\PP$.

\begin{definition}\label{def:standard ideal}%
    Let $\R$ be a partial ideal triangulation of $\surf$. A tagged arc $\delta\in\TA(\surf)$ is called \emph{$\R$-standard} if and only if one of the following holds.
    \begin{enumerate}
        \item $\delta\in\R^\times$ or $\delta$ is adjoint to some arc in $\R^\times$.
        \item $\delta$ itself is an arc segment and cuts out an angle $\theta$ between two arcs in $\R\cup B$, such that the tagging of $\delta$ has the form as shown in the first picture of Figure~\ref{fig:std}, where the left endpoint of $\delta$ is a puncture with tagging $-1$, and the right endpoint of $\delta$ is either a marked point or a puncture with tagging $1$.
        \item $\delta$ is divided by $\R$ into at least two arc segments, satisfying that each arc segment $\eta$ cuts out an angle $\theta$ between two arcs in $\R\cup B$ (see the last three pictures of Figure~\ref{fig:std}), and in case $\eta$ having an endpoint $\jiaodian$ in $\PP\cup\MM$, the edge of $\theta$ with vertex $\jiaodian$ is the first arc in $\R\cup B$ next to $\eta$ in the anticlockwise order around $\jiaodian$ if $\jiaodian\in\PP$ and the tagging of $\eta$ at this end is $-1$ (see the third picture of Figure~\ref{fig:std}), or in the clockwise order otherwise (see the last picture of Figure~\ref{fig:std}).
	\end{enumerate}
	\begin{figure}[htpb]
		\begin{tikzpicture}[xscale=1.5,yscale=1.5]
			\draw[red,thick,bend right=10](0,1)to(-.8,.3)node[black]{$\bullet$};
			\draw[red,thick,bend left=10](0,1)to(.8,.3);
			\draw[blue,thick,bend right](-.8,.3)to(.8,.3)node[black]{$\bullet$};
			\draw[orange,bend right](-.3,.8)tonode[below]{$\theta$}(.3,.8);
			\draw[blue,thick](0,0)[below]node{$\delta$};
			\draw[blue](-.68,.24)node[rotate=15]{$+$};
			\draw(0,1)node{$\bullet$};
		\end{tikzpicture}\quad
		\begin{tikzpicture}[xscale=1.5,yscale=1.5]
			\draw[red,thick,bend right=10](0,1)to(-1,0);
			\draw[red,thick,bend left=10](0,1)to(1,0);
			\draw[blue,thick,bend right](-.8,.3)to(.8,.3);
			\draw[orange,bend right](-.3,.8)tonode[below]{$\theta$}(.3,.8);
			\draw[blue,thick](0,0)[below]node{$\eta$};
			\draw(0,1)node{$\bullet$};
		\end{tikzpicture}\quad
		\begin{tikzpicture}[xscale=1.5,yscale=1.5]
			\draw[red,thick,bend right=10](0,1)to(-.8,.3)node[black]{$\bullet$};
			\draw[blue](-.68,.24)node[rotate=15]{$+$};
			\draw[red,thick,bend left=10](0,1)to(1,0);
			\draw[blue,thick,bend right](-.8,.3)to(.8,.3);
			\draw[orange,bend right](-.3,.8)tonode[below]{$\theta$}(.3,.8);
			\draw[blue,thick](0,0)[below]node{$\eta$};
			\draw(0,1)node{$\bullet$}(-.9,.5)node[black]{$\jiaodian$};
		\end{tikzpicture}\quad
		\begin{tikzpicture}[xscale=1.5,yscale=1.5]
			\draw[red,thick,bend right=10](0,1)to(-1,0);
			\draw[red,thick,bend left=10](0,1)to(.8,.3)node[black]{$\bullet$};
			\draw[blue,thick,bend right](-.8,.3)to(.8,.3);
			\draw[orange,bend right](-.3,.8)tonode[below]{$\theta$}(.3,.8);
			\draw[blue,thick](0,0)[below]node{$\eta$};
			\draw(0,1)node{$\bullet$}(.9,.5)node[black]{$\jiaodian$};
		\end{tikzpicture}
		\caption{Arc segments of standard tagged arcs}\label{fig:std}
	\end{figure}
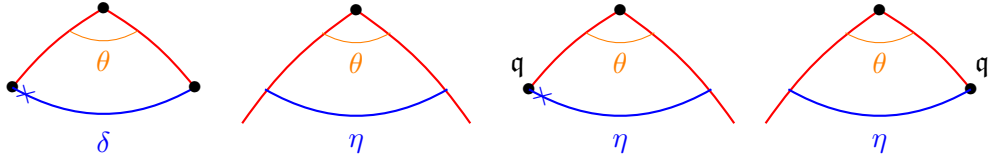
	
	Dually, a tagged arc $\delta\in\TA(\surf)$ is called \emph{$\R$-co-standard} if and only if one of the following holds.
	\begin{enumerate}
		\item  $\delta\in\R^\times$ or $\delta$ is adjoint to some arc in $\R^\times$.
		\item $\delta$ itself is an arc segment and cuts out an angle $\theta$ between two arcs in $\R\cup B$, such that the tagging of $\delta$ has the form as shown in the first picture of Figure~\ref{fig:cstd}, where the right endpoint of $\delta$ is a puncture with tagging $-1$, and the left endpoint of $\delta$ is either a marked point or a puncture with tagging $1$.
		\item $\delta$ is divided by $\R$ into at least two arc segments, satisfying that each arc segment $\eta$ cuts out an angle $\theta$ between two arcs in $\R\cup B$, and in case $\eta$ having an endpoint $\jiaodian\in\PP\cup\MM$, the edge of $\theta$ which is incident to the vertex $\jiaodian$ is the first arc in $\R\cup B$ next to $\eta$ in the clockwise order around $\jiaodian$ if $\jiaodian\in\PP$ and the tagging of $\eta$ at this end is $-1$, or in the anticlockwise order otherwise. See the last three pictures of Figure~\ref{fig:cstd}.
	\end{enumerate}
	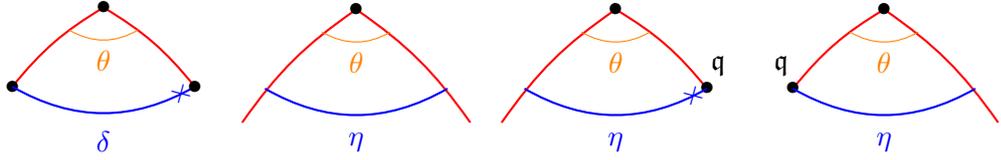
\begin{figure}[htpb]
		\begin{tikzpicture}[xscale=1.5,yscale=1.5]
			\draw[red,thick,bend right=10](0,1)to(-.8,.3)node[black]{$\bullet$};
			\draw[red,thick,bend left=10](0,1)to(.8,.3);
			\draw[blue,thick,bend right](-.8,.3)to(.8,.3)node[black]{$\bullet$};
			\draw[orange,bend right](-.3,.8)tonode[below]{$\theta$}(.3,.8);
			\draw[blue,thick](0,0)[below]node{$\delta$};
			\draw[blue](.68,.24)node[rotate=-15]{$+$};
			\draw(0,1)node{$\bullet$};
		\end{tikzpicture}\quad
		\begin{tikzpicture}[xscale=1.5,yscale=1.5]
			\draw[red,thick,bend right=10](0,1)to(-1,0);
			\draw[red,thick,bend left=10](0,1)to(1,0);
			\draw[blue,thick,bend right](-.8,.3)to(.8,.3);
			\draw[orange,bend right](-.3,.8)tonode[below]{$\theta$}(.3,.8);
			\draw[blue,thick](0,0)[below]node{$\eta$};
			\draw(0,1)node{$\bullet$};
		\end{tikzpicture}\quad
		\begin{tikzpicture}[xscale=1.5,yscale=1.5]
			\draw[red,thick,bend right=10](0,1)to(-1,0);
			\draw[red,thick,bend left=10](0,1)to(.8,.3)node[black]{$\bullet$};
			\draw[blue](.68,.24)node[rotate=-15]{$+$};
			\draw[blue,thick,bend right](-.8,.3)to(.8,.3);
			\draw[orange,bend right](-.3,.8)tonode[below]{$\theta$}(.3,.8);
			\draw[blue,thick](0,0)[below]node{$\eta$};
			\draw(0,1)node{$\bullet$}(.9,.5)node[black]{$\jiaodian$};
		\end{tikzpicture}\quad
		\begin{tikzpicture}[xscale=1.5,yscale=1.5]
			\draw[red,thick,bend right=10](0,1)to(-.8,.3)node[black]{$\bullet$};
			\draw[red,thick,bend left=10](0,1)to(1,0);
			\draw[blue,thick,bend right](-.8,.3)to(.8,.3);
			\draw[orange,bend right](-.3,.8)tonode[below]{$\theta$}(.3,.8);
			\draw[blue,thick](0,0)[below]node{$\eta$};
			\draw(0,1)node{$\bullet$}(-.9,.5)node[black]{$\jiaodian$};
		\end{tikzpicture}
		\caption{Arc segments of co-standard tagged arcs}\label{fig:cstd}
	\end{figure}

    Let $\R$ be a partial tagged triangulation of $\surf$. A tagged arc $\delta\in\TA(\surf)$ is called \emph{$\R$-standard} (resp. \emph{$\R$-co-standard}) if $\delta^\R$ is $\R^\circ$-standard (resp. $\R^\circ$-co-standard). Denote by $\stTA{\R}$ (resp. $\cstTA{\R}$) the set of $\R$-standard (resp. $\R$-co-standard) tagged arcs.
\end{definition}

\begin{remark}\label{rmk:std}
    In Definition~\ref{def:standard ideal}, since $\delta$ is a tagged arc, the two edges of $\theta$ are not from the same arc in $\R$, unless in the second pictures of Figures~\ref{fig:std} and \ref{fig:cstd}, where they may be the different ends of an arc in $\R$.
\end{remark}


\begin{figure}[htpb]
    \begin{tikzpicture}[xscale=1.5,yscale=1.6]
        \draw[ultra thick,bend left=10](-1.05,.4)to(-1.05,-.4) ;
        \draw[ultra thick,bend left=10](1.05,-.4)to(1.05,.4);
        \draw[red,thick,bend right=10](-1,0)to(1,0);
        
		\draw[blue,thick,bend right=10](1,0)to(-1,0)node[black]{$\bullet$};
		\draw[blue,thick,dashed](-1.04,-.3)to[out=0,in=180](1.04,.3);
		\draw(-1,0)node{$\bullet$}(-.15,.25)node[blue]{$\delta$}(-.15,-.3)node[blue]{$e(\delta)$};
		\draw(1,0)node{$\bullet$};
		
		\draw[red,thick,dashed](-.6,-.4)to(-1,0)to(-.6,.4) (.6,-.4)to(1,0)to(.6,.4);
	\end{tikzpicture}\qquad
	\begin{tikzpicture}[xscale=1.5,yscale=1.6]
		\draw[ultra thick,bend right=10](-1.05,-.4)to(-1.05,.4) ;
		\draw[red,thick,bend right=10](-1,0)to(1,0);
		\draw(-.15,.25)node[blue]{$\delta$}(.7,.4)node[blue]{$e(\delta)$}(-1,0)node{$\bullet$}(1,0)node[black]{$\bullet$};
		\draw[blue,thick,bend left=10](-1,0)to(1,0);
		
		\draw[blue,thick,dashed](-1.04,-.3)to[out=0,in=180](1,.25)to[out=0,in=90](1.2,0)to[out=-90,in=0](1,-.15)to[out=180,in=-90](.85,0)to[out=90,in=180](1,.15)to[out=0,in=90](1.13,0);
		\draw[red,thick,dashed](-.6,-.4)to(-1,0)to(-.6,.4) (1.4,-.4)to(1,0)to(1.4,.4);
	\end{tikzpicture}\qquad
	\begin{tikzpicture}[xscale=1.5,yscale=1.6]
        \draw[red,thick,bend right=10](-1,0)to(1,0);
        \draw(-.15,.25)node[blue]{$\delta$}(.7,.4)node[blue]{$e(\delta)$}(-1,0)node{$\bullet$}(1,0)node[black]{$\bullet$}(-1,-.4)node{};
        \draw[blue,thick,bend left=10](-1,0)to(1,0);
        \draw[red,thick,dashed](-1.4,-.4)to(-1,0)to(-1.4,.4) (1.4,-.4)to(1,0)to(1.4,.4);
        
		\draw[blue,thick,dashed](-1.1,0)to[out=-90,in=180](-1,-.15)to[out=0,in=-90](-.8,0)to[out=90,in=0](-1,.2)to[out=180,in=90](-1.2,0)to[out=-90,in=180](-1,-.25)to[out=0,in=180](1,.25)to[out=0,in=90](1.2,0)to[out=-90,in=0](1,-.15)to[out=180,in=-90](.85,0)to[out=90,in=180](1,.15)to[out=0,in=90](1.1,0);
	\end{tikzpicture}

	\begin{tikzpicture}[xscale=2,yscale=1.6]
		\draw[ultra thick,bend right=10](-.4,1)to(.4,1);
		\draw(0,.95)node{$\bullet$};
		\draw[red,thick](0,.95)to[out=-60,in=90](.4,0)to[out=-90,in=0](0,-.5)to[out=180,in=-90](-.4,0)to[out=90,in=-120](0,.95);
		\draw[red,thick,bend right=10](0,.95)to(0,0)node[black]{$\bullet$};
		\draw[red,thick,dashed,bend right=5](0,.95)to(-.5,.4);
		\draw[red,thick,dashed,bend right=5](0,.95)to(-.5,.7);
		
		\draw[blue,thick,bend left=10](0,.95)to(0,0);
		\draw[blue](.03,.3)node{$\times$}(.15,.45)node{$\delta$}(0,-.34)node[blue]{$e(\delta)$};
		\draw[blue,thick,dashed](-.3,1)to[out=-90,in=180](0,-.2)to[out=0,in=-90](.15,0)to[out=90,in=0](0,.15)to[out=180,in=90](-.13,0)to[out=-90,in=180](0,-.13)to[out=0,in=-90](.1,0);
		\draw(0,1.2)node{};
	\end{tikzpicture}\qquad
	\begin{tikzpicture}[xscale=2,yscale=1.6]
		\draw(0,.95)node{$\bullet$};
		\draw[red,thick](0,.95)to[out=-60,in=90](.4,0)to[out=-90,in=0](0,-.5)to[out=180,in=-90](-.4,0)to[out=90,in=-120](0,.95);
		\draw[red,thick,bend right=10](0,.95)to(0,0)node[black]{$\bullet$};
		\draw[red,thick,dashed](-.4,1.1)to(0,.95)to(.4,1.1);
		
		\draw[blue,thick,bend left=10](0,.95)to(0,0);
		\draw[blue](.03,.3)node{$\times$}(.15,.45)node{$\delta$}(0,-.34)node[blue]{$e(\delta)$};
		\draw[blue,thick,dashed](0,1.05)to[out=180,in=90](-.1,.95)to[out=-90,in=180](0,.8)to[out=0,in=-90](.15,1)to[out=90,in=0](0,1.15)to[out=180,in=90](-.2,1)to[out=-90,in=90](-.2,0)to[out=-90,in=180](0,-.2)to[out=0,in=-90](.15,0)to[out=90,in=0](0,.15)to[out=180,in=90](-.13,0)to[out=-90,in=180](0,-.13)to[out=0,in=-90](.1,0);
		\draw(0,1.2)node{};
	\end{tikzpicture}\qquad
	\begin{tikzpicture}[xscale=2,yscale=1.6]
		\draw[ultra thick,bend right=10](-.4,1)to(.4,1);
		\draw(0,.95)node{$\bullet$};
		\draw[red,thick](0,.95)to[out=-60,in=90](.4,0)to[out=-90,in=0](0,-.5)to[out=180,in=-90](-.4,0)to[out=90,in=-120](0,.95);
		\draw[red,thick,bend right=10](0,.95)to(0,0)node[black]{$\bullet$};
		\draw[red,thick,dashed,bend right=5](0,.95)to(-.5,.4);
		\draw[red,thick,dashed,bend right=5](0,.95)to(-.5,.7);
		
		\draw[blue,thick,bend left=10](0,.95)to(0,0);
		\draw[blue](.15,.45)node{$\delta$}(0,-.34)node[blue]{$e(\delta)$};
		\draw[blue,thick,dashed](-.3,1)to[out=-90,in=90](.15,0)to[out=-90,in=0](0,-.2)to[out=180,in=-90](-.15,0)to[out=90,in=180](0,.15)to[out=0,in=90](.1,0)to[out=-90,in=0](0,-.13)to[out=180,in=-90](-.1,0);        
		\draw(0,1.2)node{};
	\end{tikzpicture}\qquad
	\begin{tikzpicture}[xscale=2,yscale=1.6]
		\draw(0,.95)node{$\bullet$};
		\draw[red,thick](0,.95)to[out=-60,in=90](.4,0)to[out=-90,in=0](0,-.5)to[out=180,in=-90](-.4,0)to[out=90,in=-120](0,.95);
		\draw[red,thick,bend right=10](0,.95)to(0,0)node[black]{$\bullet$};
		\draw[red,thick,dashed](-.4,1.1)to(0,.95)to(.4,1.1);
		
		\draw[blue,thick,bend left=10](0,.95)to(0,0);
		\draw[blue](.15,.45)node{$\delta$}(0,-.34)node[blue]{$e(\delta)$};
		\draw[blue,thick,dashed](0,1.05)to[out=180,in=90](-.1,.95)to[out=-90,in=180](0,.8)to[out=0,in=-90](.15,1)to[out=90,in=0](0,1.15)to[out=180,in=90](-.2,1)to[out=-90,in=90](.15,0)to[out=-90,in=0](0,-.2)to[out=180,in=-90](-.15,0)to[out=90,in=180](0,.15)to[out=0,in=90](.1,0)to[out=-90,in=0](0,-.13)to[out=180,in=-90](-.1,0);  
		\draw(0,1.2)node{};
	\end{tikzpicture}
    \\
    \begin{tikzpicture}[xscale=1.5,yscale=1.6]
    	\draw[red,thick,bend right=10](-1,0)to(1,0);
    	\draw(.2,.25)node[blue]{$\delta$}(-.5,.4)node[blue]{$e(\delta)$}(-1,0)node{$\bullet$}(1,0)node[black]{$\bullet$};
    	\draw[blue,thick,bend left=10](-1,0)to(1,0);
    	\draw[blue](-.7,.05)node[rotate=10]{$\times$}(.7,.05)node[rotate=-10]{$\times$}(0,-.6)node{};
    	\draw[red,thick,dashed](-1.4,-.4)to(-1,0)to(-1.4,.4) (1.4,-.4)to(1,0)to(1.4,.4);
    	
    	\draw[blue,thick,dashed](-1.1,0)to[out=90,in=180](-1,.15)to[out=0,in=90](-.85,0)to[out=-90,in=0](-1,-.2)to[out=180,in=-90](-1.2,0)to[out=90,in=180](-1,.25)to[out=0,in=180](1,-.25)to[out=0,in=-90](1.2,0)to[out=90,in=0](1,.2)to[out=180,in=90](.85,0)to[out=-90,in=180](1,-.15)to[out=0,in=-90](1.1,0);
    \end{tikzpicture}\qquad
    \begin{tikzpicture}[xscale=2,yscale=1.6]
    	\draw(0,.95)node{$\bullet$};
    	\draw[red,thick](0,.95)to[out=-60,in=90](.4,0)to[out=-90,in=0](0,-.5)to[out=180,in=-90](-.4,0)to[out=90,in=-120](0,.95);
    	\draw[red,thick,bend left=10](0,.95)to(0,0)node[black]{$\bullet$};
    	\draw[red,thick,dashed](-.4,1.1)to(0,.95)to(.4,1.1);
    	
    	\draw[blue,thick,bend right=10](0,.95)to(0,0);
    	\draw[blue](-.06,.6)node{$\times$}(-.17,.45)node{$\delta$}(0,-.34)node[blue]{$e(\delta)$};
    	\draw[blue,thick,dashed](0,1.1)to[out=0,in=90](.1,.95)to[out=-90,in=0](0,.8)to[out=180,in=-90](-.15,1)to[out=90,in=180](0,1.15)to[out=0,in=90](.15,1)to[out=-90,in=90](.15,0.05)to[out=-90,in=0](0,-.2)to[out=180,in=-90](-.15,0)to[out=90,in=180](0,.15)to[out=0,in=90](.1,0)to[out=-90,in=0](0,-.13)to[out=180,in=-90](-.1,0);
    	\draw(0,1.2)node{};
    \end{tikzpicture}\qquad
	\begin{tikzpicture}[xscale=2,yscale=1.6]
		\draw(0,.95)node{$\bullet$};
		\draw[red,thick](0,.95)to[out=-60,in=90](.4,0)to[out=-90,in=0](0,-.5)to[out=180,in=-90](-.4,0)to[out=90,in=-120](0,.95);
		\draw[red,thick,bend left=10](0,.95)to(0,0)node[black]{$\bullet$};
		\draw[red,thick,dashed](-.4,1.1)to(0,.95)to(.4,1.1);
		
		\draw[blue,thick,bend right=10](0,.95)to(0,0);
		\draw[blue](-.06,.6)node{$\times$}(-.05,.25)node{$\times$}(-.17,.45)node{$\delta$}(0,-.34)node[blue]{$e(\delta)$};
		\draw[blue,thick,dashed](0,1.1)to[out=0,in=90](.1,.95)to[out=-90,in=0](0,.8)to[out=180,in=-90](-.15,1)to[out=90,in=180](0,1.15)to[out=0,in=90](.15,1)to[out=-90,in=90](-.15,0.05)to[out=-90,in=180](0,-.2)to[out=0,in=-90](.15,0)to[out=90,in=0](0,.15)to[out=180,in=90](-.1,0)to[out=-90,in=180](0,-.13)to[out=0,in=-90](.1,0);
		\draw(0,1.2)node{};
	\end{tikzpicture}
    \\
	\begin{tikzpicture}[xscale=1.5,yscale=1.6]
		\draw[red,thick,bend right=10](0,1)node[black]{$\bullet$}to(-.8,.15)node[black]{$\bullet$};
		\draw[red,thick,bend left=10](0,1)to(.8,.15)node[black]{$\bullet$};
		\draw[blue,thick,dashed](-.8,.05)to[out=180,in=-90](-.9,.15)to[out=90,in=180](-.8,.3)to[out=0,in=90](-.6,.15)to[out=-90,in=0](-.8,-.05)to[out=180,in=-90](-1,.15)to[out=90,in=180](-.8,.35)to[out=0,in=180](.8,.35)to[out=0,in=90](1,.15)to[out=-90,in=0](.8,0)to[out=180,in=-90](.65,.15)to[out=90,in=180](.8,.3)to[out=0,in=90](.9,.15)to[out=-90,in=0](.8,.05);
		\draw(0,-.15)node[blue]{$\delta$}(-.5,.09)node[blue,rotate=-8]{$\times$};
		\draw[blue,thick,bend right=10](-.8,.15)to(.8,.15);
		\draw[red,thick,dashed](-1,.5)to(-.8,.15)to(-1,-.2) (1,.5)to(.8,.15)to(1,-.2);
		
		\draw(.8,-.1)node{$\jiaodian$}(.65,.35)node{$\bullet$}(.78,.5)node{$q$}(.4,.8)node[red]{$\gamma$}(0,.52)node[blue]{$e(\delta)$};
	\end{tikzpicture}\qquad
	\begin{tikzpicture}[xscale=1.5,yscale=1.6]
	    \draw[ultra thick,bend right=15]plot(.88,.5)to(.88,0);
		\draw[red,thick,bend right=10](0,1)node[black]{$\bullet$}to(-.8,.15)node[black]{$\bullet$};
		\draw[red,thick,bend left=10](0,1)to(.8,.15)node[black]{$\bullet$};
		\draw[blue,thick,dashed](-.8,.05)to[out=180,in=-90](-.9,.15)to[out=90,in=180](-.8,.3)to[out=0,in=90](-.6,.15)to[out=-90,in=0](-.8,-.05)to[out=180,in=-90](-1,.15)to[out=90,in=180](-.8,.35)to[out=0,in=180](.85,.35);
		\draw(0,-.15)node[blue]{$\delta$}(-.5,.09)node[blue,rotate=-8]{$\times$};
		\draw[blue,thick,bend right=10](-.8,.15)to(.8,.15);
		\draw[red,thick,dashed](-1,.5)to(-.8,.15)to(-1,-.2) (.8,.15)to(.75,.7);
		
		\draw(.7,-.05)node{$\jiaodian$}(.65,.35)node{$\bullet$}(.53,.22)node{$q$}(.4,.8)node[red]{$\gamma$}(0,.52)node[blue]{$e(\delta)$};
	\end{tikzpicture}\qquad
	\begin{tikzpicture}[xscale=1.5,yscale=1.6]
		\draw[red,thick,bend right=10](0,1)to(-.9,0);
		\draw[red,thick,bend left=10](0,1)to(.9,0);
		\draw[blue,thick,dashed,bend right=10](-.7,.3)to(.7,.3);
		\draw(0,-.15)node[blue]{$\delta$};
		\draw(0,1)node{$\bullet$};
		\draw[blue,thick,bend right=10](-.8,.15)to(.8,.15);
		
		\draw(.68,0)node{$\jiaodian$}(.7,.3)node{$\bullet$}(.78,.15)node{$\bullet$}(.8,.4)node{$q$}(.4,.8)node[red]{$\gamma$}(0,.38)node[blue]{$e(\delta)$};
	\end{tikzpicture}
    \\
	\begin{tikzpicture}[xscale=-1.5,yscale=1.5]
        \draw[ultra thick,bend left=10]plot(-.5,.5)to(-.5,-.2);
        \draw(-.4,-.2)node{$\jiaodian$}(-.2,.3)node{$\bullet$}(-.08,.1)node{$q$};
        \draw[red,thick,dashed,bend right=10](-.45,0)to(-.35,.7);
        
        \draw[red,thick,bend left=10](-.45,0)node[black]{$\bullet$}to(.5,1)node[black]{$\bullet$};
        \draw[red,thick,bend left=10](.5,1)to(1,-.2);
        \draw[blue,thick,bend left=10](1,0)to(-.45,0);
        \draw[blue,thick,dashed,bend left=10](.9,.2)to(-.5,.4);
        \draw(.3,-.25)node[blue]{$\delta$}(0,.75)node[red]{$\gamma$}(.35,.45)node[blue]{$e(\delta)$};
    \end{tikzpicture}\qquad
	\begin{tikzpicture}[xscale=1.5,yscale=1.7]
		\draw[red,thick,bend right=10](0,1)to(-1,0);
		\draw[red,thick,bend left=10](0,1)to(.5,.3)node[black]{$\bullet$};
		\draw[blue,thick,bend right=10](-.8,.3)to(.5,.3);
		\draw(0,1)node{$\bullet$};
		\draw(-.2,.1)node[blue]{$\delta$};
		\draw[red,thick,dashed](1,1)to(.5,.3)to(1,-.2);
		\draw[red,thick,dashed](.5,.3)to(0.2,-.2);
		\draw[blue,thick,dashed](-.6,.45)to[out=10,in=180](.5,.6)to[out=0,in=90](.8,.3)to[out=-90,in=0](.5,0)to[out=180,in=-90](.2,.3)to[out=90,in=-180](.5,.55)to[out=0,in=90](.75,.3)to[out=-90,in=0](.5,.05)to[out=180,in=-90](.25,.3);
		\draw[blue,thick,dotted](.25,.3)to[out=90,in=180](.5,.5);
		
		\draw(.63,.3)node{$\jiaodian$}(.32,.6)node{$\bullet$}(.4,.75)node{$q$}(.25,.85)node[red]{$\gamma$}(1.1,.3)node[blue]{$e(\delta)$};
	\end{tikzpicture}\quad
	\begin{tikzpicture}[xscale=1.5,yscale=1.7]
		\draw[red,thick,bend right=10](0,1)to(-.5,.3)node[black]{$\bullet$};
		\draw[blue](-.36,.28)node{$\times$};
		\draw[red,thick,bend left=10](0,1)to(1,0);
		\draw[blue,thick,bend right=10](-.5,.3)to(.8,.3);
		\draw(0,1)node{$\bullet$};
		\draw(.2,.1)node[blue]{$\delta$};
		\draw[red,thick,dashed](-1,1)to(-.5,.3)to(-1,-.2);
		\draw[red,thick,dashed](-.5,.3)to(-.2,-.2);
		\draw[blue,thick,dashed](.6,.45)to[out=170,in=0](-.5,.6)to[out=180,in=90](-.8,.3)to[out=-90,in=180](-.5,0)to[out=0,in=-90](-.2,.3)to[out=90,in=0](-.5,.55)to[out=180,in=90](-.75,.3)to[out=-90,in=180](-.5,.05)to[out=0,in=-90](-.25,.3);
		\draw[blue,thick,dotted] (-.25,.3)to[out=90,in=0](-.5,.5);
		
		\draw(-.63,.3)node{$\jiaodian$}(-.32,.6)node{$\bullet$}(-.4,.78)node{$q$}(.4,.8)node[red]{$\gamma$}(-1.1,.3)node[blue]{$e(\delta)$};
	\end{tikzpicture}
	\caption{The elementary laminate of a standard tagged arc}
	\label{fig:lam of std}
\end{figure}
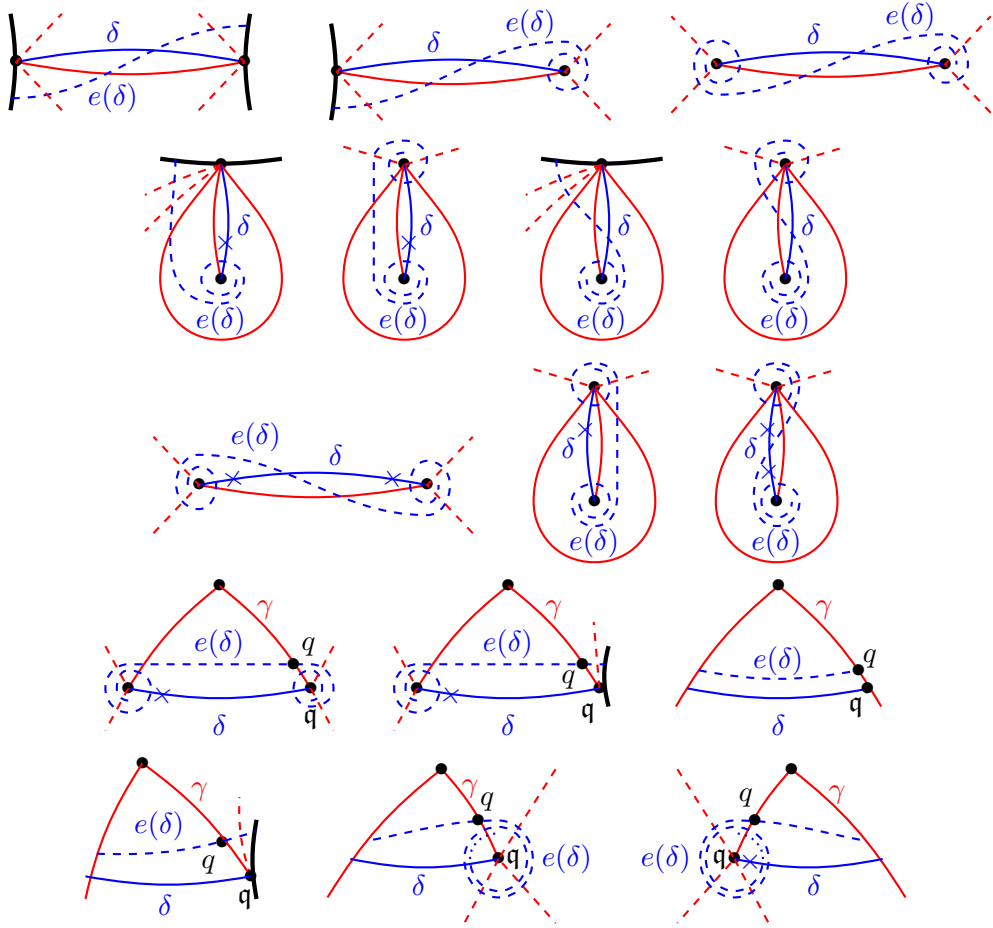

In the following, we give sufficient and necessary conditions on tagged arcs to be standard/co-standard.

\begin{proposition}\label{lem:=}
    Let $\R$ be a partial tagged triangulation of $\surf$ and $\delta$ a tagged arc. Then the following are equivalent.
    \begin{enumerate}
        \item[(1)] $\delta$ is $\R$-standard (resp. $\R$-co-standard).
        \item[(2)] $e(\delta)$ (resp. $e^{op}(\delta)$) shears $\R$ (see Definition~\ref{def:shear}).
        \item[(3)] For some/any tagged triangulation $\T$ such that $\R\subset\T$, we have $b_{\gamma,\T}(e(\delta))=0$ (resp. $b_{\gamma,\T}(e^{op}(\delta))=0$) for any $\gamma\in\T\setminus\R$.
        \item[(4)] $X(\delta)\in X(\R)\ast X(\R)[1]$.
    \end{enumerate}
\end{proposition}

\begin{proof}
    We only show the equivalences between statements for the $\R$-standard case, since the $\R$-co-standard case can be shown dually. By Corollary~\ref{cor:iff}, (4) is equivalent to (3). By definition, (1) is equivalent to that $\delta^\R$ is $\R^\circ$-standard, (2) is equivalent to that $e(\delta^\R)$ shears $\R^\circ$. By Remark~\ref{rmk:good completion}, (3) is equivalent to that for some/any tagged triangulation $\T$ satisfying the condition in Remark~\ref{rmk:good completion} and $\R\subset\T$, we have $b_{\gamma^\circ,\T^\circ}(e(\delta^\R))=0$ for any $\gamma\in\T\setminus\R$. So we may assume that $\R$ is a partial ideal triangulation.  

    ``(1)$\Rightarrow$(2)": We list all the possible cases of arc segments of an $\R$-standard tagged arc $\delta$ and its elementary laminate $e(\delta)$ in Figure~\ref{fig:lam of std}, where the pictures in the first row are for $\delta\in\R^\times$ and $\delta^\circ$ not a side of a self-folded triangle, those in the second row are for $\delta\in\R^\times$ and $\delta^\circ$ a side of a self-folded triangle, those in the third row are for $\delta$ adjoint to $\R^\times$ but not in $\R^\times$, those in the fourth row are for the arc segments in the first two pictures of Figure~\ref{fig:std}, and those in the last row are for the arc segments in the last two pictures of Figure~\ref{fig:std}. In each case, any segment of $e(\delta)$ cuts out an angle between two arcs in $\R\cup B$. Hence (2) holds.
    
    ``(2)$\Rightarrow$(3)": By (2), $e(\delta)$ is divided by $\R$ into segments, each of which cuts out an angle between two arcs in $\R\cup B$. Let $\T$ be an ideal triangulation containing $\R$, $\gamma\in\T\setminus\R$ and $q$ an intersection between $\gamma$ and a segment $\eta$ of $e(\delta)$. Then $\gamma$ divides the angle cut by $\eta$ into two parts. Hence $b_{q,\gamma,\T}(e(\delta))=0$ (cf. the third picture in Figure~\ref{fig:shear}). Thus, we have $b_{\gamma,\T}(e(\delta))=0$.
	
    ``(3)$\Rightarrow$(1)": If $\delta$ has no interior intersections with any arc in $\R$, then we have the following three subcases.
    \begin{enumerate}
        \item[(i)] $\Int(\delta,\gamma^\times)=0$ for any $\gamma\in\R$. Then there is a tagged triangulation $\T_1$ containing $\R^\times$ and $\delta$. By Corollary~\ref{cor:shear}, we have 
		$b_{\delta,\T_1}(e(\delta))=-1$. Hence by (3), we have $\delta\in\R^\times$, which implies $\delta$ is $\R$-standard.
        \item[(ii)] Exactly one endpoint of $\delta$, say $\delta(0)$, is a tagged intersection with some arc in $\R^\times$. Then $\delta(0)\neq\delta(1)$. If $\kappa_\delta(0)=1$, then there is an arc $\epsilon\in\R$ such that $\epsilon^\times(0)=\delta(0)$ and $\kappa_{\epsilon^\times}(0)=-1$. Then $\epsilon$ is a loop enclosing $\delta(0)$. Note that $\delta$ has no interior intersections with any arc in $\R$. So $\delta$ is in the once-punctured monogon enclosed by $\epsilon$. Thus, $\delta$ is homotopic to $\epsilon$. Then by the definition of tagged intersections, $\delta(1)$ is also a tagged intersection between $\delta$ and $\epsilon^\times$, a contradiction with the setting of subsection (ii). Hence $\kappa_\delta(0)=-1$. Let $\delta_0$ be the tagged arc homotopic to $\delta$ and whose tagging at each tagged end is $1$. Since $\delta_0$ has zero (if $\kappa_\delta(1)=1$) or two (if $\kappa_\delta(1)=-1$) tagged intersections with $\delta$, there are arcs in $\R$ which have $\delta(0)$ as an endpoint and is not homotopic to $\delta_0$. Take $\gamma$ to be the first arc among them next to $\delta_0$ in the anticlockwise order around $\delta(0)$. Let $\gamma'$ be the arc or the boundary segment such that $\Delta=\{\gamma',\delta_0,\gamma\}$ is a (possibly self-folded) triangle with the angle from $\delta_0$ to $\gamma$ in the anticlockwise order as an inner angle. Since $\delta_0$ has no interior intersections with arcs in $\R$, neither does $\gamma'$. Let $\T_2$ be an ideal triangulation containing $\R$, $\delta_0$ and $\gamma'$. If $\Delta$ is self-folded, since $\gamma$ is not homotopic to $\delta$ and $\delta$ is a tagged arc (which implies that $\delta$ is not the non-folded side of a self-fold triangle), we have that $\gamma'$ and $\delta_0$ are folded, and $\gamma$ is the non-folded side (cf. the last picture in Figure~\ref{fig:one-int}). So $\delta(1)$ is a tagged intersection with either $\gamma^\times$ (if $\kappa_\delta(1)=1$) or ${\gamma'}^\times=\delta_0$ (if $\kappa_\delta(1)=-1$). But since $\gamma\in\R$, by Definition~\ref{def:ideal tri}, we have $\gamma'\in\R$, a contradiction with the setting of subsection (ii). Hence the triangle $\Delta$ is not self-folded. In particular, none of $\gamma$, $\delta_0$ and $\gamma'$ is the folded side of a self-folded triangle in $\T_2$. 
        \begin{enumerate}
            \item If $\gamma'$ is a boundary segment, see the first picture in Figure~\ref{fig:one-int}, by Definition~\ref{def:standard ideal}~(2) (see the first picture in Figure~\ref{fig:std}), $\delta$ is $\R$-standard. 
            \item If $\gamma'$ is not a boundary segment and either $\delta(1)\in\MM$ or $\delta(1)\in\PP$ with $\kappa_\delta(1)=1$, see the second and the third pictures in Figure~\ref{fig:one-int}, by definition, $b_{q,\gamma',\T_2}(e(\delta))\neq 0$. Then by Remark~\ref{rem:total}, we have $b_{\gamma',\T_2}(e(\delta))\neq 0$, which implies $\gamma'\in\R$ by (3). Then by Definition~\ref{def:standard ideal}~(2) (see the first picture in Figure~\ref{fig:std}), $\delta$ is $\R$-standard. 
            \item If $\delta(1)\in\PP$ and $\kappa_\delta(1)=-1$, see the fourth picture in Figure~\ref{fig:one-int}, we have $\delta=\rho(\delta_0)$. So by Corollary~\ref{cor:shear}, $b_{\delta_0,\T_2}(e(\delta))=-1$. Then by (3), $\delta_0\in\R$, a contradiction with the setting of subsection (ii), since $\delta(1)$ is a tagged intersection between $\delta$ and $\delta^\times_0\in\R^\times$.
		\end{enumerate}
		\begin{figure}[htpb]
			\begin{tikzpicture}[scale=1.9]
				\draw[red,thick,bend left=10](-1,0)node[black]{$\bullet$}to(0,1)node[black]{$\bullet$};
				\draw[ultra thick,bend left=10](0,1)to(1,0)node[black]{$\bullet$};
				\draw(-.5,.8)node[red]{$\gamma$}(.5,.8)node{$\gamma'$}(-.65,0.05)node[blue]{$\times$};
				\draw[blue,thick,bend left=10,->-=.5,>=stealth](-1,0)to(1,0);
				\draw(0,.25)node[blue]{$\delta$}(0,-.25)node[red]{$\delta_0$};
				\draw[red,thick,bend right=10](-1,0)to(1,0);
			\end{tikzpicture}\quad
			\begin{tikzpicture}[scale=1.9]
				\draw[red,thick,bend left=10](-1,0)node[black]{$\bullet$}to(0,1)node[black]{$\bullet$};
				\draw[red,thick,bend left=10](0,1)to(1,0)node[black]{$\bullet$};
				\draw(-.5,.8)node[red]{$\gamma$}(.5,.8)node[red]{$\gamma'$}(-.65,0.05)node[blue]{$\times$};
				\draw[ultra thick,bend left=5](1.05,-.3)to(1,.5);
				\draw[blue,thick,bend left=10,->-=.5,>=stealth](-1,0)to(1,0);
				\draw(.4,.18)node[blue]{$\delta$}(0,-.25)node[red]{$\delta_0$};
				\draw[red,thick,bend right=10](-1,0)to(1,0);
				\draw[blue,thick,dashed](-.8,.2)to[out=-45,in=0](-1,-.3)to[out=180,in=180](-1,.3)to[out=0,in=180](1,.3);
				\draw(0,.4)node[blue]{$e(\delta)$}(.77,.3)node{$\bullet$}(.83,.43)node{$q$};
			\end{tikzpicture}\quad
			\begin{tikzpicture}[scale=1.9]
				\draw[red,thick,bend left=10](-1,0)node[black]{$\bullet$}to(0,1)node[black]{$\bullet$};
				\draw[red,thick,bend left=10](0,1)to(1,0)node[black]{$\bullet$};
				\draw(-.5,.8)node[red]{$\gamma$}(.5,.8)node[red]{$\gamma'$}(-.65,0.05)node[blue]{$\times$};
				\draw[blue,thick,bend left=10,->-=.5,>=stealth](-1,0)to(1,0);
				\draw(.4,.18)node[blue]{$\delta$}(0,-.25)node[red]{$\delta_0$};
				\draw[red,thick,bend right=10](-1,0)to(1,0);
				\draw[blue,thick,dashed](-.8,.2)to[out=-45,in=0](-1,-.3)to[out=180,in=180](-1,.3)to[out=0,in=180](1,.3)to[out=0,in=0](1,-.3)to[out=180,in=-135](.8,.2);
				\draw(0,.4)node[blue]{$e(\delta)$}(.77,.3)node{$\bullet$}(.83,.43)node{$q$};
			\end{tikzpicture}\quad
			\begin{tikzpicture}[scale=1.9]
				\draw[red,thick,bend left=10](-1,0)node[black]{$\bullet$}to(0,1)node[black]{$\bullet$};
				\draw[red,thick,bend left=10](0,1)to(1,0)node[black]{$\bullet$};
				\draw(-.5,.8)node[red]{$\gamma$}(.5,.8)node[red]{$\gamma'$}(-.65,0.05)node[blue]{$\times$}(.65,0.05)node[blue]{$\times$};
				\draw[blue,thick,bend left=10,->-=.6,>=stealth](-1,0)to(1,0);
				\draw(0,.2)node[blue]{$\delta$}(0,-.25)node[red]{$\delta_0$};
				\draw[red,thick,bend right=10](-1,0)to(1,0);
			\end{tikzpicture}\quad
			\begin{tikzpicture}[scale=1.9]
				\draw[red,thick,bend right=10](-1,.5)node[black]{$\bullet$}to(.8,.5)node[black]{$\bullet$};
				\draw[red,thick](-1,.5)to[out=30,in=90](1.2,.5)to[out=-90,in=-30](-1,.5);
				\draw[blue,thick,bend left=5,->-=.5,>=stealth](-1,.5)to(.8,.5);
				\draw(.8,1.1)node[red]{$\gamma$}(.2,.25)node[red]{$\gamma'\sim\delta_0$}(-.55,.53)node[blue]{$\times$}(.4,.55)node[blue]{$?$}(0,.65)node[blue]{$\delta$};
			\end{tikzpicture}
			\caption{No interior intersections but one tagged intersection}
			\label{fig:one-int}
		\end{figure}
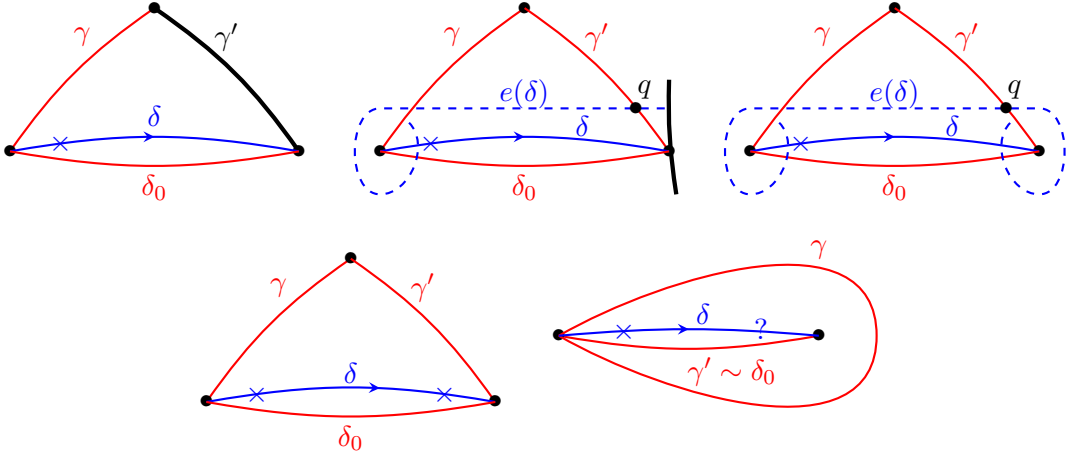
		
		\item[(iii)] Both endpoints of $\delta$ are tagged intersections between $\delta$ and some arcs in $\R^\times$. Let $\delta_0$ be the underlying arc of $\delta$. Note that $\delta_0$ has no intersections with $\R$. Let $\T_3$ be an ideal triangulation containing $\R$ and $\delta_0$. Then $b_{\delta_0,\T_3}(e(\delta))\neq 0$ by Corollary~\ref{cor:shear}, which by (3) implies $\delta_0\in\R$. So $\delta$ is adjoint to $\delta^\times_0\in\R^\times$ and hence is $\R$-standard.
	\end{enumerate}
	
    If $\delta$ has interior intersections with some arc in $\R$, then $\delta$ is divided by $\R$ into at least two arc segments. Let $\gamma$ be the first arc in $\T\cup B$ next to $\delta$ around $\delta(0)$ in the anticlockwise order if $\kappa_\delta(0)=-1$ or in the clockwise order otherwise. If $\gamma\in\T$, we have $b_{\gamma,\T}(e(\delta))\neq 0$ as shown in the pictures of the last row in Figure~\ref{fig:lam of std}. Then by (3), we have $\gamma\in\R$. So the end arc segment of $\delta$ is as shown in the last two pictures in Figure~\ref{fig:std}. For any arc segment $\eta$ of $\delta$ between two arcs $\gamma_1,\gamma_2\in\R$, let $\gamma'_0=\gamma_1,\gamma'_1,\cdots,\gamma'_{s-1},\gamma'_s=\gamma_2$ be the arcs in $\T$ which $\eta$ crosses in order. By (3), we have $b_{\gamma'_i,\T}(e(\delta))=0$ for any $1\leq i\leq s-1$. Hence $\eta$ cuts out an angle bounded by $\gamma_1$ and $\gamma_2$ (cf. the third picture in Figure~\ref{fig:shear}). So $\delta$ is $\R$-standard.
\end{proof}

\begin{remark}\label{rmk:two local tri}
    Let $\R$ be a partial ideal triangulation of $\surf$ and $\gamma\in\R$ which is not the folded side of a self-folded triangle of $\R$. Let $\delta$ be an $\R$-standard (resp. $\R$-co-standard) tagged arc which is neither in $\R^\times$ nor adjoint to some arc in $\R^\times$, and let $L=e(\delta)$ (resp. $L=e^{op}(\delta)$). We denote by $\gamma\cap \alpha=\{(t_1,t_2) \mid \gamma(t_1)=\alpha(t_2) \}$ the set of intersections between $\gamma$ and $\alpha$, where $\alpha=\delta$ or $L$. Note that $\gamma\cap\delta$ contains $\cap(\gamma,\delta)$ (see Definition~\ref{def:intersection}) as a subset. There is an injective map 
    $$\begin{array}{rccc}
        E: 
        &\gamma\cap\delta & \to & \gamma\cap L\\
        &\jiaodian & \mapsto & q \\
         & 
    \end{array}$$
    as shown in the last two rows of Figure~\ref{fig:lam of std} (for the case $L=e(\delta)$). Any intersection $q\in \gamma\cap L$ which is not in the image of $E$ does not contribute to $b_{\gamma,\R}(L)$, i.e., $b_{q,\gamma,\R}(L)=0$. For any $\jiaodian\in\gamma\cap\delta$ and its corresponding $q\in\gamma\cap L$, we have the following equivalences.
    \begin{itemize}
        \item $b_{q,\gamma,\R}(L)>0$ if and only if either $\jiaodian\in\surfi$ and we are in the situation shown in the first picture of the first row of Figure~\ref{fig:p/n int}, or $\jiaodian\in\PP$ with tagging $-1$ (resp. $\jiaodian\in\MM$ or $\PP$ with tagging $1$) and $\gamma$ is the first arc in $\R$ next to $\delta$ anticlockwise around $\jiaodian$, see 
        the first picture of the second (resp. third) row of Figure~\ref{fig:p/n int}. In each case, we call $\jiaodian$ \emph{positive}.
        \item $b_{q,\gamma,\R}(L)<0$ if and only if either $\jiaodian\in\surfi$ and we are in the situation shown in the second picture of the first row of Figure~\ref{fig:p/n int}, or $\jiaodian\in\MM$ or $\PP$ with tagging $1$ (resp. $\jiaodian\in\PP$ with tagging $-1$) and $\gamma$ is the first arc in $\R$ next to $\delta$ clockwise around $\jiaodian$, see 
        the second picture of the second (resp. third) row of Figure~\ref{fig:p/n int}. In each case, we call $\jiaodian$ \emph{negative}.
    \end{itemize}
    Any positive or negative intersection $\jiaodian\in\gamma\cap\delta$ is called \emph{alternative}.

    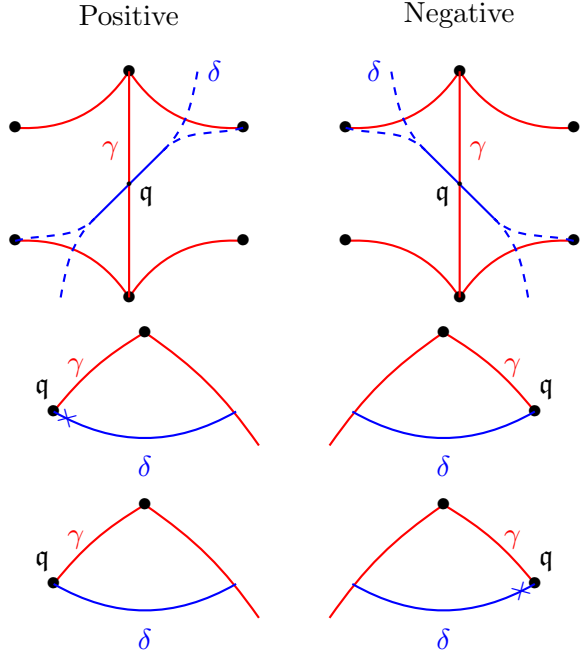
\begin{figure}[htpb]
	\begin{tikzpicture}[scale=1.5]
		\draw[red,thick,bend right](-1,.5)node[black]{$\bullet$}to(0,1)node[black]{$\bullet$}to(1,0.5)node[black]{$\bullet$} (1,-0.5)node[black]{$\bullet$}to(0,-1)node[black]{$\bullet$}to(-1,-.5)node[black]{$\bullet$};
		\draw[red,thick](0,1)to(0,-1);
  
		\draw[blue,thick](-.3,-.3)to(.3,.3);
        \draw[blue,thick,dashed](.3,.3)to[out=45,in=-100](.6,1) (-.6,-1)to[out=80,in=-135](-.3,-.3) (1,.5)to[out=-165,in=45](.3,.3) (-1,-.5)to[out=15,in=-135](-.3,-.3);
		
        \draw (0,1.5)node{Positive};
		\draw[red] (0,.3)node[left]{$\gamma$};
		\draw[blue] (.75,1)node{$\delta$};
		\draw (.15,-.1)node{$\jiaodian$};
		\draw (0,0)node{$\boldsymbol{\cdot}$};
        
	\end{tikzpicture}
	\qquad
	\begin{tikzpicture}[scale=1.5]
		\draw[red,thick,bend right](-1,.5)node[black]{$\bullet$}to(0,1)node[black]{$\bullet$}to(1,0.5)node[black]{$\bullet$} (1,-0.5)node[black]{$\bullet$}to(0,-1)node[black]{$\bullet$}to(-1,-.5)node[black]{$\bullet$};
		\draw[red,thick](0,1)to(0,-1);
		\draw[blue,thick](-.3,.3)to(.3,-.3);
        \draw[blue,thick,dashed](-.3,.3)to[out=135,in=-80](-.6,1) (.6,-1)to[out=100,in=-45](.3,-.3) (1,-.5)to[out=165,in=-45](.3,-.3) (-1,.5)to[out=-15,in=135](-.3,.3);
  
        \draw (0,1.5)node{Negative};
		\draw[red] (0,.3)node[right]{$\gamma$};
		\draw[blue] (-.75,1)node{$\delta$};
		\draw (-.15,-.1)node{$\jiaodian$};
		\draw (0,0)node{$\boldsymbol{\cdot}$};
	\end{tikzpicture}
	
	\begin{tikzpicture}[xscale=1.5,yscale=1.5]
		\draw[red,thick,bend right=10](0,1)tonode[left]{$\gamma$}(-.8,.3)node[black]{$\bullet$};
		\draw[blue](-.68,.24)node[rotate=15]{$+$};
		\draw[red,thick,bend left=10](0,1)to(1,0);
		\draw[blue,thick,bend right](-.8,.3)to(.8,.3);
		\draw[blue,thick](0,0)[below]node{$\delta$};
		\draw(0,1)node{$\bullet$}(-.9,.5)node[black]{$\jiaodian$};
	\end{tikzpicture}
	\qquad
	\begin{tikzpicture}[xscale=1.5,yscale=1.5]
		\draw[red,thick,bend right=10](0,1)to(-1,0);
		\draw[red,thick,bend left=10](0,1)tonode[right]{$\gamma$}(.8,.3)node[black]{$\bullet$};
		\draw[blue,thick,bend right](-.8,.3)to(.8,.3);
		\draw[blue,thick](0,0)[below]node{$\delta$};
		\draw(0,1)node{$\bullet$}(.9,.5)node[black]{$\jiaodian$};
	\end{tikzpicture}

	\begin{tikzpicture}[xscale=1.5,yscale=1.5]
		\draw[red,thick,bend right=10](0,1)tonode[left]{$\gamma$}(-.8,.3)node[black]{$\bullet$};
		\draw[red,thick,bend left=10](0,1)to(1,0);
		\draw[blue,thick,bend right](-.8,.3)to(.8,.3);
		\draw[blue,thick](0,0)[below]node{$\delta$};
		\draw(0,1)node{$\bullet$}(-.9,.5)node[black]{$\jiaodian$};
	\end{tikzpicture}
 \qquad
    \begin{tikzpicture}[xscale=1.5,yscale=1.5]
		\draw[red,thick,bend right=10](0,1)to(-1,0);
		\draw[red,thick,bend left=10](0,1)tonode[right]{$\gamma$}(.8,.3)node[black]{$\bullet$};
		\draw[blue](.68,.24)node[rotate=-15]{$+$};
		\draw[blue,thick,bend right](-.8,.3)to(.8,.3);
		\draw[blue,thick](0,0)[below]node{$\delta$};
		\draw(0,1)node{$\bullet$}(.9,.5)node[black]{$\jiaodian$};
	\end{tikzpicture}
	\caption{Alternative intersections of (co-)standard tagged arcs}\label{fig:p/n int}
\end{figure}
\end{remark}

\begin{definition}\label{def:dissection}
    Let $\R$ be a partial tagged triangulation of $\surf$. An \emph{$\R$-dissection} (resp. \emph{$\R$-co-dissection}) is a maximal collection $\U$ of $\R$-standard (resp. $\R$-co-standard) arcs such that $\Int(\gamma_1,\gamma_2)=0$ for any $\gamma_1,\gamma_2\in\U$. We denote by $D(\R)$ (resp. $D^{op}(\R)$) the set of $\R$-dissections (resp. $\R$-co-dissections).
\end{definition}

By definition, any $\R$-dissection is also a partial tagged triangulation of $\surf$.

\begin{theorem}\label{thm:arc and rigid}
    Let $\R$ be a partial tagged triangulation of $\surf$. The bijection $X$ in Theorem~\ref{thm:QZ} restricts to bijections
    \[X:\stTA{\R}\to\r(X(\R)\ast X(\R)[1])\ \emph{and}\ X:\cstTA{\R}\to\r(X(\R)[-1]\ast X(\R)),\]
    which induce bijections
    \begin{equation}\label{eq:bi-st}
        \begin{array}{rccc}
     & D(\R) & \to & \mr(X(\R)\ast X(\R)[1])\\
            & \U & \mapsto & \bigoplus_{\delta\in\U}X(\delta) 
    \end{array}
    \end{equation}
    and
    \begin{equation}\label{eq:bi-cst}
    \begin{array}{rccc}
	 & D^{op}(\R) & \to & \mr(X(\R)[-1]\ast X(\R))\\
            & \U & \mapsto & \bigoplus_{\delta\in\U}X(\delta) 
    \end{array}
    \end{equation}
    respectively. Moreover, for any $\delta\in\stTA{\R}$ and any $\gamma\in\R$, we have
    \begin{equation}\label{eq:ind=b1}
        [\ind_{X(\R)}X(\delta):X(\gamma)]=-b_{\gamma,\R}(e(\delta)),
    \end{equation}
    and for any $\delta'\in\cstTA{\R}$ and any $\gamma\in\R$, we have
    \begin{equation}\label{eq:ind=b2}
        [\ind_{X(\R)[-1]}X(\delta'):X(\gamma)[-1]]=-b_{\gamma,\R}(e^{op}(\delta')).
    \end{equation}
\end{theorem}

\begin{proof}
    By the equivalence between (1) and (4) in Proposition~\ref{lem:=}, we get the required bijections. The last assertion then follows from Proposition~\ref{prop:index} and the equality~\eqref{eq:R=T}.
\end{proof}

One consequence of the above theorem is the rank of a dissection/co-dissection.

\begin{corollary}
    Let $\R$ be a partial tagged triangulation of $\surf$, and $\U$ a set of $\R$-standard (resp. $\R$-co-standard) arcs such that $\Int(\gamma_1,\gamma_2)=0$ for any $\gamma_1,\gamma_2\in\U$. Then $\U$ is an $\R$-dissection (resp. $\R$-co-dissection) if and only if $|\U|=|\R|$.
\end{corollary}

\begin{proof}
    This follows directly from Theorem~\ref{thm:arc and rigid} and Proposition~\ref{prop:rank}.
\end{proof}

The following result tells us that dissection and co-dissection are dual notions.
	
\begin{corollary}\label{cor:dual}
    Let $\R$ and $\U$ be two partial tagged triangulations of $\surf$. Then $\U$ is an $\R$-dissection if and only if $\R$ is a $\U$-co-dissection.
\end{corollary}

\begin{proof}
    By Theorem~\ref{thm:arc and rigid}, $\U\in D(\R)$ if and only if $X(\U)\in\mr(X(\R)\ast X(\R)[1])$, and $\R\in D^{op}(\U)$ if and only if $X(\R)\in\mr(X(\U)[-1]\ast X(\U))$. So the statement follows by Corollary~\ref{cor:obj dual}.
\end{proof}

The following lemma is useful in Section~\ref{sec:eg}.

\begin{lemma}\label{lem:neq0}
    Let $\R$ be a partial tagged triangulation of $\surf$, and $\U$ an $\R$-dissection. Then for any $l\in\U$, there exists $\gamma\in\R$ such that $b_{l,\U}(e^{op}(\gamma))\neq 0$. Moreover, if $b_{l,\U}(e^{op}(\gamma))> 0$ (resp. $<0$), then for any $\gamma'\in\R$, we have $b_{l,\U}(e^{op}(\gamma'))\geq 0$ (resp. $\leq 0$).
\end{lemma}

\begin{proof}
    By Corollary~\ref{cor:dual}, $\R$ is a $\U$-co-dissection. In particular, any $\gamma\in\R$ is $\U$-co-standard. Then for any $l\in\U$, by \eqref{eq:ind=b2}, we have
    $$\sum_{\gamma\in\R}b_{l,\U}(e^{op}(\gamma))=-[\ind_{X(\U)[-1]}X(\R):X(l)[-1]],$$
    which is not zero by Lemma~\ref{cor:non-zero}. Hence there exists $\gamma\in\R$ such that $b_{l,\U}(e^{op}(\gamma))\neq 0$. The last assertion then follows by Remark~\ref{rem:total1}.
\end{proof}

We have the following geometric interpretation of a certain subcategory of the module category of a surface rigid algebra.

\begin{theorem}\label{thm:geo}
    Let $\R$ be a partial tagged triangulation of a punctured marked surface $\surf$, and $\Lambda_\R=\End_{\C(\surf)}X(\R)$ the corresponding surface rigid algebra. Then there is a bijection
    $$M:\stTA{\R}\setminus\rho(\R)\to\ind \tr\mod\Lambda_\R,$$
    where $\ind\tr\mod\Lambda_\R$ is the set of (isoclasses of) indecomposable $\tau$-rigid $\Lambda_\R$-modules, such that for any $\gamma_1,\gamma_2\in\stTA{\R}\setminus\rho(\R)$, we have
    $$\Int(\gamma_1,\gamma_2)=\dim_\k\Hom_{\Lambda_\R}(M(\gamma_1),\tau M(\gamma_2))+\dim_\k\Hom_{\Lambda_\R}(M(\gamma_2),\tau M(\gamma_1)).$$
\end{theorem}

\begin{proof}
    For any $\gamma\in\stTA{\R}\setminus\rho(\R)$, define $M(\gamma)=\Hom_{\C(\surf)}(X(\R),X(\gamma))$. Then by Theorem~\ref{thm:arc and rigid} and Theorem~\ref{thm:CZZ}, $M$ is a bijection from $\stTA{\R}\setminus\rho(\R)$ to $\ind \tr\mod\Lambda$. 
	
    Let $X_i=X(\gamma_i)$ and $M_i=M(\gamma_i)$, $i=1,2$. Then we have
	\[\begin{aligned}
		\Hom_{\C(\surf)}(X_1,X_2[1])&{}=\Hom_{\C(\surf)/R[1]}(X_1,X_2[1])\oplus[R[1]](X_1,X_2[1])\\
		&\cong\Hom_{\C(\surf)/R[1]}(X_1,X_2[1])\oplus D\Hom_{\C(\surf)/R[1]}(X_2,X_1[1])\\
		&\cong\Hom_{\Lambda_R}(M_1,\tau M_2)\oplus D\Hom_{\Lambda_R}(M_2,\tau M_1),
	\end{aligned}\]
	where the first isomorphism is due to \cite[Lemma~2.3]{CZZ} and the last one is due to Theorem~\ref{thm:CZZ}. By Theorem~\ref{thm:QZ}, we have $\Int(\gamma_1,\gamma_2)=\dim_\k\Hom_{\C(\surf)}(X_1,X_2[1])$. Hence we get the required formula.
\end{proof}

\subsection{Skew-gentle algebras as surface rigid algebras}\label{subsec:sg}
In this subsection, we show that skew-gentle algebras are surface rigid algebras. So the geometric model in Theorem~\ref{thm:geo} is a generalization of a weak version of our previous work \cite{HZZ}.

\begin{definition}\label{def:sg}
    A triple $(Q,Sp,I)$ of a quiver $Q$, a subset $Sp\subseteq Q_0$ and a set $I$ of paths of length 2 in $Q$ is called \emph{skew-gentle} if $(Q^{sp},I^{sp})$ satisfies the following conditions, where $Q^{sp}_0=Q_0, Q^{sp}_1=Q_1\cup\{\epsilon_i\mid i\in Sp\}$ with $\epsilon_i$ a loop at $i$ and $I^{sp}=I\cup\{\epsilon^2_i\mid i\in Sp\}$.
    \begin{itemize}
        \item Each vertex in $Q^{sp}_0$ is the start of at most two arrows in $Q_1^{sp}$, and is the terminal of at most two arrows in $Q_1^{sp}$;
        \item For each arrow $\alpha\in Q^{sp}_1$, there is at most one arrow $\beta\in Q^{sp}_1$ (resp. $\gamma\in Q^{sp}_1$) such that $\alpha\beta\in I^{sp}$ (resp. $\alpha\gamma\notin I^{sp}$);
        \item For each arrow $\alpha\in Q^{sp}_1$, there is at most one arrow $\beta\in Q^{sp}_1$ (resp. $\gamma\in Q^{sp}_1$) such that $\beta\alpha\in I^{sp}$ (resp. $\gamma\alpha\notin I^{sp}$).
    \end{itemize}

    A finite dimensional algebra $\Lambda$ is said to be \emph{skew-gentle} if $\Lambda\cong\k Q^{sp}/\<I^{sg}\>$ for some skew-gentle triple $(Q,Sp,I)$, where $\<I^{sg}\>$ is the ideal generated by $I^{sg}=I\cup\{\epsilon^2_i-\epsilon_i\mid i\in Sp\}$. 
\end{definition}

A partial ideal triangulation of $\surf$ is called \emph{admissible} if each puncture is contained in a self-folded triangle.
\begin{definition}[{\cite[Definition~2.1]{BCS},\cite[Definition~1.8]{HZZ}}]
	Let $\R$ be an admissible partial ideal triangulation of $\surf$. Denote by $\R_0$ the subset of $\R$ such that $\R\setminus\R_0$ consists of the folded sides of self-folded triangles of $\R$, and by $\R_1$ the subset of $\R_0$ consisting of the non-folded sides of self-folded triangles of $\R$. The \emph{tiling algebra} of $\R$ is defined to be $\Lambda^t_\R=\k Q_\R/\<I^t_\R\>$, where the quiver $Q_\R=((Q_\R)_0,(Q_\R)_1,s,t)$ and the relation set $I^t_\R$ are given by the following. 
	\begin{itemize}
		\item The vertices in $(Q_\R)_0$ are (indexed by) the arcs in $\R_0$.
		\item There is an arrow $\alpha\in (Q_\R)_1$ from $i$ to $j$ whenever the corresponding arcs $i$ and $j$ share an endpoint $p_{\alpha}\in \MM$ such that $j$ follows $i$ anticlockwise immediately in $\R_0$. Note that by this construction, each vertex in $(Q_\R)_0$ admits at most one loop.
		\item The relation set $I^t_\R=I^t_{\R,1}\cup I^t_{\R,2}$, where
		\begin{itemize}
		    \item $I^t_{\R,1}$ consists of squares of loops in $(Q_\R)_1$, and
			\item $I^t_{\R,2}$ consisting of all $\alpha\beta$ if $p_{\beta}\neq p_{\alpha}$, or the endpoints of the curve (corresponding to) $t(\beta)=s(\alpha)$ coincide and we are in one of the situations in Figure~\ref{fig:I2 loop}.
			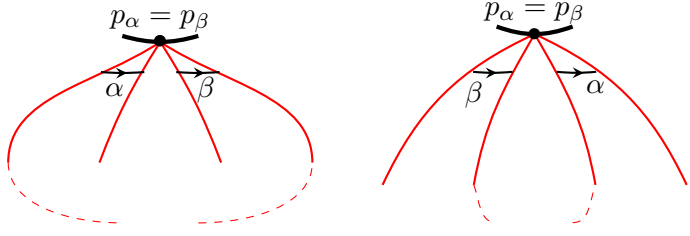
\begin{figure}[htpb]\centering
				\begin{tikzpicture}[xscale=1,yscale=.8]
					\draw[ultra thick,bend right=20](-.5,2.1)to(.5,2.1);
					\draw[red,thick](0,2)to[out=-140,in=90](-2,0);
					\draw[red,thick](2,0)to[out=90,in=-40](0,2);
					\draw[red,dashed](-2,0)to[out=-90,in=180](-.5,-1);
					\draw[red,dashed](.5,-1)to[out=0,in=-90](2,0); 
					\draw(0,2)node{$\bullet$}(0,.8);
					\draw[red,thick,bend right=5](0,2)to(-.8,0);
					\draw[red,thick,bend left=5](0,2)to(.8,0);
					\draw[thick,->-=.6,>=stealth,bend right=5](-.78,1.5)to(-.21,1.5);
					\draw[thick,-<-=.6,>=stealth,bend left=5](.78,1.5)to(.21,1.5);
					\draw(-.6,1.2)node{$\alpha$}(.6,1.2)node{$\beta$}; 
					\draw(0,2)node{$\bullet$}(0,.8);
     \draw(0,2)node[above]{$p_\alpha=p_\beta$};
				\end{tikzpicture}
				\qquad
				\begin{tikzpicture}[xscale=1,yscale=1]
					\draw[ultra thick,bend right=20](-.5,2.1)to(.5,2.1);
					\draw[red,thick](0,2)to[out=-120,in=80](-.8,0);
					\draw[red,dashed](-.8,0)to[out=-70,in=180](-.5,-.5);
					\draw[red,dashed](.5,-.5)to[out=0,in=-110](.8,0);
					\draw[red,thick](.8,0)to[out=100,in=-60](0,2);
					\draw[red,thick,bend right=20](0,2)to(-2,0);
					\draw[red,thick,bend left=20](0,2)to(2,0);
					\draw[thick,->-=.6,>=stealth,bend right=5](-.81,1.5)to(-.29,1.5);
					\draw[thick,-<-=.6,>=stealth,bend left=5](.81,1.5)to(.29,1.5);
					\draw(-.8,1.2)node{$\beta$}(.8,1.3)node{$\alpha$};
					\draw(0,2)node{$\bullet$};
     \draw(0,2)node[above]{$p_\alpha=p_\beta$};
				\end{tikzpicture}
				\caption{Relations in $I^t_{\R,2}$, the case $p_\beta=p_\alpha$ and $s(\alpha)=t(\beta)$ is a loop}
				\label{fig:I2 loop}
			\end{figure}
		\end{itemize}
	\end{itemize}
    
    Let $I^{s\text{-}t}_{\R}$ be the subset obtained from $I^t_{\R}$ by replacing $\epsilon^2$ with $\epsilon^2-\epsilon$ for each loop $\epsilon$ at a vertex in $\R_1$. The \emph{skew-tiling algebra} of $\R$ is defined to be $\Lambda^{s\text{-}t}_\R=\k Q_\R/\<I^{s\text{-}t}_\R\>$.
\end{definition}

By definition, skew-tiling algebras are obtained from tiling algebras by specializing nilpotent loops at vertices in $\R_1$ to be idempotents.

\begin{remark}\label{rmk:surf}
    In the original definition of tiling algebras introduced in \cite{BCS}, there are no punctures on the surface. However, after replacing each puncture by a boundary component with a marked point, and replacing each self-folded triangle of $\R$ by a monogon with the corresponding new boundary component in its interior, our definition coincides with the original one.
\end{remark}

\begin{theorem}[{\cite[Theorem~1]{BCS} and \cite[Corollary~1.12]{HZZ}}]\label{thm:sg}
	A finite dimensional algebra is a (skew-)gentle algebra if and only if it is a (skew-)tiling algebra.
\end{theorem}

In the following, we show that the skew-gentle algebras form a special class of surface rigid algebras.

\begin{theorem}\label{thm:sg is end}
	Let $\R$ be an admissible partial ideal triangulation of $\surf$ and $\Lambda_\R=\End_{\C(\surf)}(X(\R))$ the corresponding surface rigid algebra. Then there is an algebra isomorphism
	\[\Lambda_\R\cong\Lambda^{s\text{-}t}_\R.\]
	In particular, skew-gentle algebras are surface rigid algebras.
\end{theorem}

\begin{proof}
    Let $\T$ be an admissible triangulation containing $\R$. By \cite[Theorem~3.5]{A1}, there is an algebra isomorphism 
    \[\psi:\End_{\C(\surf)}X(\T)\cong\Lambda^\T\]
    sending the identity $\operatorname{id}_{X(\gamma)}$ of $X(\gamma)$ to the idempotent $e^\gamma$ of $\Lambda^\T$ corresponding to $\gamma\in Q_0^\T$, for any $\gamma\in\T$.	Let $e^\R=\sum_{\gamma\in\R}e^\gamma\in\Lambda^\T$ and $e_{\R}=\sum_{\gamma\in\R_0}e_\gamma\in\Lambda^{s\text{-}t}_\T$, where $e_\gamma$ is the idempotent of $\Lambda^{s\text{-}t}_\T$ corresponding to $\gamma\in\R_0$. So we have
    \[\Lambda_\R=\End_{\C(\surf)}X(\R))=\operatorname{id}_{R}\End_{\C(\surf)}X(\T))\operatorname{id}_{R}\cong e^\R\Lambda^\T e^{\R}\cong e_{\R}\Lambda^{s\text{-}t}_\T e_{\R},\]
    where the first isomorphism is induced by $\psi$, and the last isomorphism is due to \cite[Proposition~4.4]{QZ}.

    Let $\surf^t$ be the unpunctured marked surface obtained from $\surf$ by replacing each puncture by a boundary component with one marked point. By Remark~\ref{rmk:surf}, $\Lambda^t_\T$ and $\Lambda^t_\R$ are exactly tiling algebras defined in \cite{BCS} on $\surf^t$. So we have $e_{\R}\Lambda^t_\T e_\R\cong\Lambda^t_\R$ by \cite[Theorem~2.8]{BCS}. Denote by $\T_1$ the subset of $\T$ consisting of all non-folded sides of self-folded triangles. Since $\R$ is admissible, we have $\R_1=\T_1$. So $\Lambda^{s\text{-}t}_\T$ and $\Lambda^{s\text{-}t}_\R$ are obtained from $\Lambda^t_\T$ and $\Lambda^t_\R$ respectively by specializing the same set of nilpotent loops to be idempotents. Hence $e_{\R}\Lambda^{s\text{-}t}_\T e_\R\cong\Lambda^{s\text{-}t}_\R$, which implies the isomorphism $\Lambda_\R\cong \Lambda^{s\text{-}t}_\R$. 
    
    The last assertion then follows from Theorem~\ref{thm:sg}.
\end{proof}

\section{Connectedness of exchange graphs}\label{sec:eg}

Throughout this section, let $\R$ be a partial tagged triangulation of a punctured marked surface $\surf$ and $\Lambda_\R=\End_{\C(\surf)}X(\R)$ the corresponding surface rigid algebra.

\subsection{Flips of dissections}

Recall from Definition~\ref{def:tag arc} that $\TA(\surf)$ denotes the set of tagged arcs on $\surf$. For any tagged arc $\eta\in\TA(\surf)$, let $$\TA(\surf)_{\{\eta\}}=\{\gamma\in\TA(\surf)\setminus\{\eta\}\mid \Int(\gamma,\eta)=0\},$$
and let $\surf/\{\eta\}$ be the punctured marked surface obtained from $\surf$ by cutting along $\eta$  (cf. \cite[Section~5.2]{QZ}). Then there is a natural bijection
$$F_{\{\eta\}}:\TA(\surf)_{\{\eta\}}\to\TA(\surf/{\{\eta\}}),$$
sending $\gamma$ to the tagged arc on $\surf/\{\eta\}$ obtained from $\gamma$ by forgetting the taggings at each puncture which is an endpoint of $\eta$ (and hence becomes a marked point on the boundary of $\surf/{\{\eta\}}$), unless $\gamma$ is homotopic to $\eta$, where $F_{\{\eta\}}(\gamma)$ is the tagged arc shown in Figure~\ref{fig:cut}. Note that for any $\gamma_1,\gamma_2\in\TA(\surf)_{\{\eta\}}$, we have $\Int(\gamma_1,\gamma_2)=0$ if and only if $\Int(F_{\{\eta\}}(\gamma_1),F_{\{\eta\}}(\gamma_2))=0$. Therefore, $F_{\{\eta\}}$ induces a bijection from the set of partial tagged triangulations of $\surf$ containing $\eta$ to the set of partial tagged triangulations of $\surf/\{\eta\}$.
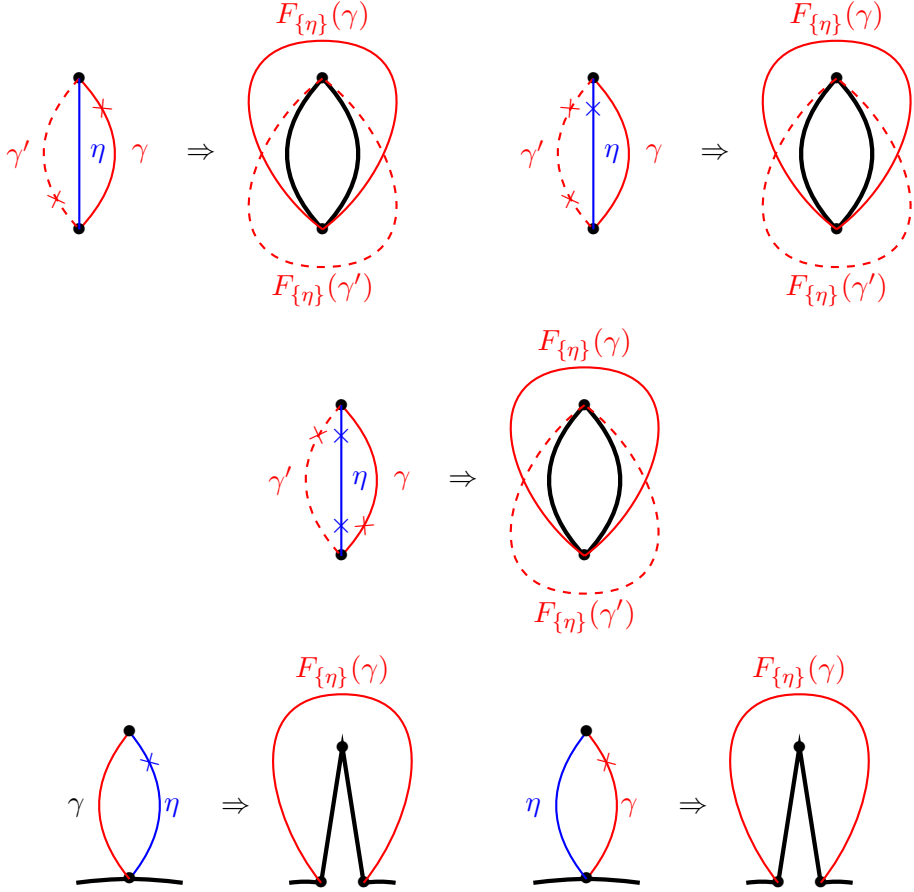
\begin{figure}[htpb]\centering
	\begin{tikzpicture}[xscale=1.6,yscale=1]
		\draw[red,thick,bend right=30](-3,-1)node[black]{$\bullet$}to(-3,1)node[black]{$\bullet$};
		\draw[red,thick,dashed,bend right=30](-3,1)to(-3,-1);
		\draw[blue,thick](-3,1)tonode[right]{$\eta$}(-3,-1);
		\draw(-2,0)node{$\Rightarrow$}(-2.81,.6)node[rotate=40,red]{$\times$}(-3.19,-.6)node[red,rotate=20]{$\times$}(-2.5,0)node[red]{$\gamma$}(-3.5,0)node[red]{$\gamma'$}(-1,-1.8)node[red]{$F_{\{\eta\}}(\gamma')$}(-1,1.8)node[red]{$F_{\{\eta\}}(\gamma)$};
		\draw[ultra thick,bend right=30](-1,1)node[black]{$\bullet$}to(-1,-1)node[black]{$\bullet$}to(-1,1);
		\draw[red,thick](-1,-1)to[out=130,in=180](-1,1.5)to[out=0,in=50](-1,-1);
		\draw[red,thick,dashed](-1,1)to[out=-130,in=180](-1,-1.5)to[out=0,in=-50](-1,1);
	\end{tikzpicture}
	\qquad
	\begin{tikzpicture}[xscale=1.6,yscale=1]
		\draw[red,thick,bend right=30](-3,-1)node[black]{$\bullet$}to(-3,1)node[black]{$\bullet$};
		\draw[red,thick,dashed,bend right=30](-3,1)to(-3,-1);
		\draw[blue,thick](-3,1)tonode[right]{$\eta$}(-3,-1)(-3,.6)node{$\times$};
		\draw(-2,0)node{$\Rightarrow$}(-3.19,.6)node[rotate=40,red]{$\times$}(-3.19,-.6)node[red,rotate=20]{$\times$}(-2.5,0)node[red]{$\gamma$}(-3.5,0)node[red]{$\gamma'$}(-1,-1.8)node[red]{$F_{\{\eta\}}(\gamma')$}(-1,1.8)node[red]{$F_{\{\eta\}}(\gamma)$};
		\draw[ultra thick,bend right=30](-1,1)node[black]{$\bullet$}to(-1,-1)node[black]{$\bullet$}to(-1,1);
		\draw[red,thick](-1,-1)to[out=130,in=180](-1,1.5)to[out=0,in=50](-1,-1);
		\draw[red,thick,dashed](-1,1)to[out=-130,in=180](-1,-1.5)to[out=0,in=-50](-1,1);
	\end{tikzpicture}
	\qquad

	\begin{tikzpicture}[xscale=1.6,yscale=1]
		\draw[red,thick,bend right=30](-3,-1)node[black]{$\bullet$}to(-3,1)node[black]{$\bullet$};
		\draw[red,thick,dashed,bend right=30](-3,1)to(-3,-1);
		\draw[blue,thick](-3,1)tonode[right]{$\eta$}(-3,-1)(-3,.6)node{$\times$}(-3,-.6)node{$\times$};
		\draw(-2,0)node{$\Rightarrow$}(-3.19,.6)node[rotate=-20,red]{$\times$}(-2.81,-.6)node[red,rotate=-40]{$\times$}(-2.5,0)node[red]{$\gamma$}(-3.5,0)node[red]{$\gamma'$}(-1,-1.8)node[red]{$F_{\{\eta\}}(\gamma')$}(-1,1.8)node[red]{$F_{\{\eta\}}(\gamma)$};
		\draw[ultra thick,bend right=30](-1,1)node[black]{$\bullet$}to(-1,-1)node[black]{$\bullet$}to(-1,1);
		\draw[red,thick](-1,-1)to[out=130,in=180](-1,1.5)to[out=0,in=50](-1,-1);
		\draw[red,thick,dashed](-1,1)to[out=-130,in=180](-1,-1.5)to[out=0,in=-50](-1,1);
	\end{tikzpicture}
	
	\begin{tikzpicture}[yscale=1,xscale=1.4]
		\draw[ultra thick,bend left=10](-.5,-1)to(.5,-1);
		\draw[blue,thick,bend right=30](0,-.95)node[black]{$\bullet$}to(0,1);
		\draw[red,thick,bend right=30](0,1)node[black]{$\bullet$}to(0,-.95);
		\draw(1,0)node{$\Rightarrow$}(.19,.6)node[rotate=30,blue]{$\times$}(-.5,0)node{$\gamma$}(.4,0)node[blue]{$\eta$}(2,1.8)node[red]{$F_{\{\eta\}}(\gamma)$};
		\draw[ultra thick,bend left=10](1.5,-1)to(1.8,-1)node{$\bullet$};
		\draw[ultra thick,bend left=10](2.2,-1)node{$\bullet$}to(2.5,-1);
		\draw[ultra thick](1.8,-1)to(2,.8)node{$\bullet$}to(2.2,-1);
		\draw[red,thick](1.8,-1)to[out=120,in=180](2,1.5)to[out=0,in=60](2.2,-1);
	\end{tikzpicture}
	\qquad
	\begin{tikzpicture}[yscale=1,xscale=1.4]
		\draw[ultra thick,bend left=10](-.5,-1)to(.5,-1);
		\draw[red,thick,bend right=30](0,-.95)node[black]{$\bullet$}to(0,1);
		\draw[blue,thick,bend right=30](0,1)node[black]{$\bullet$}to(0,-.95);
		\draw(1,0)node{$\Rightarrow$}(.19,.6)node[rotate=30,red]{$\times$}(-.5,0)node[blue]{$\eta$}(.4,0)node[red]{$\gamma$}(2,1.8)node[red]{$F_{\{\eta\}}(\gamma)$};
		\draw[ultra thick,bend left=10](1.5,-1)to(1.8,-1)node{$\bullet$};
		\draw[ultra thick,bend left=10](2.2,-1)node{$\bullet$}to(2.5,-1);
		\draw[ultra thick](1.8,-1)to(2,.8)node{$\bullet$}to(2.2,-1);
		\draw[red,thick](1.8,-1)to[out=120,in=180](2,1.5)to[out=0,in=60](2.2,-1);
	\end{tikzpicture}
	\caption{The bijection $F_{\{\eta\}}$}
	\label{fig:cut}
\end{figure}

For any partial tagged triangulation $\N$ of $\surf$, set
$$\TA(\surf)_\N=\{\gamma\in\TA(\surf)\setminus\N\mid \Int(\gamma,\eta)=0,\text{ for any $\eta\in\N$}\}.$$
Assume that we have defined a punctured marked surface $\surf/{\N'}$ and a bijection $F_{\N'}:\TA(\surf)_{\N'}\to\TA(\surf/\N')$ for some $\N'\subset\N$ with $\N=\N'\cup\{\eta\}$. Then $F_{\N'}$ restricts to a bijection $F_{\N'}|_{\TA(\surf)_{\N}}:\TA(\surf)_{\N}\to\TA(\surf/\N')_{F_{\N'}(\eta)}$. So we can define inductively the punctured marked surface
$$\surf/\N=(\surf/\N')/F_{\N'}(\eta)$$
and the bijection
$$F_\N=F_{F_{\N'}(\eta)}\circ F_{\N'}|_{\TA(\surf)_{\N}}:\TA(\surf)_{\N}\to\TA(\surf/\N).$$

Recall from Theorem~\ref{thm:IY} that $\C(\surf)_{X(\N)}={}^\bot X(\N)[1]/\<\add X(\N)\>$ is a triangulated category.

\begin{lemma}
    Let $\N$ be a partial tagged triangulation of $\surf$. Then there is a triangle equivalence
    \begin{equation}\label{eq:tri equiv}
        \xi:\C(\surf)_{X(\N)}\to\C(\surf/\N),
    \end{equation}
    such that 
    \begin{equation}\label{eq:cut}
    \xi(X(\gamma))=X_{\surf/\N}(F_\N(\gamma))
    \end{equation}
    holds for any $\gamma\in\TA(\surf)_\N$, where $X_{\surf/\N}$ is the bijection in Theorem~\ref{thm:QZ} for $\surf/\N$.
\end{lemma}

\begin{proof}
    Let $\T$ be a tagged triangulation containing $\N$. Then $F_{\N}(\T\setminus\N)$ is a tagged triangulation of $\surf/\N$ and the corresponding quiver with potential $(Q^{F_{\N}(\T\setminus\N)},W^{F_{\N}(\T\setminus\N)})$ can be obtained from the quiver with potential $(Q^{\T},W^{\T})$ by deleting the vertices corresponding to the tagged arcs in $\N$ and the incident arrows. Then by \cite[Theorem~7.4]{K}, the canonical projection $\pi:\Lambda^{\T}\to\Lambda^{F_{\N}(\T\setminus\N)}$ induces the triangle equivalence~\eqref{eq:tri equiv} such that \eqref{eq:cut} holds for any $\gamma\in\T\setminus\N$. Then by Theorem~\ref{thm:QZ}, \eqref{eq:cut} holds for any $\gamma\in\TA(\surf)_\N$.
\end{proof}

Recall from Construction~\ref{cons:tag to ideal} that for any partial tagged triangulation $\N$, there is an associated partial ideal triangulation $\N^\circ$. For any $l\in\N$, denote by $l^{\circ_\N}$ the arc in $\N^\circ$ corresponding to $l$, which is sometimes simply denoted by $l^\circ$ when there is no confusion arising.

\begin{remark}\label{rmk:cut}
    Let $\N$ be a partial tagged triangulation of $\surf$. By the construction, $\surf/\N$ is obtained from $\surf$ by cutting along arcs in $\N^\circ$.
\end{remark}

\begin{definition}\label{def:relrot}
    Let $\N$ be a partial tagged triangulation of $\surf$. The \emph{tagged rotation with respect to $\N$} is defined to be the bijection
    $$\rho_\N=F_\N^{-1}\circ\rho_{\surf/\N}\circ F_\N:\TA(\surf)_{\N}\to\TA(\surf)_{\N},$$
    where $\rho_{\surf/\N}$ is the tagged rotation on $\surf/\N$. 
\end{definition}

To better understand the tagged rotation with respect to a partial tagged triangulation, we introduce the following notion of flip of a partial ideal triangulation.

\begin{definition}\label{def:flip2}
    Let $\V$ be a partial ideal triangulation of $\surf$ and $v\in\V$ not the folded side of a self-folded triangle of $\V$. Recall that $B$ is the set of boundary segments. We denote
    \begin{itemize}
        \item $v_s$ (resp. $v_t$) the arc in $\V\cup B$ the first arc in $(\V\setminus\{v\})\cup B$ next to $v$ around $v(0)=v_s(0)$ (resp. $v(1)=v_t(1)$) in the clockwise order, and
        \item $v^s$ (resp. $v^t$) the first arc in $(\V\setminus\{v\})\cup B$ next to $v$ around $v(0)=v^s(0)$ (resp. $v(1)=v^t(1)$) in the anticlockwise order.
    \end{itemize}
     See Figures~\ref{fig:angle by two} and \ref{fig:angle by one} for these notations. We remark that some of $v_s,v_t,v^s,v^t$ may not exist.
    
    The negative (resp. positive) flip $f^-_{\V}(v)$ (resp. $f^+_{\V}(v)$) of $v$ with respect to $\V$ is defined to be the arc obtained from $v$ by moving $v(0)$ along $v_s$ (resp. $v^s$), if exists, to $v_s(1)$ (resp. $v^s(1)$) and moving $v(1)$ along $v_t$ (resp. $v^t$), if exists, to $v_t(0)$ (resp. $v^t(0)$). The positive flip $f^+_v(\V)$ and the negative flip $f^-_v(\V)$ of $\V$ at $v$ are defined respectively to be
    $$f^+_v(\V)=(\V\setminus\{v\})\cup\{f^+_{\V}(v)\},\ f^-_v(\V)=(\V\setminus\{v\})\cup\{f^-_{\V}(v)\}.$$
    We also denote by $\theta_s$ (resp. $\theta^s$, $\theta_t$ and $\theta^t$) the angle between a segment of $v$ near $v(0)$ (resp. $v(0)$, $v(1)$ and $v(1)$) and a segment of $v_s$ (resp. $v^s$, $v_t$ and $v^t$) near $v_s(0)$ (resp. $v^s(0)$, $v_t(1)$ and $v^t(1)$).
\end{definition}

\begin{lemma}\label{lem:flipcirc}
    Let $\U=\N\cup\{l\}$ be a partial tagged triangulation of $\surf$ such that $l^\circ$ is not the folded side of a self-folded triangle of $\U^\circ$. Then $(\N\cup\{\rho^{-1}_\N(l)\})^\circ=f^-_{l^\circ}(\U^\circ)$ and $(\N\cup\{\rho_\N(l)\})^\circ=f^+_{l^\circ}(\U^\circ)$.
\end{lemma}

\begin{proof}
    We only prove the first equality while the second can be proved similarly. Since $l^\circ$ is not the folded side of a self-folded triangle of $\U^\circ$, we have $\U^\circ=\N^\circ\cup\{l^\circ\}$ and that $F_\N(l)$ is homotopic to $l^\circ$. By Remark~\ref{rmk:cut}, $\surf/\N$ is obtained from $\surf$ by cutting the surface along arcs in $\N^\circ$. Then $l^\circ_s$ and $l^\circ_t$ (if exist) are the boundary segments of $\surf/\N$ next to $F_\N(l)$ around $l^\circ(0)$ and $l^\circ(1)$ respectively in the clockwise order. Hence, by definition, $\rho^{-1}_{\surf/\N}(F_\N(l))$ is homotopic to $f^-_{U^\circ}(l^\circ)$. 
    
    Denote $\U'=\N\cup\{\rho^{-1}_\N(l)\}$. If $\rho^{-1}_\N(l)$ is not the folded side of a self-folded triangle of ${\U'}^\circ$, then  $${\U'}^\circ=\N^\circ\cup\{\rho^{-1}_\N(l)^{\circ_{\U'}}\}=\N^\circ\cup\{F_{\N}(\rho^{-1}_\N(l))\}=\N^\circ\cup \{f^-_{U^\circ}(l^\circ)\}=f^-_{l^\circ}(U^\circ).$$
    If $\rho^{-1}_\N(l)$ is the folded side of a self-folded triangle $\Delta$ of $\U'^{\circ}$, then there is $h\in\N$ such that $h^{\circ_{\U'}}$ is the non-folded side of $\Delta$. So $h^{\circ_{\U'}}$ is homotopic to $f^-_{U^\circ}(l^\circ)$ and $\rho^{-1}_\N(l)^{\circ_{\U'}}$ is homotopic to $h^{\circ_\N}=h^{\circ_\U}$. Thus, we also have ${\U'}^\circ=f^-_{l^\circ}(U^\circ)$.
\end{proof}

By Lemma~\ref{lem:neq0}, for any $l$ in an $\R$-dissection $\U$, either there exists $\gamma\in\R$ such that $b_{l,\U}(e^{op}(\gamma))>0$ or there exists $\gamma\in\R$ such that $b_{l,\U}(e^{op}(\gamma))<0$ but not both. This ensures the well-definedness of the following notion of flip of dissections.

\begin{definition}\label{def:flip}
    Let $\U$ be an $\R$-dissection. For any $l\in\U$, the \emph{flip} $f^{\R}_l(\U)$ of $\U$ at $l$ is defined to be
    $$f^{\R}_l(\U)=\begin{cases}
    (\U\setminus\{l\})\cup\{\rho_{\U\setminus\{l\}}(l)\}&\text{if $b_{l,\U}(e^{op}(\gamma))>0$ for some $\gamma\in\R$,}\\
    (\U\setminus\{l\})\cup\{\rho^{-1}_{\U\setminus\{l\}}(l)\}&\text{if $b_{l,\U}(e^{op}(\gamma))<0$ for some $\gamma\in\R$.}
    \end{cases}$$
    We denote by $f^{\R}_{\U}(l)$ the tagged arc in $f_l^{\R}(\U)$ but not in $\U$.
\end{definition}

We will simply denote $f_{\U}(l)=f_{\U}^{\R}(l)$ and $f_{l}(\U)=f_{l}^{\R}(\U)$ when there is no confusion arising. By Lemma~\ref{lem:flipcirc}, for any $l\in \U$ such that $l^\circ$ is not the folded side of a self-folded triangle of $\U^\circ$, we have
\begin{equation}\label{eq:flip}
    f_l(\U)^\circ=\begin{cases}
f^+_{l^\circ}(\U^\circ) &\text{if $b_{l,\U}(e^{op}(\gamma))>0$ for some $\gamma\in\R$,}\\
f^-_{l^\circ}(\U^\circ)&\text{if $b_{l,\U}(e^{op}(\gamma))<0$ for some $\gamma\in\R$.}
\end{cases}
\end{equation}

\begin{remark}
    In the case that $\R$ is a tagged triangulation, the $\R$-dissections are exactly the tagged triangulations of $\surf$, and the flip in Definition~\ref{def:flip} coincides with the flip of tagged triangulations introduced in \cite{FST}.
\end{remark}

\begin{proposition}\label{prop:flip}
	For any $\R$-dissection $\U$ and any $l\in\U$, we have that $f_l(\U)$ is the unique $\R$-dissection such that $\U\setminus f_l(\U)=\{l\}$. Moreover, under the bijection~\eqref{eq:bi-st} in Theorem~\ref{thm:arc and rigid}, the flip of $\R$-dissections is compatible with the mutation of maximal rigid objects with respect to $X(\R)\ast X(\R)[1]$.
\end{proposition}

\begin{proof}
    Let $\N=\U\setminus\{l\}$. 
    It suffices to show $X(f_l(\U))$ is the unique (up to isomorphism) basic maximal rigid object with respect to $X(\R)\ast X(\R)[1]$ such that $X(f_l(\U))$ contains $X(\N)$ as a direct summand and is not isomorphic to $X(\U)$. By Corollary~\ref{cor:mut} and Remark~\ref{rem:total1}, this is equivalent to showing
    $$X(f_l(\U))=\begin{cases}
    \mu_{X(l)}^+(X(\U)) & \text{if $[\ind_{X(\U)[-1]}X(\gamma):X(l)[-1]]>0$ for some $\gamma\in\R$,}\\
    \mu_{X(l)}^-(X(\U)) & \text{if $[\ind_{X(\U)[-1]}X(\gamma):X(l)[-1]]<0$ for some $\gamma\in\R$.}
    \end{cases}$$
    Then by \eqref{eq:indb2}, it suffices to show $$X((\U\setminus\{l\})\cup\{\rho_{\N}(l)\})=\mu_{X(l)}^-(X(\U))$$
    and
    $$X((\U\setminus\{l\})\cup\{\rho^{-1}_{\N}(l)\})=\mu_{X(l)}^+(X(\U)).$$
    By Remark~\ref{rem:mu}, this is equivalent to showing that 
    $$X(\rho_{\N}(l))=X(l)\<1\>_{X(\N)}\text{ and }X(\rho^{-1}_{\N}(l))=X(l)\<-1\>_{X(\N)},$$ 
    where $\<1\>_{X(\N)}$ and $\<-1\>_{X(\N)}$ are the suspension functor of $\C(\surf)_\N$ and its inverse, respectively, see Theorem~\ref{thm:IY}.
   
    Using \eqref{eq:cut}, we have $$\xi(X(\rho_\N(l)))=X_{\surf/\N}(\rho_{\surf/\N}(F_\N(l)))=X_{\surf/\N}(F_\N(l))[1]=\xi(X(l)\<1\>_{X(\N)}).$$
    Hence $X(\rho_\N(l))=X(l)\langle1\rangle_{X(\N)}$. Similarly, we have $X(\rho^{-1}_\N(l))=X(l)\langle-1\rangle_{X(\N)}$. 
\end{proof}

\begin{definition}
	The exchange graph $\eg(\R)$ of $\R$-dissections has $\R$-dissections as vertices and has flips as edges.
\end{definition}

\begin{lemma}\label{lem:graphs}
    There is an isomorphism of graphs $\eg(\R)\cong\operatorname{EG}(\st \Lambda_\R)$.
\end{lemma}

\begin{proof}
    By Theorem~\ref{thm:CZZ} and Theorem~\ref{thm:arc and rigid}, there is a bijection between support $\tau$-tilting $A$-modules and $\R$-dissections, such that, by Theorem~\ref{thm:mutation compatible} and Proposition~\ref{prop:flip}, the mutations of support $\tau$-tilting modules are compatible with the flips of $\R$-dissections. Thus we can get the required isomorphism.
\end{proof}

A subset $\R_1$ of $\R$ is said to be \emph{connected} if there is no decomposition $\R_1=\R'\cup\R''$ such that any arc in $\R'$ and any arc in $\R''$ have no endpoints in common. A \emph{connected component} of $\R$ is a maximal connected subset of $\R$.

\begin{definition}\label{def:connect}
    We call $\R$ \emph{connects to the boundary} if each connected component of $\R$ contains at least one arc incident to a marked point in $\MM$.
\end{definition}

For any $\eta\in\R$, we denote by $\R/\{\eta\}$ the partial tagged triangulation $F_{\{\eta\}}(\R\setminus\{\eta\})$ of $\surf/\{\eta\}$.

\begin{lemma}\label{lem:condition}
    If $\R$ connects to the boundary, then so does $\R/\{\eta\}$.
\end{lemma}

\begin{proof}
    Let $\R=\R'\cup\R''$ with $\R'$ the connected component of $\R$ containing $\eta$. Let $\R'\setminus\{\eta\}=\R_1\cup\cdots\cup\R_s$ with $\R_i,1\leq i\leq s$, connected components of $\R'\setminus\{\eta\}$. Since $\R'$ is connected, for each $1\leq i\leq s$, there exists $\gamma_i\in\R_i$ having a common endpoint with $\eta$. So $F_{\{\eta\}}(\gamma_i)$ has a marked point as an endpoint. Hence $F_{\{\eta\}}(\R\setminus\{\eta\})= F_{\{\eta\}}(\R_1)\cup\cdots \cup F_{\{\eta\}}(\R_s)\cup F_{\{\eta\}}(\R'')$ also connects to the boundary.
\end{proof}

Denote by $D(\R)_{\{\eta\}}$ the subset of $D(\R)$ consisting of all $\R$-dissections containing $\eta$, and $D(\R/\{\eta\})$ the set of $\R/\{\eta\}$-dissections on $\surf/\{\eta\}$.

\begin{lemma}\label{lem:reduction}
     There is a bijection 
     $$\begin{array}{ccc}
         D(\R)_{\{\eta\}} & \to & D(\R/\{\eta\})  \\
         \U & \mapsto & F_{\{\eta\}}(\U\setminus\{\eta\})
     \end{array}$$
     such that flips on both sides are compatible.
\end{lemma}

\begin{proof}
    The bijection~\eqref{eq:bi-st} in Theorem~\ref{thm:arc and rigid} restricts a bijection 
    $$\begin{array}{ccc}
        D(\R)_{\{\eta\}} & \to & \xmr{X(\eta)}(X(\R)\ast X(\R)[1]), \\
        \U & \mapsto & \bigoplus_{\delta\in\U}X(\delta) 
    \end{array}$$
    such that, by Proposition~\ref{prop:flip}, the flip on the left side is compatible with the mutation on the right side. By Proposition~\ref{prop:reduction}, there is a bijection 
    $$\begin{array}{ccc}
        \xmr{X(\eta)}(X(\R)\ast X(\R)[1]) & \to & \mr(X(\R)\ast_{\C(\surf)_{X(\eta)}} X(\R)\<1\>), \\
        \U & \mapsto & X(\U\setminus\{\eta\})
    \end{array}$$
    such that the mutations on both sides are compatible. The triangle equivalence~\eqref{eq:tri equiv} for $\N=\{\eta\}$ gives rise to a bijection
    $$\begin{array}{ccc}
        \mr(X(\R)\ast_{\C(\surf)_{X(\eta)}} X(\R)\<1\>) & \to & \mr(R'\ast_{\C(\surf/\{\eta\})} R'[1]), \\
        X(\U\setminus\{\eta\}) & \mapsto & X_{\surf/\{\eta\}}(F_{\{\eta\}}(\U\setminus\{\eta\}))
    \end{array}$$
    such that the mutations on both sides are compatible, where $R'=X_{\surf/\{\eta\}}(\R/\{\eta\})$. Combining the above three bijections, we get the required one.
\end{proof}

The following lemma, which is crucial in the proof of the main result, will be proved in the next subsection.

\begin{lemma}\label{lem:main}
	Let $\R$ be a partial tagged triangulation on $\surf$  connecting to the boundary. Then for any $\R$-dissection $\U$, there exists a sequence of flips $f_{l_1},\cdots,f_{l_s}$ such that $\R\cap f_{l_s}\circ\cdots\circ f_{l_1}(\U)\neq\emptyset$.
\end{lemma}

The main result in this section is the following connectedness of the graph of $\R$-dissections.

\begin{theorem}\label{thm:connect}
	Let $\R$ be a partial tagged triangulation on $\surf$  connecting to the boundary. Then the graph $\eg(\R)$ is connected.
\end{theorem}

\begin{proof}
    We call two $\R$-dissections $\U$ and $\U'$ flip-connected to each other if one can be obtained from the other by a sequence of flips. To show $\eg(\R)$ is connected, it is enough to show the assertion that any $\R$-dissection $\U$ is flip-connected to $\R$. We use the induction on $|\R|$. When $|\R|=1$, the assertions follows directly from Lemma~\ref{lem:main}. Suppose the assertion is true when $|\R|\leq n-1$ for some $n>1$, and consider the case when $|\R|=n$. By Lemma~\ref{lem:main}, there exists some $\eta\in\R$ and an $\R$-dissection $\U'$ such that $\eta\in\R\cap\U'$ and $\U'$ is flip-connected to $\U$. By Lemma~\ref{lem:reduction}, $F_{\{\eta\}}(\U'\setminus\{\eta\})$ is an $\R/\{\eta\}$-dissection. By Lemma~\ref{lem:condition}, $\R/\{\eta\}$ is a partial tagged triangulation on $\surf/\{\eta\}$ connecting to the boundary. Since $|\R/\{\eta\}|=n-1$, using the induction hypothesis, $F_{\{\eta\}}(\U'\setminus\{\eta\})$ is flip-connected to $\R/\{\eta\}$. Hence by Lemma~\ref{lem:reduction} again, $\U'$ is flip-connected to $\R$. So $\U$ is flip-connected to $\R$.
\end{proof}

\begin{corollary}\label{cor:main}
	The exchange graph $\operatorname{EG}(\st A)$ of support $\tau$-tilting modules over any surface rigid algebra $A$ is connected. In particular, $\operatorname{EG}(\st A)$ is connected for $A$ a skew-gentle algebra.
\end{corollary}

\begin{proof}
	The first assertion follows directly from Lemma~\ref{lem:graphs} and Theorem~\ref{thm:connect}. Then the second assertion holds by Theorem~\ref{thm:sg is end}.
\end{proof}

\subsection{Proof of Lemma~\ref{lem:main}}

This subsection devotes to proving Lemma~\ref{lem:main}. We may assume $\U\cap\R=\emptyset$. Recall that by Corollary~\ref{cor:dual}, $\R$ is a $\U$-co-dissection. In particular, any $\gamma\in\R$ is $\U$-co-standard. So by Proposition~\ref{lem:=}, $e^{op}(\gamma)$ shears $\U$.

We first introduce some notions and notations. For any $\gamma\in\R$ and any $l\in\U$, denote by $\gamma\cap l^\circ\cap\surf^\circ$ the subset of $\gamma\cap l^\circ$ consisting of the intersections not in $\PP\cup\MM$. We also denote $$\Int^\circ(\R,l^\circ)=\sum_{\gamma\in\R}|\gamma\cap l^\circ\cap\surf^\circ|,$$
and
$$\Int^\circ(\gamma,\U^\circ)=\sum_{l\in\U}|\gamma\cap l^\circ\cap\surf^\circ|.$$
An arc $l\in\U$ is said to be 
\begin{itemize}
    \item \emph{maximal} if $\Int^\circ(\R,l^\circ)\geq\Int^\circ(\R,h^\circ)$ for any $h\in\U$,
    \item \emph{double} if $l$ is homotopic to another arc in $\U$, i.e., $l^\circ$ is a side of a self-folded triangle of $\U^\circ$,
    \item \emph{single} if $l$ is not homotopic to any other arc in $\U$.
\end{itemize}
For a maximal $l\in\U$, we also call $l^\circ$ \emph{maximal} in $\U^\circ$.

Denote by $\U^{non}$ the subset of $\U$ consisting of all arcs $l$ such that $l^\circ$ is not the folded side of a self-folded triangle of $\U^\circ$. By definition, any maximal arc in $\U$ is in $\U^{non}$.

Next, we deal with two special cases. Recall that two tagged arcs $\gamma$ and $l$ are called adjoint to each other if $\gamma(0),\gamma(1)\in\PP$ and $l=\rho(\gamma)$.

\begin{lemma}\label{lem:circ0}
	Let $\gamma$ be a tagged arc in $\R$ not adjoint to any arc in $\U$ and such that $\Int^\circ(\gamma,\U^\circ)=0$. Then there exists a sequence of flips $f_{l_1},\cdots,f_{l_s}$ such that $\gamma\in f_{l_s}\circ\cdots\circ f_{l_1}(\U)$.
\end{lemma}

\begin{proof}
    Since $\Int^\circ(\gamma,\U^\circ)=0$, by Definition~\ref{def:standard ideal} and Remark~\ref{rmk:std}, there is $l\in\U^{non}$ such that $\gamma^\U$, $l_s^\circ$ and $l^\circ$ form a contractible triangle clockwise, and with a certain orientation of $\gamma^\U$, we have $\gamma^\U(1)=l^\circ(1)$, $\kappa_{\gamma^\U}(1)=-1$, $\gamma^\U(0)=l_s^\circ(1)$ and $\kappa_{\gamma^\U}(0)=1$ if $\gamma^\U(0)\in\PP$, see Figure~\ref{fig:intcirc=0}. So by Remark~\ref{rmk:two local tri} (see also Figure~\ref{fig:p/n int}) we have $b_{l,\U}(e^{op}(\gamma))<0$. Hence by \eqref{eq:flip}, $f_l(\U)^\circ=f^-_{l^\circ}(\U^\circ)=\U^\circ\setminus\{l^\circ\}\cup\{f^-_{U^\circ}(l^\circ)\}$. Note that $\gamma^\U(0)\neq \gamma^\U(1)$ since $\gamma^\U$ is a tagged arc. 
	\begin{figure}[htpb]
	    \begin{tikzpicture}[xscale=1.8,yscale=1.4]
		\draw[blue,thick,bend right=10,->-=.6,>=stealth](0,1)to(-.8,.3)node[black]{$\bullet$};
		\draw[blue,thick,bend left=10,->-=.6,>=stealth](0,1)to(.8,.3);
		\draw[red,thick,bend right,->-=.6,>=stealth](-.8,.3)to(.8,.3)node[black]{$\bullet$};
		\draw (.8,.3)node[right]{$p$};
		\draw(0,0)[below]node[red]{$\gamma^\U$};
		\draw[red](.68,.24)node[rotate=-15]{$+$};
		\draw(0,1)node{$\bullet$}(-.8,.8)node[blue]{$l_s^\circ$}(.6,.8)node[blue]{$l^\circ$};
	\end{tikzpicture}\quad
	\caption{The case $\Int^\circ(\gamma,\U^\circ)=0$}
	\label{fig:intcirc=0}
	\end{figure}
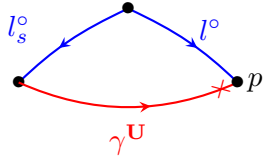
	
    Let $p=\gamma^\U(1)$ and $\U^\circ(p)$ be the set of arc segments of arcs in $\U^\circ$ that have $p$ as an endpoint. We use the induction on $|\U^\circ(p)|$. 
 
    If $|\U^\circ(p)|=1$, then $l^\circ(0)\neq p$ and there is no other $h^\circ\in\U^\circ$ having $p$ as an endpoint. This implies $\Int(\gamma,h)=0$ for any $h\in\U\setminus\{l\}$. Since $\gamma$ is $\R$-standard and $\gamma\neq l$ (due to $\R\cap\U=\emptyset$), by Proposition~\ref{prop:flip}, we have $f_\U(l)=\gamma$, i.e., $\gamma\in f_l(\U)$. 
    
    Suppose the assertion holds for $|\U^\circ(p)|<k$ with $k\geq 0$ and consider the case when $|\U^\circ(p)|=k$. We have the following cases.    
    \begin{enumerate}
        \item $l^\circ(0)=l^\circ(1)$, which contains the case that $l$ is double, see the first picture in Figure~\ref{fig:alt end}. Then one endpoint of $f_\U(l)^\circ$ is $\gamma^\U(0)\neq p$. So $|f_l(\U)^\circ(p)|<k$ and we are done by the induction hypothesis.
		
		\item $l^\circ(0)\neq l^\circ(1)$ and $l^\circ_t(0)\neq l^\circ_t(1)$, see the second picture of Figure~\ref{fig:alt end}. Then $f_\U(l)^\circ$ has $\gamma^\U(0)$ and $l^\circ_t(0)$ as the endpoints, both of which are not $p$. So $|f_l(\U)^\circ(p)|<k$ and we are done by the induction hypothesis.
		
		\item $l^\circ(0)\neq l^\circ(1)$ and $l_t^\circ(0)=l_t^\circ(1)$, see the third picture in Figure~\ref{fig:alt end}. Then exactly one of the endpoints of $f_\U(l)^\circ$ is not $p$. So $|f_l(\U)^\circ(p)|=k$ and we go back to Case (1). Thus, we are done.
	\end{enumerate}
\end{proof}
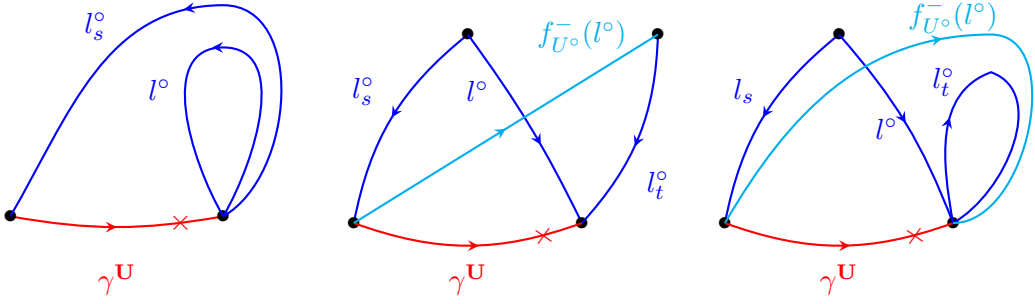
\begin{figure}[htpb]\centering
    \begin{tikzpicture}[scale=2.8]
    	\draw[red,thick,bend right=10,->-=.5,>=stealth](-1,0)node[black]{$\bullet$}to(0,0)node[black]{$\bullet$};
    	\draw[blue,thick,-<-=.5,>=stealth](0,0)to[out=120,in=180](0,.8)to[out=0,in=60](0,0);
    	\draw[blue,thick,-<-=.5,>=stealth](-1,0)to[out=60,in=180](0,1)to[out=0,in=30](0,0);
     \draw(-.5,-.3)node[red]{$\gamma^\U$}(-.6,.9)node[blue]{$l^\circ_s$}(-.3,.6)node[blue]{$l^\circ$}(-.2,-.03)node[red]{$\times$};
    \end{tikzpicture}\quad
    \begin{tikzpicture}[xscale=-2.5,yscale=2.5]
    	\draw[blue,thick,bend left=20,->-=.5,>=stealth](0,1)node[black]{$\bullet$}to(.6,0)node[black]{$\bullet$};
    	\draw[blue,thick,bend right=5,->-=.6,>=stealth](0,1)to(-.6,0)node[black]{$\bullet$};
    	\draw[red,thick,bend right=20,-<-=.5,>=stealth](-.6,0)to(.6,0);
    	\draw[blue,thick,bend left=20,-<-=.5,>=stealth](-.6,0)to(-1,1)node[black]{$\bullet$};
    	\draw[cyan,thick,->-=.5,>=stealth](.6,0)to(-1,1);
    	\draw(0,-.3)node[red]{$\gamma^\U$}(.55,.7)node[blue]{$l^\circ_s$}(-.05,.7)node[blue]{$l^\circ$}(-.4,-.07)node[red,rotate=10]{$\times$}(-.6,1)node[cyan]{$f^-_{U^\circ}(l^\circ)$}(-1,.2)node[blue]{$l^\circ_t$};
    \end{tikzpicture}\quad
	\begin{tikzpicture}[xscale=-2.5,yscale=2.5]
		\draw[blue,thick,bend left=20,->-=.5,>=stealth](0,1)node[black]{$\bullet$}to(.6,0)node[black]{$\bullet$};
		\draw[blue,thick,bend right=10,->-=.5,>=stealth](0,1)to(-.6,0)node[black]{$\bullet$};
		\draw[red,thick,bend right=20,-<-=.5,>=stealth](-.6,0)to(.6,0);
		\draw[blue,thick,->-=.3,>=stealth](-.6,0)to[out=80,in=-20](-.8,.8)to[out=-160,in=150](-.6,0);
		\draw[cyan,thick,->-=.5,>=stealth](.6,0)to[out=120,in=0](-.8,1)to[out=180,in=180](-.6,0);
		\draw(0,-.3)node[red]{$\gamma^\U$}(.5,.7)node[blue]{$l_s$}(-.25,.5)node[blue]{$l^\circ$}(-.4,-.07)node[red,rotate=10]{$\times$}(-.6,1.1)node[cyan]{$f^-_{U^\circ}(l^\circ)$}(-.55,.75)node[blue]{$l^\circ_t$};
	\end{tikzpicture}\quad
	\caption{The cases of flip of $l^\circ$ for $\Int^\circ(\gamma,\U^\circ)=0$}
	\label{fig:alt end}
\end{figure}

\begin{lemma}\label{lem:bu}
    Let $\gamma$ be a tagged arc in $\R$ adjoint to $l\in \U$. Then $l$ is the unique arc in $\U$ such that $b_{l,\U}(e^{op}(\gamma))\neq 0$. Moreover, if $\Int^\circ(\R,l^\circ)\neq 0$, then there exists a sequence of flips $f_{l_1},\cdots,f_{l_s}$ such that $\gamma\in f_{l_s}\circ\cdots\circ f_{l_1}(\U)$.
\end{lemma}

\begin{proof}
    By Remark~\ref{rmk:erho}, $e^{op}(\gamma)=e^{op}(\rho^{-1}(l))=e(l)$. So for any $h\in\U$, $b_{h,\U}(e^{op}(\gamma))=b_{h,\U}(e(l))$, which by Corollary~\ref{cor:shear} is non-zero if and only if $h=l$. Hence $l$ is the unique arc in $\U$ such that $b_{l,\U}(e^{op}(\gamma))\neq 0$. Indeed, we have $b_{l,\U}(e^{op}(\gamma))<0$ by Corollary~\ref{cor:shear}. Hence by \eqref{eq:flip}, $f_l(\U)^\circ=f^-_{l^\circ}(\U^\circ)=\U^\circ\setminus\{l^\circ\}\cup\{f^-_{U^\circ}(l^\circ)\}$.
    
    If $\Int^\circ(\R,l^\circ)\neq0$, since $l$ is adjoint to $\gamma\in\R$, there is a self-folded triangle of $\U^\circ$ such that $l^\circ$ is its non-folded side. Let ${l'}^\circ$ be the corresponding folded side. Then $\gamma$ is homotopic to ${l'}^\circ\in f_l(\U)^\circ$. So $\Int^\circ(\gamma,f_l(\U)^\circ)=0$. Since $l\notin f_l(\U)$, $\gamma$ is not adjoint to any arc in $f_l(\U)$, the assertion follows by Lemma~\ref{lem:circ0}.
\end{proof}

\begin{lemma}\label{lem:alt end}
    Let $\gamma\in\R$. Suppose $\Int^\circ(\gamma,\U^\circ)\neq 0$. Then for any maximal tagged arc $l\in\U$, any alternative intersection between $l^\circ$ and $\gamma$ is interior.
\end{lemma}

\begin{proof}
    Suppose conversely that $l^\circ$ has an alternative intersection $\jiaodian\in\PP\cup\MM$ with $\gamma$.  
    Let $\eta$ be the end arc segment of $\gamma^\U$ incident to $\jiaodian$. Then by Remark~\ref{rmk:two local tri}, we are in one of the situations shown in Figure~\ref{fig:cor-alt end}. In each case, let $\alpha$ be the number of non-end arc segments of arcs in $\R$ cutting out the same angle as $\eta$ and $\beta>0$ the number of end arc segments of arcs in $\R$ homotopic to $\eta$. Then $\Int(\R,l^\circ)=\alpha$ and $\Int(\R,l^\circ_s)\geq \alpha+\beta$, 
	a contradiction to that $l$ is maximal.
	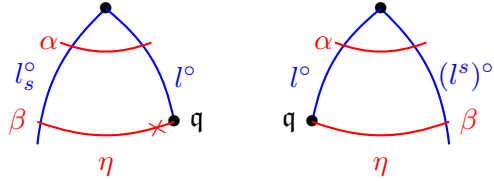
\begin{figure}[htpb]
	\begin{tikzpicture}[xscale=-1.5,yscale=1.5]
		\draw[blue,thick,bend left=20](0,1)node[black]{$\bullet$}to(.6,-.2);
		\draw[blue,thick,bend right=20](0,1)to(-.6,0)node[black]{$\bullet$};
		\draw[red,thick,bend right=20](-.6,0)to(.6,0)node[left]{$\beta$};
		\draw[red,thick,bend right=20](-.4,.7)to(.4,.7);
		
		\draw(0,-.4)node[red]{$\eta$}(.7,.4)node[blue]{$l^\circ_s$}(-.7,.4)node[blue]{$l^\circ$}(-.45,-.06)node[red,rotate=15]{$\times$}(.5,.7)node[red]{$\alpha$}(-.8,0)node{$\jiaodian$};
	\end{tikzpicture}\qquad
	\begin{tikzpicture}[xscale=1.5,yscale=1.5]
		\draw[blue,thick,bend left=20](0,1)node[black]{$\bullet$}to(.6,-.2);
		\draw[blue,thick,bend right=20](0,1)to(-.6,0)node[black]{$\bullet$};
		\draw[red,thick,bend right=20](-.6,0)to(.6,0)node[right]{$\beta$};
		\draw[red,thick,bend right=20](-.4,.7)to(.4,.7);
		
		\draw(0,-.4)node[red]{$\eta$}(.75,.4)node[blue]{$(l^s)^\circ$}(-.7,.4)node[blue]{$l^\circ$}(-.5,.7)node[red]{$\alpha$}(-.8,0)node{$\jiaodian$};
	\end{tikzpicture}
\caption{The case $l^\circ$ has an alternative intersection in $\PP\cup\MM$ with some $\gamma\in\R$ satisfying $\Int^\circ(\gamma,\U^\circ)\neq0$.}
\label{fig:cor-alt end}
\end{figure}
\end{proof}

We need the following notions similarly as in \cite{DP}. For any arc $l\in\U$, we say flipping $l$ is \emph{convenient} if $\Int^\circ(\R,f_l(\U)^\circ)<\Int^\circ(\R,\U^\circ)$, or \emph{neutral} if $\Int^\circ(\R,f_l(\U)^\circ)=\Int^\circ(\R,\U^\circ)$, where
$$\Int^\circ(\R,\U^\circ):=\sum_{l\in\U}\Int^\circ(\R,l^\circ)=\sum_{\gamma\in\R}\Int^\circ(\gamma,\U^\circ).$$
In what follows, for any $\gamma\in\R$ and a positive integer $n$, an $n$-arc segment of $\gamma$ is a segment of $\gamma$ formed by $n$ arc segments. 

\begin{lemma}\label{lem:trans}
    Let $l$ be a maximal arc in $\U$ such that there is a negative (resp. positive) interior intersection $\jiaodian$ between $l^\circ$ and some arc $\gamma\in\R$. Then flipping $l$ is either convenient or neutral. Moreover, if flipping $l$ is neutral, then
    \begin{enumerate}
        \item $l$ is single,
        \item each alternative intersection in $\surfi$ between $l^\circ$ and an arc in $\R$ is not an endpoint of an end arc segment of the arc divided by $\U^\circ$, and
        \item both $l^\circ_s$ and $l^\circ_t$ (resp. both $(l^s)^\circ$ and $(l^t)^\circ$) exist and are maximal if $\jiaodian$ is negative (resp. positive).
    \end{enumerate} 
\end{lemma}

\begin{proof}
    We only deal with the case that $\jiaodian$ is negative since the proof when $\jiaodian$ is positive is similar. By \eqref{eq:flip}, $l$ is flip-convenient if $\Int^\circ(\R,f^-_{\U^\circ}(l^\circ))<\Int^\circ(\R,l^\circ)$, or flip-neutral if $\Int^\circ(\R,f^-_{\U^\circ}(l^\circ))=\Int^\circ(\R,\U^\circ)$. We may assume that $f^-_{\U^\circ}(l^\circ)$ is not homotopic to any arc in $\R$, because otherwise  $\Int^\circ(\R,f^-_{\U^\circ}(l^\circ))=0<\Int^\circ(\R,l^\circ)$. By Remark~\ref{rmk:two local tri}, we are in the situation shown in the second picture of the first row of Figure~\ref{fig:p/n int} with replacing $\gamma$ by $l^\circ$ and replacing $\delta$ by some arc $\gamma\in\R$. There are the following cases.
	\begin{itemize}
	    \item[(a)] Suppose that there is no angle of $\U^\circ$ formed by $l^\circ$ and itself. In this case, the four angles $\theta_s,\theta^s,\theta_t,\theta^t$ (which are defined in Definition~\ref{def:flip2} for $v=l^\circ$) are different. We shall use the following notations, see the first picture in Figure~\ref{fig:angle by two}. 
	    \begin{itemize}
	    \item Let $\alpha_1$ (resp. $\alpha_2$) be the number of 2-arc segments of arcs in $\R$ that cut out the angles $\theta_t$ and $\theta^t$ (resp. $\theta_s$ and $\theta^s$) and whose arc segment cutting out $\theta_t$ (resp. $\theta_s$) does not have an endpoint in $\MM\cup\PP$. 
		\item Let $\delta_1$ (resp. $\delta_2$) be the number of 2-arc segments of arcs in $\R$ that cut out the angles $\theta_t$ and $\theta^t$ (resp. $\theta_s$ and $\theta^s$) and whose arc segment cutting out $\theta_t$ (resp. $\theta_s$) has an endpoint in $\MM\cup\PP$, see the dashed ones in the left picture of Figure~\ref{fig:angle by two}.
		\item Let $\beta_1$ (resp. $\beta_2$) be the number of arc segments of arcs in $\R$ that cross $l^\circ_t$ (resp. $l^\circ_s$) and cut out the angle at $l^\circ_t(0)$ (resp. $l^\circ_s(1)$) clockwise from $l^\circ_t$ (resp. $l^\circ_s$).
		\item Let $\zeta$ be the number of 2-arc segments of arcs in $\R$ that cut out the angles $\theta_s$ and $\theta_t$ and are not incident to any point in $\PP\cup\MM$.
		\item Let $\zeta_1$ (resp. $\zeta_2$) be the number of 2-arc segments of arcs in $\R$ that cut out the angles $\theta_s$ and $\theta_t$ and have one endpoint in $\MM\cup\PP$.
	\end{itemize}
    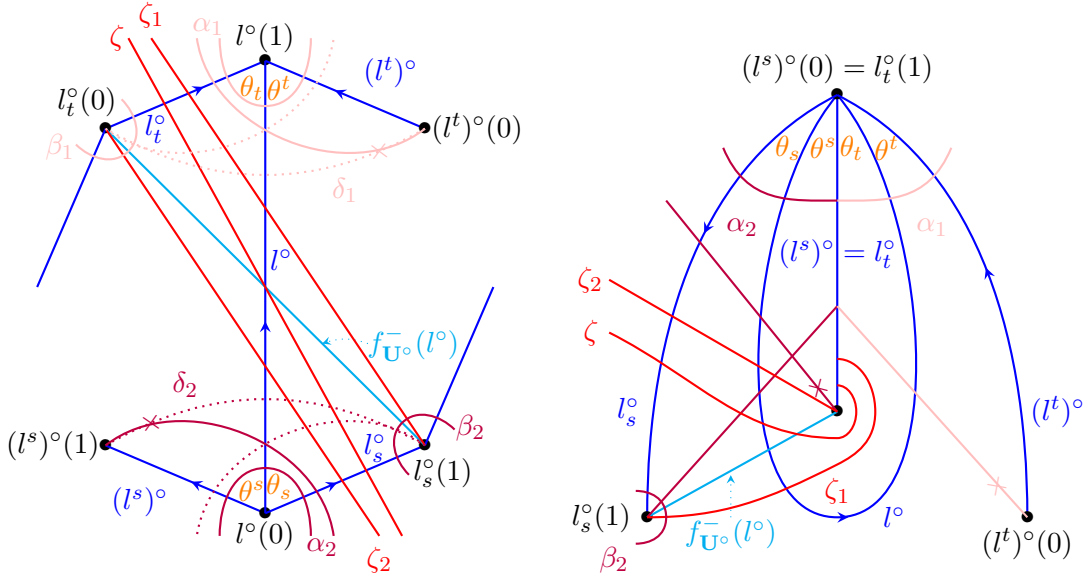
\begin{figure}[htpb]\centering
        \begin{tikzpicture}[xscale=-3,yscale=3]
		\draw[blue,thick,-<-=.45,>=stealth](0,1)node[black]{$\bullet$}to(.7,.7)node[black]{$\bullet$};
		\draw[blue,thick](.7,.7)to(1,0);
		\draw(0.7,-.7)node[black]{$\bullet$};
		\draw[blue,thick,-<-=.6,>=stealth](.7,-.7)to(0,-1)node[black]{$\bullet$};
		\draw[blue,thick,-<-=.6,>=stealth](0,1)to(0,-1);
		\draw[blue,thick,->-=.45,>=stealth](0,-1)to(-.7,-.7)node[black]{$\bullet$};
		\draw[blue,thick](-.7,-.7)to(-1,0);
		\draw[blue,thick,-<-=.5,>=stealth](0,1)to(-.7,.7)node[black]{$\bullet$};
		\draw[cyan,thick](-.7,-.7)to(.7,.7);
        \draw[cyan,dotted,bend left=20,->-=1,>=stealth](-.45,-.25)to(-.25,-.25);
  
		\draw[pink,thick](-.2,1.1)to[out=-90,in=180](0,.8)to[out=0,in=-90](.2,1.1);
		\draw[purple,thick](-.2,-1.1)to[out=90,in=180](0,-.8)to[out=0,in=90](.2,-1.1);
		\draw[pink](.25,1.15)node{$\alpha_1$};
		\draw[purple](-.25,-1.15)node{$\alpha_2$};
		\draw[red,thick](-.6,-1.1)to(.6,1.1);
		\draw[red,thick](-.7,-.7)to(.5,1.1)node[above]{$\zeta_1$};
		\draw[red,thick](.7,.7)to(-.5,-1.1)node[below]{$\zeta_2$};
		\draw[purple,thick](-.63,-.83)to[out=45,in=-45](-.6,-.6)to[out=135,in=45](-.83,-.63);
		\draw[pink,thick](.63,.83)to[out=-135,in=135](.6,.6)to[out=-45,in=-135](.83,.63);
		\draw[purple,thick,bend right=60](.7,-.7)to(-.3,-1.1);
		\draw[purple,dotted,thick,bend left=60](-.7,-.7)to(.3,-1.1);
		\draw[purple,dotted,thick,bend right=30](.7,-.7)to(-.7,-.7);
		\draw[pink,thick,bend right=60](-.7,.7)to(.3,1.1);
		\draw[pink,dotted,thick,bend left=60](.7,.7)to(-.3,1.1);
		\draw[pink,dotted,thick,bend right=30](-.7,.7)to(.7,.7);
		\draw(.65,1.1)node[red]{$\zeta$}(-.9,-.62)node[purple]{$\beta_2$}(-.07,.12)node[blue]{$l^\circ$}(-.65,-.25)node[cyan]{$f^-_{\U^\circ}(l^\circ)$}(.9,.62)node[pink]{$\beta_1$}(.55,-.95)node[blue]{$(l^s)^\circ$}(-.48,-.71)node[blue]{$l_s^\circ$}(-.55,.95)node[blue]{$(l^t)^\circ$}(.48,.7)node[blue]{$l_t^\circ$}(-.5,.6)node[pink,rotate=8]{$\times$}(.5,-.6)node[purple,rotate=10]{$\times$}(-.93,.7)node{$(l^t)^\circ(0)$}(.93,-.7)node{$(l^s)^\circ(1)$}(.78,.82)node{$l^\circ_t(0)$}(-.78,-.82)node{$l^\circ_s(1)$}(0,1.1)node{$l^\circ(1)$}(0,-1.1)node{$l^\circ(0)$}(-.35,.43)node[pink]{$\delta_1$}(.35,-.43)node[purple]{$\delta_2$};
		\draw[orange] (.06,.88)node{$\theta_t$}(-.06,.88)node{$\theta^t$}(.06,-.9)node{$\theta^s$}(-.06,-.88)node{$\theta_s$};
	\end{tikzpicture}\quad
	\begin{tikzpicture}[xscale=5,yscale=7]
		\draw[blue,thick,-<-=.5,>=stealth](0,1)to[out=-40,in=0](0,.2)to[out=180,in=-140](0,1);
		\draw[blue,thick](0,1)node[black]{$\bullet$}to(0,.4)node[black]{$\bullet$};
		\draw[blue,thick](0,.7)node{$(l^s)^\circ=l_t^\circ$};
		\draw[blue,thick,-<-=.5,>=stealth](0,1)to[out=-20,in=90](.5,.2)node[black]{$\bullet$};
		\draw[blue,thick,->-=.4,>=stealth](0,1)to[out=-160,in=90](-.5,.2)node[black]{$\bullet$};
		\draw[cyan,thick](0,.4)to(-.5,.2);
        \draw[cyan,dotted,->-=1,>=stealth](-.28,.2)to(-.28,.28);
  
        \draw[purple,thick](-.3,.9)to[out=-60,in=180](0,.8);
		\draw[purple,thick](-.45,.8)to(0,.4) (-.05,.445)node[rotate=30]{$\times$};
		\draw[purple,thick](-.25,.75)node{$\alpha_2$};
		\draw[purple,thick](-.5,.2)to(0,.6);
		\draw[pink,thick](.3,.9)to[out=-120,in=0](0,.8);
		\draw[pink,thick](.25,.75)node{$\alpha_1$};
		\draw[pink,thick](.5,.2)to(0,.6);
		\draw[pink,thick](.42,.26)node[rotate=-55]{$\times$};
		\draw[red,thick](-.6,.55)node[left]{$\zeta$}to[out=-20,in=-180](0,.35)to[out=0,in=-90](.05,.4)to[out=90,in=0](0,.45);
		\draw[red,thick](-.5,.2)to[out=0,in=-160](0,.3)to[out=20,in=-90](.1,.4)to[out=90,in=0](0,.5) (0,.25)node{$\zeta_1$};
		\draw[red,thick](-.6,.65)to(0,.4);
		\draw(-.65,.65)node[red]{$\zeta_2$}(-.55,.4)node[blue]{$l^\circ_s$}(.58,.4)node[blue]{$(l^t)^\circ$}(-.28,.17)node[cyan]{$f^-_{\U^\circ}(l^\circ)$}(.15,.2)node[blue]{$l^\circ$}(0,1.05)node{$(l^s)^\circ(0)=l^\circ_t(1)$}(.5,.15)node{$(l^t)^\circ(0)$}(-.62,.2)node{$l_s^\circ(1)$};
		\draw[purple,thick](-.53,.25)to[out=0,in=90](-.45,.2)to[out=-90,in=0](-.53,.15) (-.52,.12)node[left]{$\beta_2$};
		\draw[orange](-.13,.9)node{$\theta_s$} (-.037,.9)node{$\theta^s$} (.037,.9)node{$\theta_t$} (.13,.9)node{$\theta^t$};
	\end{tikzpicture}
    \caption{The case that no angle of $\U^\circ$ is formed by the two ends of $l^\circ$}
    \label{fig:angle by two}
    \end{figure}
 
	Then we have the following.
	\[\Int^\circ(\R,l^\circ)=\alpha_1+\alpha_2+\delta_1+\delta_2+\zeta+\zeta_1+\zeta_2,\]
	\[\Int^\circ(\R,f^-_{\U^\circ}(l^\circ))=\beta_1+\beta_2+\zeta,\]
	\[\Int^\circ(\R,l^\circ_t)=\alpha_1+\beta_1+\zeta+\zeta_1,\]
	\[\Int^\circ(\R,l^\circ_s)=\alpha_2+\beta_2+\zeta+\zeta_2.\]
    There are the following subcases.
    \begin{enumerate}
        \item Both $\delta_1$ and $\delta_2$ are non-zero. Since arcs in $\R$ do not cross each other in the interior, we have $\beta_1=\beta_2=0$. So
        \[\Int^\circ(\R,l^\circ)-\Int^\circ(\R,f^-_{\U^\circ}(l^\circ))=\alpha_1+\alpha_2+\delta_1+\delta_2+\zeta_1+\zeta_2>0.\]
        \item Exactly one of $\delta_1$ and $\delta_2$ is not zero. Without loss of generality, suppose $\delta_1=0$ and $\delta_2\neq0$. Since arcs in $\R$ do not cross each other in the interior, we have $\beta_2=\zeta=\zeta_2=0$, which implies $\zeta_1\neq 0$ by the existence of a negative intersection $\jiaodian$. So
        \[\Int^\circ(\R,l^\circ)-\Int^\circ(\R,f^-_{\U^\circ}(l^\circ))=\Int^\circ(\R,l^\circ)-\Int^\circ(\R,l^\circ_t)+\alpha_1+\zeta_1>0.\]
        \item Both $\delta_1$ and $\delta_2$ are zero. So
        \[\Int^\circ(\R,l^\circ)-\Int^\circ(\R,f^-_{\U^\circ}(l^\circ))=2\Int^\circ(\R,l^\circ)-\Int^\circ(\R,l^\circ_s)-\Int^\circ(\R,l_t^\circ)\geq0.\]
    \end{enumerate}
    Hence flipping $l$ is either convenient or neutral. When flipping $l$ is neutral, we have $\delta_1=\delta_2=0$ and $\Int^\circ(\R,l^\circ)=\Int^\circ(\R,l^\circ_s)=\Int^\circ(\R,l_t^\circ)$, which implies that both  $l_s$ and $l_t$ exist and are maximal, and $\alpha_2+\zeta_2=\beta_1$ and $\alpha_1+\zeta_1=\beta_2$. Since at least one of $\beta_1$ and $\zeta_2$ (resp. $\beta_2$ and $\zeta_1$) is zero, we have $\zeta_1=\zeta_2=0$, $\alpha_2=\beta_1$ and $\alpha_1=\beta_2$. If in addition, $l$ is double, then with the chosen orientations, we have $(l^s)^\circ=l_t^\circ$ (with opposite orientations), see the second picture of Figure~\ref{fig:angle by two}. Since the arc segment crossing $l_t^\circ$ and cutting out the angle at $l^\circ_t(0)$ clockwise from $l^\circ_t$ is a loop enclosing the puncture $l^\circ_t(0)$, we have $\beta_1=0$. So $\alpha_2=0$, which implies $\zeta=0$. So we have $\zeta+\zeta_1+\zeta_2=0$ which contradicts with that $\jiaodian$ is negative.
	\item[(b)] 
    Otherwise, by Definition~\ref{def:standard ideal}, $\gamma$ cuts out an angle $\theta$ formed by the different end segments of $l^\circ$, see the first picture in Figure~\ref{fig:angle by one}. In this case, we have $l^\circ_s=l^\circ_t$ and $(l^s)^\circ=(l^t)^\circ$, both with opposite orientations. The three angles $\theta,\theta_s$ and $\theta^t$ are different. We shall use the following notations, see the first picture in Figure~\ref{fig:angle by one}.
    \begin{itemize}
	    \item Let $\alpha$ be the number of 3-arc segments of arcs in $\R$ that cut out the angles $\theta_s$, $\theta$ and $\theta^t$ in order. 
		\item Let $\beta$ be the number of arc segments of arcs in $\R$ that cross $l^\circ_s$ and cut out the angle at $l^\circ_s(1)$ clockwise from $l^\circ_s$.
		\item Let $\zeta$ be the number of 3-arc segments of arcs in $\R$ that cut out the angles $\theta_s$, $\theta$ and $\theta_s$ in order and are not incident to any point in $\PP\cup\MM$.
		\item Let $\zeta_1$ be the number of 3-arc segments of arcs in $\R$ that cut out the angles $\theta_s$, $\theta$ and $\theta_s$ in order and such that $l^\circ$ is the first or last arc in $\U$ they cross.
	\end{itemize}
    Then
    \[\Int^\circ(\R,l^\circ)=2\alpha+2\zeta+2\zeta_1,\]
    \[\Int^\circ(\R,f^-_{\U^\circ}(l^\circ))=2\zeta+2\beta,\]
    \[\Int^\circ(\R,l^\circ_s)=\alpha+\beta+2\zeta+\zeta_1.\]
    So we have
    \[\begin{array}{rl}
         &\Int^\circ(\R,l^\circ)-\Int^\circ(\R,f^-_{\U^\circ}(l^\circ))  \\
         =&2\alpha+2\zeta_1-2\beta\\
         =&2(\Int^\circ(\R,l^\circ)-\Int^\circ(\R,l^\circ_s))\geq0
    \end{array}.\]
    Hence flipping $l$ is either convenient or neutral. When flipping $l$ is neutral, we have $\Int^\circ(\R,l^\circ)=\Int^\circ(\R,l^\circ_s)$, which implies that $l_s=l_t$ exists and is maximal, and $\alpha+\zeta_1=\beta$. Since at least one of $\beta$ and $\zeta_1$ is zero, we have $\zeta_1=0$ and $\alpha=\beta$. So $\zeta\neq 0$ by the existence of a negative intersection $\jiaodian$. If in addition, $l$ is double, then $l^\circ_s=(l^t)^\circ$ (with opposite orientations) be the folded side that is enclosed by $l^\circ$, see the second picture of Figure~\ref{fig:angle by one}. Then any arc in $\R$ cuts out one of $\theta_s$, $\theta$ and $\theta^t$ can only cut out these three angles (by the $\U$-co-standard property) and hence does not exist. This contradicts $\zeta\neq 0$.
\end{itemize}
\end{proof}
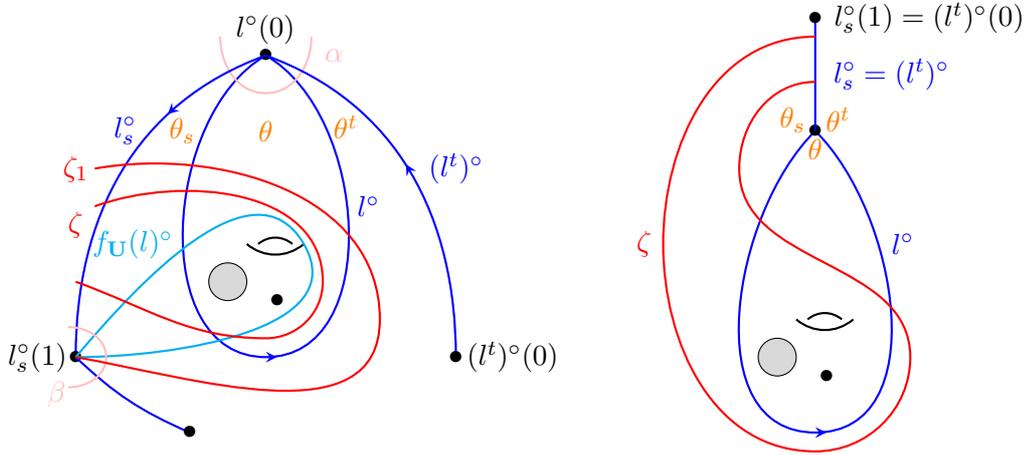
\begin{figure}[htpb]\centering
	\begin{tikzpicture}[scale=5]
	    \draw[blue,thick,-<-=.5,>=stealth](0,1)node[black]{$\bullet$}to[out=-30,in=0](0,.2)to[out=180,in=-150](0,1);
		\draw[fill=gray!30](-.1,.4)circle(.05);
		\draw[thick](-.05,.5)to[bend right=40](.1,.5) (-.02,.5)to[bend left=40](.07,.5) (.03,.35)node{$\bullet$};
		\draw[orange](-.22,.8)node{$\theta_s$}(.21,.81)node{$\theta^t$}(0,.8)node{$\theta$};
		\draw[pink,thick](-.12,1.05)to[out=-90,in=180](0,.9)to[out=0,in=-90](.12,1.05);
		\draw[blue,thick,-<-=.5,>=stealth](0,1)to[out=-20,in=90](.5,.2)node[black]{$\bullet$};
		\draw[blue,thick,->-=.3,>=stealth](0,1)to[out=-160,in=90](-.5,.2)node[black]{$\bullet$};
		\draw[blue,thick](-.5,.2)to[bend right=10](-.2,0)node[black]{$\bullet$};
		
		\draw[cyan,thick](-.5,.2)to[out=50,in=120](.1,.5)to[out=-60,in=0](-.5,.2);
		
		\draw[red,thick](-.45,.6)to[out=20,in=90](.15,.4)to[out=-90,in=0](0,.25)to[out=180,in=-20](-.5,.4);
		\draw[pink,thick](-.52,.28)to[out=0,in=90](-.42,.2)to[out=-90,in=0](-.52,.12) (-.55,.1)node{$\beta$}; 
        \draw[red,thick](-.5,.2)to[out=-10,in=-90](.3,.3)to[out=90,in=10](-.45,.7) (-.5,.7)node{$\zeta_1$};
        \draw(.65,.2)node{$(l^t)^\circ(0)$}(0,1.07)node{$l^\circ(0)$}(-.6,.2)node{$l^\circ_s(1)$};
		\draw(-.5,.55)node[red]{$\zeta$}(-.37,.8)node[blue]{$l^\circ_s$}(.5,.7)node[blue]{$(l^t)^\circ$}(-.35,.5)node[cyan]{$f_\U(l)^\circ$}(.18,1)node[pink]{$\alpha$}(.27,0.6)node[blue]{$l^\circ$}(0,-.15)node{};
	\end{tikzpicture}\qquad
    \begin{tikzpicture}[scale=5]
	    \draw[blue,thick,-<-=.5,>=stealth](0,.8)node[black]{$\bullet$}to[out=-45,in=0](0,0)to[out=180,in=-135](0,.8);
		\draw[fill=gray!30](-.1,.2)circle(.05);
		\draw[thick](-.05,.3)to[bend right=40](.1,.3) (-.02,.3)to[bend left=40](.07,.3) (.03,.15)node{$\bullet$};
        \draw[blue,thick](0,.8)to(0,1.1)node[black]{$\bullet$};
        \draw(0,-.15)node{};
		\draw[orange](-.06,.83)node{$\theta_s$}(.06,.83)node{$\theta^t$}(0,.75)node{$\theta$}; 
		\draw[red,thick](0,.93)to[out=180,in=90](-.2,.7)to[out=-90,in=90](.25,.2)to[out=-90,in=0](0,-.05)to[out=180,in=-90](-.4,.5)to[out=90,in=180](0,1.05);
		\draw(.2,.95)node[blue]{$l^\circ_s=(l^t)^\circ$}(-.45,.5)node[red]{$\zeta$}(.23,.5)node[blue]{$l^\circ$}(.3,1.1)node{$l^\circ_s(1)=(l^t)^\circ(0)$};
	\end{tikzpicture}
	\caption{The case that one angle of $\U^\circ$ is formed by the two ends of $l^\circ$}
	\label{fig:angle by one}
\end{figure}

Now we are ready to complete the proof of Lemma~\ref{lem:main}.

    We use induction hypothesis on $\Int^\circ(\R,\U^\circ)$. For the starting case $\Int^\circ(\R,\U^\circ)=0$, since $\R$ connects to the boundary, there is a tagged arc $\gamma'\in\R$ which has an endpoint in $\MM$. By definition, $\gamma'$ is not adjoint to any arc in $\U$. So the assertion holds by Lemma~\ref{lem:circ0}. 

    Suppose the assertion holds when $\Int^\circ(\R,\U^\circ)\leq k-1$ for some $k\geq 1$, and consider the case when $\Int^\circ(\R,\U^\circ)=k$. 
	
    Let $l$ be a maximal arc in $\U$. Then $\Int^\circ(\R,l^\circ)\neq 0$. By Lemma~\ref{lem:neq0}, the set
    $$\R(l):=\{\gamma\in\R\mid b_{l,\U}(e^{op}(\gamma))\neq 0\}$$
    is non-empty. For any $\gamma\in\R(l)$, if $\gamma$ is adjoint to $\U$, then by Lemma~\ref{lem:bu}, the assertion follows; if $\gamma$ is not adjoint to $\U$ and $\Int^\circ(\gamma,\U^\circ)=0$, then the assertion follows by Lemma~\ref{lem:circ0}. Hence we only need to deal with the following case. 
    \begin{itemize}
        \item[($\ast$)] For any maximal $l\in\U$ and any $\gamma\in\R(l)$, we have $\Int^\circ(\gamma,\U^\circ)\neq 0$. 
    \end{itemize}
    Note that in this case, each $\gamma\in\R(l)$ has an alternative intersection with $l^\circ$ by Remark~\ref{rmk:two local tri} and any such alternative intersection is interior by Lemma~\ref{lem:alt end}. So by Lemma~\ref{lem:trans}, each maximal arc in $\U$ is either flip-convenient or flip-neutral. We use the induction on the minimum number $m$ satisfying that there exists $\gamma\in\R$ which crosses $l_1^\circ,\cdots,l_m^\circ\in\U^\circ$ in succession, at $\jiaodian_1,\cdots,\jiaodian_m$ respectively, such that
    \begin{enumerate}
        \item $\jiaodian_1,\cdots,\jiaodian_m$ are interior, with only $\jiaodian_m$ is alternative, and
        \item $l_1^\circ,\cdots,l_m^\circ$ are maximal in $\U$, but the previous arc $l_0^\circ$  (if exists) in $\U^\circ$ that $\gamma$ crosses before $l_1^\circ$ is not.
    \end{enumerate}
    Note that $m$ is well-defined because some arc in $\R$ has an interior alternative intersection with $\widetilde{l}^\circ$ for a maximal arc $\widetilde{l}\in\U$ and hence we can track along $\gamma$ to get a sequence satisfying the conditions above (although $\widetilde{l}$ may be one of $l_1,\cdots,l_m$), see Figure~\ref{fig:H}. 

    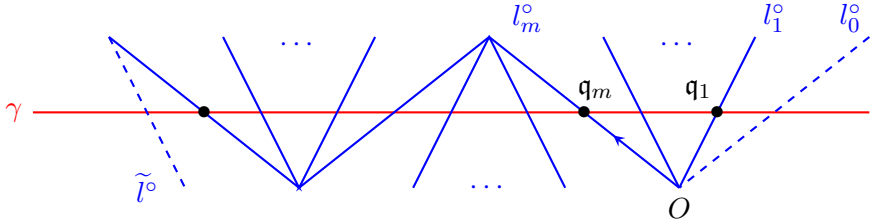
\begin{figure}[htpb]\centering
    \begin{tikzpicture}[xscale=5,yscale=-1]
    \draw[red,thick](-.7,0)node[left]{$\gamma$}to(1.5,0);
	\draw[blue,thick](.3,1)to(.5,-1)to(.7,1);
	\draw[blue](.5,1)node{$\cdots$};
	\draw[blue,thick](1,1)to(1.2,-1);
	\draw[blue,thick,-<-=.7,>=stealth](0.5,-1)to(1,1);
        \draw[blue,thick](1,1)to(0.8,-1);
	\draw(1,-.9)node[blue]{$\cdots$};
        \draw(1,1)node[below]{$O$};
	\draw[blue,thick](-0.5,-1)to(0,1)to(-0.2,-1);
	\draw[blue,thick](0.5,-1)to(0,1)to(0.2,-1);
	\draw[blue,thick,dashed](-.5,-1)to(-.3,1);
	\draw[blue,thick,dashed](1.5,-1)to(1,1);
	\draw(0,-.9)node[blue]{$\cdots$};
	\draw(-.4,1)node[blue]{$\widetilde{l}^\circ$}(1.45,-1.25)node[blue]{$l^\circ_0$}(1.25,-1.25)node[blue]{$l^\circ_1$}(.6,-1.25)node[blue]{$l^\circ_m$};
	\draw(-.25,0)node{$\bullet$}(1.1,0)node{$\bullet$}(.75,0)node{$\bullet$}(.78,-.3)node{$\jiaodian_m$}(1.05,-.3)node{$\jiaodian_1$};
	\end{tikzpicture}
	\caption{The conditions for the definition of $m$}\label{fig:H}
    \end{figure}
 
    We denote $l=l_m$ and without loss of generality, assume that $\jiaodian_m$ is negative. For the starting case $m=1$, if flipping $l$ is neutral, then by Lemma~\ref{lem:trans}, $\jiaodian_1$ is not an endpoint of an end arc segment of $\gamma$ divided by $\U^\circ$, and both $l_s^\circ$ and $l_t^\circ$ exist and are maximal. However, in this case, either $l_s^\circ$ or $l_t^\circ$ is $l_0^\circ$ which is not maximal, a contradiction. So flipping $l$ is convenient and the assertion holds by the induction hypothesis on $\Int^\circ(\R,\U^\circ)$. 
    
    Suppose the assertion holds when $m<k'$ for some $k'>1$ and consider the case when $m=k'$. Let $O$ be the common endpoint of $l_1^\circ,\cdots,l_m^\circ$ and fix the orientation of $l_m^\circ$ with $O$ as the starting point, see Figure~\ref{fig:H}. Since if flipping $l$ is convenient then the assertion follows by the induction hypothesis on $\Int^\circ(\R,\U^\circ)$, in what follows, we assume that flipping $l$ is neutral. By Lemma~\ref{lem:trans}, $l$ is single, $\jiaodian_m$ is not an endpoint of an end arc segment of $\gamma$ divided by $\U^\circ$, and both $l^\circ_s$ and $l^\circ_t$ exist and maximal. It follows that the next intersection $\jiaodian_{m+1}$ of $\gamma$ with $\U^\circ$ after $\jiaodian_m$ is also interior and with a maximal arc $l_{m+1}^\circ\in\U^\circ$. For any $(l')^\circ\in\U^\circ$, let 
    $$\jiaodian((l')^\circ)=\{\jiaodian_i\in (l')^\circ\mid 1\leq i\leq m\}.$$
    Note that we have $|\jiaodian((l')^\circ)|\leq 2$, and the equality holds only if $(l')^\circ$ is a loop at $O$. Depending on $l^\circ$ being whether $l^\circ_{m+1}$, $l^\circ_{m-1}$ or not, there are the following three cases.

    \begin{figure}[htpb]\centering
    \begin{tikzpicture}[scale=2.2]
        \draw[blue,thick,-<-=.5,>=stealth](0,-1)node[black]{$\bullet$}to[out=165,in=-123](-1.4,.4)to[out=57,in=180](-.97,.65)to[out=0,in=105](0,-1);
        \draw[blue,thick,-<-=.5,>=stealth](0,-1)to[out=15,in=-57](1.4,.4)to[out=123,in=0](.97,.65)to[out=180,in=75](0,-1);
        \draw[blue,thick](0,-1)to(-1,-1.5) (0,-1)to(1,-1.5);
        \draw[thick,fill=gray!30](-.5,.03)circle(.12) (-.9,.11)node{$\bullet$};
        \draw[thick,bend left=40](-1,.3)to(-.6,.3) (-.65,.32)to(-.95,.32);
        \draw[orange,thick](-.4,-1.2)to[out=90,in=-135](-.3,-.84)to[out=45,in=170](0.05,-.74) (-.5,-1.05)node{$\alpha$};
        \draw[orange](-.23,-1.03)node{$\theta^t$}(-.1,-.85)node{$\theta$}(0,-.6)node{$\theta_s$};
        \draw[orange,thick,bend right=10](.2,-1.09)to(.2,-.94)(.3,-1.02)node{$\beta$};
        
        \draw[red,thick,->-=.15,>=stealth](.7,-1.35)to[out=90,in=-45](.41,.37)to[out=135,in=0](-.65,.7)to[out=180,in=90](-1.3,.17)to[out=-90,in=180](.1,-.35)to[out=0,in=90](.4,-1.4) (.8,-1.2)node{$\gamma$};
        \draw[red,thick](.36,.28)to[out=140,in=0](-.6,.6)to[out=180,in=90](-1.2,.17)to[out=-90,in=175](.17,-.25) (0,.31)node{$\zeta$};
        
        \draw(.72,-.75)node[black]{$\bullet$}(.85,-.77)node{$\jiaodian_i$} (.42,.37)node[black]{$\bullet$}(.67,.37)node{$\jiaodian_{m-1}$} (.14,-.36)node[black]{$\bullet$}(.38,-.26)node{$\jiaodian_{m+2}$} (-.15,-.36)node[black]{$\bullet$}(-.34,-.45)node{$\jiaodian_{m+1}$}(-.9,.65)node[black]{$\bullet$}(-.9,.8)node{$\jiaodian_m$}  (-1,-.8)node[blue]{$l^\circ=l^\circ_{m+1}$}(.9,.8)node[blue]{$l^\circ_{m-1}=l^\circ_i$};
    \end{tikzpicture}\qquad
    \begin{tikzpicture}[scale=2.2]
            \draw[blue,thick,->-=.4,>=stealth](0,-1)node[black]{$\bullet$}to[out=160,in=-90](-1.2,.3)to[out=90,in=180](-.8,.7)to[out=0,in=90](0,-1);
            \draw[blue,thick,bend right=10](0,-1)to(1,.7)node[black]{$\bullet$};
            \draw[blue,thick](0,-1)to(1,-1);
            \draw[red,thick,->-=.15,>=stealth](.8,-1.1)to[out=90,in=0](-.25,.8)to[out=180,in=90](-1.1,.2)to[out=-90,in=180](-.1,-.3)to[out=0,in=90](.4,-1.1);
            \draw[thick,fill=gray!30](-.5,.1)circle(.15);
            \draw[thick,bend left=40](-.85,.45)to(-.45,.45) (-.5,.47)to(-.8,.47);
            \draw[cyan,thick](1,.7)to[out=-130,in=0](-.5,-.2)to[out=180,in=-90](-1,.4)to[out=90,in=165](1,.7) (.15,1.05)node{$f^-_{\U^\circ}(l^\circ)$};

            \draw(.67,-.1)node[black]{$\bullet$}(.95,-.1)node{$\jiaodian_{m-1}$} (-.05,-.3)node[black]{$\bullet$}(.17,-.25)node{$\jiaodian_{m+1}$}(-.71,.69)node[black]{$\bullet$}(-.35,.7)node{$\jiaodian_m$} (-.65,-.9)node[blue]{$l^\circ=l^\circ_{m+1}$}(1.1,.3)node[blue]{$l^\circ_{m-1}$}(.9,-.8)node[red]{$\gamma$};
        \end{tikzpicture}
        \caption{The case $l^\circ=l_{m+1}^\circ$}
        \label{fig:self-angle}
    \end{figure}
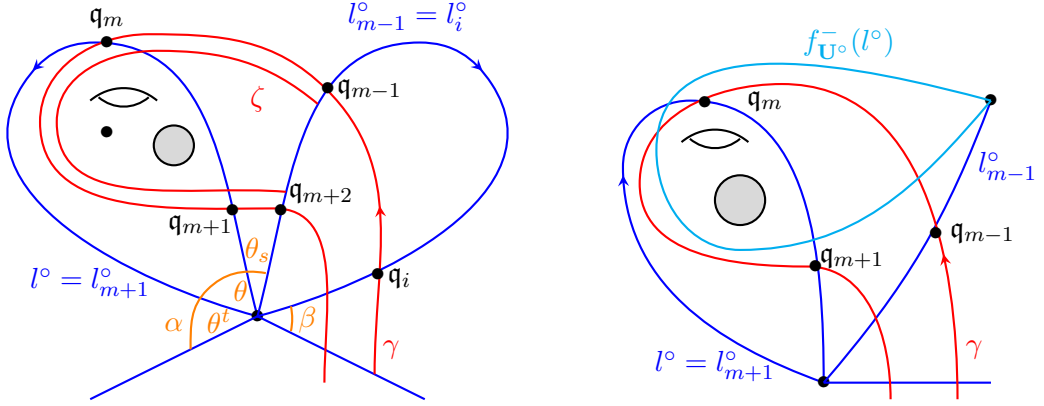

    \begin{enumerate}
        \item Suppose $l^\circ=l_{m+1}^\circ$. Then $l^\circ$ cuts out an angle $\theta$ by its two ends, $|\jiaodian(l^\circ)|=1$ and $l^\circ_t=l^\circ_s=l^\circ_{m-1}$. If $|\jiaodian(l^\circ_s)|=2$, then there is $1\leq i\leq m-2$ such that $l^\circ_{m-1}=l^\circ_i$, see the first picture in Figure~\ref{fig:self-angle}. Similarly as in the case (b) in the proof of Lemma~\ref{lem:trans} (cf. the first picture of Figure~\ref{fig:angle by one}), we denote
    \begin{itemize}
        \item $\zeta$ the number of 3-arc segments of arcs in $\R$ that cut out the angle $\theta_s$, $\theta$ and $\theta_s$ in order and are not incident to any point in $\PP\cup\MM$.
        \item $\alpha$ the number of 3-arc segments of arcs in $\R$ that cut out the angle $\theta^t, \theta$ and $\theta_s$, and
        \item $\beta$ the number of arc segments of arcs in $\R$ that cross $l^\circ_s$ and cut out the angle at $l^\circ_s(1)$ clockwise from $l^\circ_s$.
    \end{itemize}
    Since any arcs in $\R$ have no intersections with each other, we have $\beta\geq\alpha+1$ in this case. So
        \[\Int(\R,l^\circ)=2\zeta+2\alpha\]
        and \[\Int(\R,l^\circ_s)\geq2\zeta+\alpha+\beta\geq 2\zeta+2\alpha+1>\Int(\R,l^\circ),\]
        which contradicts with $l^\circ$ is maximal. So $|\jiaodian(l_s^\circ)|=1$ in this case, see the second picture in Figure~\ref{fig:self-angle}. Then after flipping $l$, the value $m$ decreases. If we are no longer in the case ($\ast$), then we are already done. If we are still in the case ($\ast$), then applying the induction hypothesis on $m$, we get the assertion.

    \begin{figure}[htpb]\centering
    \begin{tikzpicture}[scale=2.1]
        \draw[blue,thick,->-=.4,>=stealth](0,-1)node{$\bullet$}to[out=165,in=-123](-1.4,.4)to[out=57,in=180](-.97,.65)to[out=0,in=105](0,-1);
        \draw[blue,thick](0,-1)to[out=15,in=-57](1.4,.4)to[out=123,in=0](.97,.65)to[out=180,in=75](0,-1);
        \draw[blue,thick](0,-1)to(-1,-1.5) (0,-1)to(1,-1.5);
        \draw[thick,fill=gray!30](-.5,.03)circle(.12) (-.9,.11)node{$\bullet$};
        \draw[thick,bend left=40](-1,.3)to(-.6,.3) (-.65,.32)to(-.95,.32);
        \draw[orange,thick](-.4,-1.2)to[out=90,in=-135](-.3,-.84)to[out=45,in=170](0.05,-.74) (-.35,-.75)node{$\alpha$};
        \draw[orange,thick,bend right=10](.2,-1.1)to(.2,-.95);
        \draw[orange](-.23,-1.03)node{$\theta^t$}(-.1,-.85)node{$\theta$}(0,-.6)node{$\theta_s$}(.3,-1.03)node{$\beta$};
        
        \draw[red,thick,->-=.15,>=stealth](.4,-1.2)to[out=90,in=0](.1,-.35)to[out=180,in=-90](-1.3,.17)to[out=90,in=180](-.65,.7)to[out=0,in=140](.42,.37) (-.2,.8)node{$\gamma$};
        \draw[red,thick](.36,.28)to[out=140,in=0](-.6,.6)to[out=180,in=90](-1.2,.17)to[out=-90,in=175](.17,-.25) (0,.31)node{$\zeta$};
        
        \draw(.4,-.86)node[black]{$\bullet$}(.55,-.9)node{$\jiaodian_i$} (.42,.37)node[black]{$\bullet$}(.67,.37)node{$\jiaodian_{m+1}$} (.14,-.36)node[black]{$\bullet$}(.38,-.26)node{$\jiaodian_{m-2}$} (-.15,-.36)node[black]{$\bullet$}(-.34,-.45)node{$\jiaodian_{m-1}$}(-.92,.65)node[black]{$\bullet$}(-.9,.8)node{$\jiaodian_m$}  (-1,-.8)node[blue]{$l^\circ=l^\circ_{m-1}$}(.9,.8)node[blue]{$l^\circ_{m+1}=l^\circ_{m-2}=l^\circ_i$};
    \end{tikzpicture}\quad
    \begin{tikzpicture}[scale=2.1]
        \draw[blue,thick,->-=.4,>=stealth](0,-1)node[black]{$\bullet$}to[out=160,in=-90](-1.2,.3)to[out=90,in=180](-.8,.7)to[out=0,in=90](0,-1);
        \draw[blue,thick,bend right=10](0,-1)to(1,.7)node[black]{$\bullet$};
        \draw[blue,thick](0,-1)to(1,-1);
        \draw[red,thick,->-=.6,>=stealth](.6,-1.2)to[out=90,in=0](0,-.4)to[out=180,in=-90](-1.4,.3)to[out=90,in=180](-.33,.97)to[out=0,in=140](.87,.4)node[black]{$\bullet$};
        \draw[thick,fill=gray!30](-.5,.1)circle(.15) (-.8,.13)node{$\bullet$};
        \draw[thick,bend left=40](-.85,.45)to(-.45,.45) (-.5,.47)to(-.8,.47);
        \draw[cyan,thick](1,.7)to[out=-130,in=0](-.5,-.2)to[out=180,in=-90](-1,.4)to[out=90,in=165](1,.7) (.85,1)node{$f^-_{\U^\circ}(l^\circ)$};

        \draw(.35,-.57)node[black]{$\bullet$}(.6,-.57)node{$\jiaodian_{m-2}$} (-.04,-.42)node[black]{$\bullet$}(.17,-.27)node{$\jiaodian_{m-1}$}(-1.08,-.23)node[black]{$\bullet$}(-1.15,-.3)node{$\jiaodian_m$} (1.05,.2)node{$\jiaodian_{m+1}$} (-.65,-.9)node[blue]{$l^\circ=l^\circ_{m-1}$}(1.1,-.3)node[blue]{$l^\circ_{m+1}=l^\circ_{m-2}$}(.4,-1.1)node[red]{$\gamma$};
    \end{tikzpicture}
    \caption{The case $l^\circ=l_{m-1}^\circ$}
    \label{fig:self-angle2}
    \end{figure}
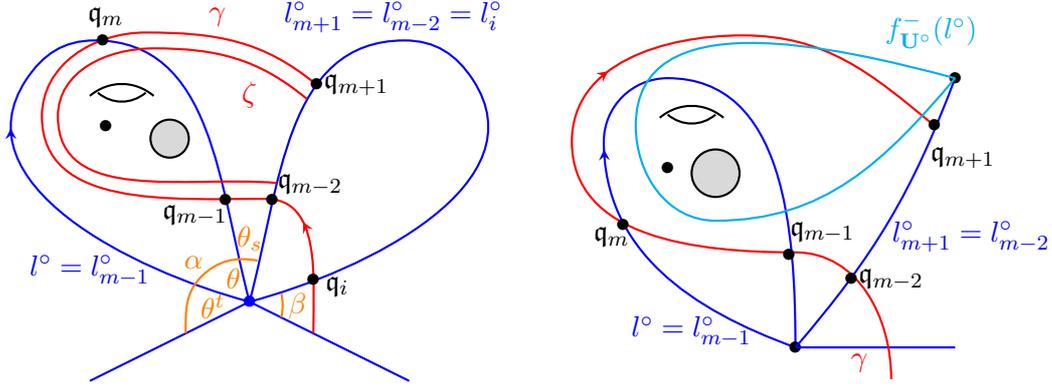

    \item Suppose $l^\circ=l^\circ_{m-1}$. Then $l^\circ$ cuts out an angle $\theta$ by its two ends and $|\jiaodian(l^\circ)|=2$. So $l_t^\circ=l_s^\circ=l_{m+1}^\circ=l_{m-2}^\circ$. Since flipping $l$ is neutral, by Lemma~\ref{lem:trans}, $l_{m+1}^\circ=l_s^\circ$ is maximal. But $l^\circ_0$ (if exists) is not maximal, so we have $m-2\geq 1$. Then $|\jiaodian(l_t^\circ)|\geq 1$. The same argument in the above case $l^\circ=l^\circ_{m+1}$ can be used in the current case (see Figure~\ref{fig:self-angle2}) to get the required assertion.

    \item Suppose that $l^\circ$ is neither $l_{m+1}^\circ$ nor $l_{m-1}^\circ$. Then $l^\circ_t=l^\circ_{m+1}$, $l^\circ_s=l^\circ_{m-1}$ and $|\jiaodian(l_t^\circ)|\leq|\jiaodian(l^\circ)|$. We shall use the notations in case (a) in the proof of Lemma~\ref{lem:trans}, i.e., we denote 
    \begin{itemize}
        \item $\alpha_1$ (resp. $\alpha_2$) the number of 2-arc segments of arcs in $\R$ that cut out the angles $\theta_t$ and $\theta^t$ (resp. $\theta_s$ and $\theta^s$) and whose arc segment cutting out $\theta_t$ (resp. $\theta_s$) do not have an endpoint in $\MM\cup\PP$. 
        \item $\beta_1$ (resp. $\beta_2$) the number of arc segments of arcs in $\R$ that cross $l^\circ_t$ (resp. $l^\circ_s$) and cut out the angle at $l^\circ_t(0)$ (resp. $l^\circ_s(1)$) clockwise from $l^\circ_t$ (resp. $l^\circ_s$).
	\item $\zeta$ the number of 2-arc segments of arcs in $\R$ that cut out the angles $\theta_s$ and $\theta_t$ and are not incident to any point in $\PP\cup\MM$.
    \end{itemize}
    There are the following cases.
    \begin{enumerate}
    \begin{figure}[htpb]\centering
    \begin{tikzpicture}[scale=2]
        \draw[blue,thick,->-=.4,>=stealth](0,-1)node[black]{$\bullet$}to[out=165,in=-123](-1.4,.4)to[out=57,in=180](-.97,.65)to[out=0,in=105](0,-1);
        \draw[blue,thick,->-=.4,>=stealth](0,-1)to[out=15,in=-57](1.4,.4)to[out=123,in=0](.97,.65)to[out=180,in=75](0,-1);
        
        \draw[blue,thick](0,-1)to(-1,-1.5) (0,-1)to(1,-1.5);
		\draw[blue,thick](0,-1)to(-.8,0) (0,-1)to(-1.1,0);
        
        \draw[orange](-.22,-.55)node{$\theta^t$}(0,-.55)node{$\theta_t$}(-.4,-.75)node{$\theta_s$}(-.2,-1.03)node{$\theta^s$};
        \draw[pink,thick](-.4,-1.2)to[out=90,in=-150](-.21,-.76)(-.52,-1.05)node{$\alpha_2$};
        \draw[pink,thick,dotted](-.21,-.76)to[out=30,in=180](0,-.7)to[out=0,in=110](.3,-.9);
        \draw[orange,thick,bend left=20](.3,-.9)to(.3,-1.16) (.44,-1)node{$\beta_1$};
        \draw[pink,thick,bend left=40](-.38,-.51)to(.12,-.45) (-.35,-.39)node{$\alpha_1$};

        \draw[red,thick,->-=.15,>=stealth](.6,-1.4)to[out=90,in=0](0.1,-.28)to[out=180,in=-90](-1.4,.25)to[out=90,in=180](-.6,.8)to[out=0,in=160](.59,.56)node[black]{$\bullet$};
        \draw[red,thick](-1,-.1)to[out=160,in=-90](-1.25,.3)to[out=90,in=180](-.6,.65)to[out=0,in=160](.48,.43) (.7,-1.2)node{$\gamma$};

        \draw (.55,-.8)node[black]{$\bullet$}(.67,-.86)node{$\jiaodian_i$} (.16,-.3)node[black]{$\bullet$}(.4,-.2)node{$\jiaodian_{k-1}$} (-.16,-.3)node[black]{$\bullet$}(-.35,-.2)node{$\jiaodian_k$}(-.8,-.26)node[black]{$\bullet$}(-.95,.2)node[blue]{$l^\circ_{m-1}$} (-1.3,.53)node[black]{$\bullet$}(-1.15,.8)node{$\jiaodian_m$} (.75,.42)node{$\jiaodian_{m+1}$} (-1,-.9)node[blue]{$l^\circ=l^\circ_k$}(.85,.8)node[blue]{$l^\circ_{m+1}=l^\circ_{k-1}=l^\circ_i$} (0,.4)node[red]{$\zeta$};
    \end{tikzpicture}\quad
    \begin{tikzpicture}[scale=2]
      \draw[blue,thick,->-=.4,>=stealth](0,-1)node[black]{$\bullet$}to[out=135,in=-60](-1,0)to[out=120,in=180](0,1)to[out=0,in=60](1,0)to[out=-120,in=45](0,-1);
      \draw[blue,thick](0,-1)to(-1.2,-1.5) (0,-1)to(1.2,-1.5);
      \draw[blue,thick,bend left=10](0,-1)to(-.5,.2) (.5,.2)to(0,-1);
      
      \draw[blue,thick](0,-1)to[out=0,in=-165](1.4,-1)to[out=15,in=-45](1.4,.4);
      \draw[blue,thick,dashed](1.4,.4)to[out=135,in=75](0,.4);
      \draw[blue,thick,-<-=.25,>=stealth](0,.4)to[out=-105,in=90](0,-1);

      \draw[orange](-.4,-.5)node{$\theta_s$}(-.2,-1)node{$\theta^s$}(.36,-.55)node{$\theta^t$}(.24,-.91)node{$\theta_t$};
      \draw[pink,thick](-.4,-1.18)to[out=90,in=-150](-.19,-.73)(-.52,-1.05)node{$\alpha_2$};
      \draw[pink,thick,dotted](-.19,-.73)to[out=30,in=180](0,-.7);
      \draw[pink,thick,bend left=30](.22,-.65)to(.43,-1.02) (.53,-.85)node{$\alpha_1$};
      \draw[blue,thick,bend right=5](0,-1)to(.3,.2);
      \draw[orange,thick,bend left=10](-.02,-.58)to(.12,-.58) (.07,-.37)node{$\beta_1$};

      \draw[red,thick,->-=.5,>=stealth](.7,-1.3)to[out=70,in=0](0,-.2)to[out=180,in=-60](-.8,0)to[out=120,in=180](0,.8)to[out=0,in=140](1.65,0)node[black]{$\bullet$} (-.95,.2)node{$\gamma$};
      \draw[red,thick](-.42,-.1)to[out=180,in=-60](-.6,0)to[out=120,in=180](0,.6)to[out=0,in=120](1.73,-.5) (1.4,-.4)node{$\zeta$};
      
      \draw(.73,-1.05)node[black]{$\bullet$}(.95,-.95)node{$\jiaodian_{k-1}$} 
      (.56,-.5)node[black]{$\bullet$}(.75,-.55)node{$\jiaodian_k$} 
      (.35,-.32)node[black]{$\bullet$}(-.06,-.2)node[black]{$\bullet$} 
      (0.05,-.1)node{$\jiaodian_i$}(-.4,-.22)node[black]{$\bullet$} 
      (.95,.52)node[black]{$\bullet$}(.78,.51)node{$\jiaodian_m$}(1.85,.1)node{$\jiaodian_{m+1}$};

      \draw[blue](-1,-.5)node{$l^\circ=l^\circ_k$}(-.3,0)node{$l^\circ_s$}(1.3,-1.2)node{$l^\circ_{k-1}=l^\circ_i$};
    \end{tikzpicture}
    \caption{The case $|\jiaodian(l)|=|\jiaodian(l_t)|=2$}
    \label{fig:jiaodian=2}
    \end{figure}
        \item $|\jiaodian(l^\circ)|=2$, i.e., there is $1\leq k\leq m-2$ such that $l^\circ=l_k^\circ$. If $k=1$, then $l^\circ_0$ exists and is the same as $l^\circ_{m+1}$. Note that $l^\circ_0$ is not maximal but $l^\circ_{m+1}$ is maximal, a contradiction. So we have $k\neq 1$ and $l^\circ_{m+1}=l^\circ_{k-1}$. Then $|\jiaodian(l^\circ_t)|\geq 1$. If $|\jiaodian(l^\circ_t)|=2$, then there is $1\leq i< m$, $i\neq k-1,k$, such that $l^\circ_{k-1}=l^\circ_i$, see the pictures of Figure~\ref{fig:jiaodian=2}, where the left is for $k<i<m$ and the right is for $1\leq i<k-1$. Since arcs in $\R$ do not cross each other in the interior, we have $\beta_1\geq\alpha_2+1$. Hence 
        \[\Int(\R,l^\circ)=\alpha_1+\alpha_2+\zeta,\]
        and
        \[\Int(\R,l^\circ_t)\geq\alpha_1+\beta_1+\zeta\geq\alpha_1+\alpha_2+\zeta+1>\Int(\R,l^\circ),\]
        contradicts to $l^\circ$ is maximal. So in this case, we have $|\jiaodian(l^\circ_t)|=1$, see Figure~\ref{fig:l1t2}, where the left picture is for $|\jiaodian(l_{m-1}^\circ)|=1$ and the right one is for $|\jiaodian(l_{m-1}^\circ)|=2$. Then after flipping $l$, the value $m$ decreases 2 for the left picture and decreases 1 for the right picture. If we are no longer in the case $(\ast)$, then we are already done. If we are still in the case $(\ast)$, then applying the induction hypothesis on $m$, we get the assertion.
        \begin{figure}[htpb]\centering
        \begin{tikzpicture}[scale=2.1]
            \draw[blue,thick,->-=.5,>=stealth](0,-1)node[black]{$\bullet$}to[out=170,in=-90](-1.2,.13)to[out=90,in=180](-.4,.8)to[out=0,in=90](.38,.22)to[out=-90,in=60](0,-1);
            \draw[blue,thick,bend right=10](0,-1)to(1,.7)node[black]{$\bullet$};
            \draw[blue,thick](0,-1)to(1,-1);

            \draw[red,thick,->-=.5,>=stealth](.6,-1.2)to[out=90,in=0](0,-.4)to[out=180,in=-90](-1.4,.3)to[out=90,in=180](-.33,.97)to[out=0,in=140](.87,.4)node[black]{$\bullet$} (.7,-.85)node{$\gamma$};
            \draw[blue,thick,bend left=10](0,-1)to(-.7,.1)node[black]{$\bullet$};
            \draw[cyan,thick,bend left=20](-.7,.1)to(1,.7)(.85,.9)node{$f^-_{\U^\circ}(l^\circ)$};

            \draw(.35,-.57)node[black]{$\bullet$}(.6,-.57)node{$\jiaodian_{k-1}$} (.25,-.48)node[black]{$\bullet$}(.2,-.3)node{$\jiaodian_k$}(-.47,-.4)node{$\bullet$}(-.3,-.26)node{$\jiaodian_{m-1}$} (-1.15,-.18)node[black]{$\bullet$}(-1.22,-.3)node{$\jiaodian_m$} (1.05,.2)node{$\jiaodian_{m+1}$} (-.5,.55)node[blue]{$l^\circ=l^\circ_k$} (1.1,-.3)node[blue]{$l^\circ_{m+1}=l^\circ_{k-1}$};
        \end{tikzpicture}\quad
        \begin{tikzpicture}[scale=2.1]
            \draw[blue,thick,->-=.5,>=stealth](0,-1)node[black]{$\bullet$}to[out=170,in=-90](-1.2,.13)to[out=90,in=180](-.4,.8)to[out=0,in=90](.38,.22)to[out=-90,in=60](0,-1);
            \draw[blue,thick,bend right=10](0,-1)to(1,.7)node[black]{$\bullet$};
            \draw[blue,thick](0,-1)to(1,-1);

            \draw[red,thick,->-=.6,>=stealth](.6,-1.2)to[out=90,in=0](0,-.4)to[out=180,in=-90](-1.4,.3)to[out=90,in=180](-.33,.97)to[out=0,in=140](.87,.4)node[black]{$\bullet$};
            \draw[blue,thick](0,-1)to[out=150,in=-140](-.9,.3)to[out=40,in=180](-.7,.39);
            \draw[blue,thick,dashed](-.7,.39)to[out=0,in=100](0,-1)(-.7,.39)to[out=0,in=120](.8,-.3)to[out=-60,in=0](0,-1);
            \draw[cyan,thick,dashed](1,.7)to[out=130,in=0](.4,.97)to[out=180,in=95](0,-1) (1.03,1)node{$f^-_{\U^\circ}(l^\circ)$} (1,.7)to[out=-30,in=90](1.4,.1)to[out=-90,in=0](0,-1);

            \draw(.35,-.57)node[black]{$\bullet$}(.6,-.55)node{$\jiaodian_{k-1}$} (.25,-.48)node[black]{$\bullet$}(.2,-.3)node{$\jiaodian_k$}(-.92,-.28)node{$\bullet$}(-.75,-.16)node{$\jiaodian_{m-1}$} (-1.15,-.18)node[black]{$\bullet$}(-1.22,-.3)node{$\jiaodian_m$} (1.07,.25)node{$\jiaodian_{m+1}$} (-.5,.6)node[blue]{$l^\circ=l^\circ_k$} (1.25,0)node[blue]{$l^\circ_{m+1}=l^\circ_{k-1}$}(.45,-1.1)node[red]{$\gamma$};
        \end{tikzpicture}
        \caption{The case $|\jiaodian(l)|=2$, $|\jiaodian(l_t)|=1$}
        \label{fig:l1t2}
        \end{figure}
        \item $|\jiaodian(l^\circ)|=1$. Then $|\jiaodian(l_{m+1}^\circ)|\leq 1$. If $|\jiaodian(l_{m+1}^\circ)|=1$, then there is $1\leq i\leq m-1$ such that $l_{m+1}^\circ=l_i^\circ$, see Figure~\ref{fig:jiaodian neq}. Since arcs in $\R$ do not cross each other in the interior, we have $\beta_1\geq\alpha_2+1$. Hence 
        \[\Int(\R,l^\circ)=\alpha_1+\alpha_2+\zeta,\]
        and
        \[\Int(\R,l^\circ_t)\geq\alpha_1+\beta_1+\zeta\geq\alpha_1+\alpha_2+\zeta+1>\Int(\R,l^\circ),\]
        contradicts to $l^\circ$ is maximal. So in this case, we have $|\jiaodian(l^\circ_{m+1})|=0$. There are the following two subcases
        \begin{enumerate}
            \item $|\jiaodian(l^\circ_{m-1})|=1$, see the first picture in Figure~\ref{fig:final}. Then after flipping $l$, the value $m$ decreases $1$. If we are no longer in the case $(\ast)$, then we are already done. If we are still in the case $(\ast)$, then applying the induction hypothesis on $m$, we get the assertion.
            \item $|\jiaodian(l^\circ_{m-1})|=2$, i.e., there is $1\leq i< m-1$ such that $l_{m-1}^\circ=l_i^\circ$, see the second picture in Figure~\ref{fig:final}. Then after flipping $l$, the value $m$ does not change. If we are no longer in the case $(\ast)$, then we are already done. If we are still in the case $(\ast)$, then we go back to cases (2) and (3.a). So we are also done.
        \end{enumerate}
    \end{enumerate}
    
    \end{enumerate}

    \begin{figure}[htpb]\centering
    \begin{tikzpicture}[scale=2.2]
		\draw[blue,thick,-<-=.5,>=stealth](0,1)node[black]{$\bullet$}to(0,-1)node[black]{$\bullet$};
        \draw[blue,thick](0,-1)to(-1.2,-.2) (0,-1)to(1.2,-.2) (0,-1)to(-1.2,-1.5) (0,-1)to(1.2,-1.5) (0,-1)to(-.5,-.3) (0,1)to(.6,.8);
        \draw[blue,thick,-<-=.5,>=stealth](0,1)to[out=-180,in=90](-1.5,0)to[out=-90,in=180](0,-1);
        \draw[orange](-.1,.9)node{$\theta_t$}(.1,-.75)node{$\theta_s$}(.15,.83)node{$\theta^t$}(-.1,-.6)node{$\theta^s$};
        \draw[pink,thick,bend right=60](-.4,.95)to(.4,.88) (.4,.72)node{$\alpha_1$};
        \draw[pink,thick,bend right=50](.33,-.8)to(-.33,-.5) (-.2,-.35)node{$\alpha_2$};
        \draw[pink,thick,dotted](.33,.-.8)to[out=-80,in=0](0,-1.2)to[out=180,in=-70](-.4,-.95);
        \draw[orange,thick,bend left=10](-.3,-.97)to(-.25,-.83) (-.4,-.84)node{$\beta_1$};
        
        \draw[red,thick](-1,.78)to[out=-90,in=160](0,-.2)to[out=-20,in=110](.5,-.67);
        \draw[red,thick,-<-=.8,>=stealth](-.7,.89)node[black]{$\bullet$}to[out=-90,in=160](0,.2)node[black]{$\bullet$}to[out=-20,in=110](.8,-.45)node[black]{$\bullet$}to[out=-70,in=0](0,-1.4)to[out=180,in=-110](-.8,-.45)node[black]{$\bullet$};

        \draw(-.81,-.85)node[black]{$\bullet$}(-.75,1.05)node{$\jiaodian_{m+1}$}(.15,.25)node{$\jiaodian_m$}(1.1,-.55)node{$\jiaodian_{m-1}$}(-.9,-.97)node{$\jiaodian_i$}(-.7,-.33)node{$\jiaodian_{i-1}$}(-.5,.2)node[red]{$\zeta$}(-.13,.5)node[blue]{$l^\circ$}(.9,-.2)node[blue]{$l^\circ_{m-1}$}(-1.3,.8)node[blue]{$l^\circ_{m+1}$}(-1.2,-.8)node[blue]{$l^\circ_i$} (.8,-1.1)node[red]{$\gamma$};
	\end{tikzpicture}    
	\caption{The case $|\jiaodian(l^\circ)|=1$, $|\jiaodian(l_t^\circ)|=1$}
	\label{fig:jiaodian neq}
\end{figure}
\begin{figure}[htpb]\centering
	\begin{tikzpicture}[scale=2.2]
        \draw[blue,thick](.5,-1)node[black]{$\bullet$}to(-.5,.7)node[black]{$\bullet$};
        \draw[blue,thick,bend right=5](-.5,.7)to(-.6,-1)node[black]{$\bullet$} (.5,-1)to(.5,.7)node[black]{$\bullet$} (.5,-1)to(1.2,0);
        \draw[cyan,thick,bend left=10](-.6,-1)to(.5,.7) (-.18,-.7)node{$f^-_{\U^\circ}(l^\circ)$};
        \draw[red,thick,bend right=10,->-=.4,>=stealth](1,-.6)to(-1,.4) (-.85,.5)node{$\gamma$};
           
		\draw(.55,-.3)node[black]{$\bullet$}(.8,-.15)node{$\jiaodian_{m-1}$} (-.16,.09)node[black]{$\bullet$}(-.05,.23)node{$\jiaodian_m$} (-.56,.24)node[black]{$\bullet$}(-.8,.2)node{$\jiaodian_{m+1}$} (.75,.2)node[blue]{$l^\circ_{m-1}$} (-.2,.5)node[blue]{$l^\circ$} (-.85,-.2)node[blue]{$l^\circ_{m+1}$};
	\end{tikzpicture}\qquad
    \begin{tikzpicture}[xscale=-2.2,yscale=2.2]
        \draw[blue,thick,bend left=5](0,-1)to(.2,.7)node[black]{$\bullet$}; 
        \draw[blue,thick](0,-1)to(-1,-1)(.2,.7)to(.7,-.3)node[black]{$\bullet$};
        \draw[red,thick,bend left=30,->-=.4,>=stealth](-.6,-1.2)to(.7,.5) (-.7,-1.1)node{$\gamma$};
        \draw[blue,thick](0,-1)to[out=150,in=-140](-.9,.3)to[out=30,in=100](0,-1);
        \draw[blue,thick,dashed](0,-1)node[black]{$\bullet$}to(-.7,-.2);
        \draw(-.68,.47)node[blue]{$l^\circ_{m-1}$};
        \draw[cyan,thick](.7,-.3)to[out=130,in=0](-.7,.7)to[out=180,in=90](-1.2,.1)to[out=-90,in=170](0,-1) (-.3,.78)node{$f^-_{\U^\circ}(l^\circ)$};
        \draw(-.48,-.7)node[black]{$\bullet$}(-.2,-.17)node[black]{$\bullet$}(-.47,-.1)node{$\jiaodian_{m-1}$} (.13,.35)node{$\bullet$}(-0.02,.3)node{$\jiaodian_m$} (.39,.33)node[black]{$\bullet$}(.65,.3)node{$\jiaodian_{m+1}$} (.4,-.5)node[blue]{$l^\circ$};
    \end{tikzpicture}
\caption{The case $|\jiaodian(l^\circ)|=1$, $|\jiaodian(l_{m+1}^\circ)|=0$.}
\label{fig:final}
\end{figure}

\end{document}